\documentclass[reqno,11pt]{amsart}
\usepackage{amsmath,amssymb,mathrsfs,amsthm,amsfonts}
\usepackage[inline]{enumitem}
\usepackage[usenames,dvipsnames]{xcolor}
\usepackage{hyperref}
\usepackage[percent]{overpic}
\usepackage{comment}
\usepackage{stmaryrd}
\usepackage{algorithm}
\usepackage{algorithmic}
\usepackage{pgfplots}
\usepackage{subcaption}
\usepackage{tikz}
\usetikzlibrary{calc}
\usetikzlibrary{arrows.meta,backgrounds}
\usepackage{multirow,array,longtable,booktabs}
\usetikzlibrary{arrows}

\hypersetup{
	colorlinks=true, linkcolor=blue,
	citecolor=ForestGreen
}
\usepackage[paper=letterpaper,margin=1in]{geometry}
\DeclareMathAlphabet{\mathpzc}{OT1}{pzc}{m}{it}

\newtheorem{theorem}{Theorem}[section]
\newtheorem{lemma}[theorem]{Lemma}
\newtheorem{proposition}[theorem]{Proposition}
\newtheorem{corollary}[theorem]{Corollary}
\theoremstyle{definition}
\newtheorem{definition}[theorem]{Definition}

\newtheorem{assumption}[theorem]{Assumption}
\newtheorem{remark}[theorem]{Remark}
\numberwithin{equation}{section}
\allowdisplaybreaks
\usepackage{acronym}

\acrodef{KPZ}{Kardar--Parisi--Zhang}
\acrodef{SHE}{Stochastic Heat Equation}
\acrodef{LDP}{Large Deviation Principle}

\newcounter{dummy}
\renewcommand{\Pr}{\mathbb{P}}
\newcommand{\Ex}{\mathbb{E}}
\newcommand{\ind}{\mathbf{1}}	

\newcommand{\fa}{\mathsf{f}_{\theta}}
\newcommand{\ffa}{\mathsf{f}}
\newcommand{\ga}{\mathsf{g}_{\alpha}}

\newcommand{\bcd}{$[\m{BCD}]$}

\newcommand{\se}[2]{S_{#1}(#2)}

\newcommand{\iise}{S_2(1)}
\newcommand{\iks}{S_1(k)}
\newcommand{\SSS}{Q}
\newcommand{\iiks}{S_2(k)}

\makeatletter
\newcommand\myitem[1][]{\item[#1]\refstepcounter{dummy}\def\@currentlabel{#1}}

\renewcommand{\ll}{\llbracket}
\newcommand{\rr}{\rrbracket}

\newcommand{\blue}{\textcolor{blue}{\textbf{blue}}}
\newcommand{\black}{\textcolor{black}{\textbf{black}}}

\newcommand{\purple}{\textcolor{gray}{\textbf{gray}}}
\newcommand{\yellow}{\textcolor{yellow!70!black}{\textbf{yellow}}}

\definecolor{rvwvcq}{rgb}{0.08235294117647059,0.396078431372549,0.7529411764705882}
\definecolor{dtsfsf}{rgb}{0.8274509803921568,0.1843137254901961,0.1843137254901961}

\newcommand{\hslg}{\mathpzc{HSLG}}

\newcommand{\Con}{C} 

\newcommand{\R}{\mathbb{R}} 
\newcommand{\Z}{\mathbb{Z}} 

\newcommand{\wsc}{W_{n}}

\newcommand{\e}{\varepsilon}

\renewcommand{\L}{\mathcal{L}}
\newcommand{\m}{\mathsf}

\renewcommand{\hat}{\widehat}
\newcommand{\til}{\widetilde}
\renewcommand{\bar}{\overline}

\renewcommand{\ni}{\mathsf{NI}}

\usepackage{graphicx}

\title[Convergence to stationary measures for the HSLG polymer]{Convergence to stationary measures for the half-space log-gamma polymer}

\author[S.\ Das]{Sayan Das}
\address{S.\ Das,
	Department of Mathematics, University of Chicago,
	\newline\hphantom{\quad \ \ S. Das}
	5734 S.~University Avenue, Chicago, IL 60637, USA
}
\email{sayan.das@columbia.edu}
\author[C.\ Serio]{Christian Serio}
\address{C.\ Serio,
	Department of Mathematics, Stanford University,
	\newline\hphantom{\quad \ \ C. \ Serio}
	450 Jane Stanford Way,
	Stanford, CA 94305, USA
}
\email{cdserio@stanford.edu}

\begin{document}
	\begin{abstract}
We consider the point-to-point half-space log-gamma polymer model in the unbound phase. We prove that the free energy increment process on the anti-diagonal path converges to the top marginal of a two-layer Markov chain with an explicit description, which can be interpreted as two random walks conditioned softly never to intersect. This limiting law is a stationary measure for the polymer on the anti-diagonal path. 

The starting point of our analysis is an embedding of the free energy into the half-space log-gamma line ensemble recently constructed in \cite{half1}. Given the Gibbsian line ensemble structure, the main  contribution of our work lies in developing a route to access and prove convergence to stationary measures via line ensemble techniques. Our argument relies on a description of the limiting behavior of two softly non-intersecting random walk bridges around their starting point, a result established in this paper that may be of independent interest.
	\end{abstract}
	
	
	\maketitle
	{
			\hypersetup{linkcolor=black}
			\setcounter{tocdepth}{1}
			\tableofcontents
		}

\section{Introduction}

\subsection{The model and main results}
Fix $\theta>0$, $\alpha>-\theta$, and consider a family of independent variables $(\mathsf W_{i,j})_{(i,j)\in \mathcal{I}}$ with $\mathcal{I}:=\{(i,j) \in \Z^2 :  j\le i\}$ such that
\begin{align}\label{eq:wt}
\mathsf W_{i,j}\sim \operatorname{Gamma}^{-1}(\alpha+\theta)  \textrm{ for } i=j \qquad \textrm{and } \quad \mathsf W_{i,j}\sim \operatorname{Gamma}^{-1}(2\theta) \textrm{ for } j<i.
	\end{align}
Here $X\sim \operatorname{Gamma}^{-1}(\beta)$ means $X$ is a random variable with density $\Gamma^{-1}(\beta)x^{-\beta-1}e^{-1/x}\ind_{x>0}$. A directed lattice path $\pi=\big((x_i,y_i)\big)_{i=1}^k$ confined to the half-space index set $\mathcal{I}$ is an up-right path with all $(x_i,y_i)\in \mathcal{I}$ which only makes unit steps in the coordinate directions, that is, $(x_{i+1},y_{i+1})=(x_i,y_i)+(0,1)$ or $(x_{i+1},y_{i+1})=(x_i,y_i)+(1,0)$; see Figure \ref{fig000}.  Given $(m,n)\in\mathcal{I}$, we let $\Pi_{m,n}$ denote the set of all directed paths from $(1,1)$ to $(m,n)$ confined to $\mathcal{I}$.  Given the random variables from \eqref{eq:wt}, we define the weight of a path $\pi$ and the point-to-point partition function of the \textit{half-space log-gamma} ($\hslg$) \textit{polymer} by
	\begin{align*}
		w(\pi):=\prod_{(i,j)\in \pi} \mathsf W_{i,j},  \qquad Z_{(\alpha,\theta)}(m,n):=\sum_{\pi\in \Pi_{m,n}} w(\pi).
	\end{align*}
	
	\begin{figure}[h!]
		\centering
		\begin{tikzpicture}[line cap=round,line join=round,>=triangle 45,x=0.8cm,y=0.8cm]
			\foreach \x in {0,1,2,...,7}
			{        
				\coordinate (A\x) at ($(9,0)+(\x,0)$) {};
				\coordinate (B\x) at ($(9,0)+(0,\x)$) {};
				\draw[dotted] ($(A\x)+(0,\x)$) -- ($(A\x)$);
				\draw[dotted] ($(B\x)+(7,0)$) -- ($(B\x)+(\x,0)$);
				\coordinate (C\x) at ($(16,0)+(\x,0)$) {};
				\coordinate (D\x) at ($(16,0)+(0,\x)$) {};
				\draw[dotted] ($(C\x)+(0,7-\x)$) -- ($(C\x)$);
				\draw[dotted] ($(D\x)+(7-\x,0)$) -- ($(D\x)+(0,0)$);
			}
			\draw [line width=1.5pt,color=blue] (9,0)-- (11,0)--(11,2)--(14,2)--(14,4)--(15,4)--(15,6)--(16,6)--(16,7);
			\draw [line width=1.5pt] (9,0)-- (10,0)--(10,1)--(13,1)--(13,3)--(15,3)--(15,4)--(16,4)--(16,7);
            \draw[line width=1.5pt,color=red,dashed] (16,7)--(23,0);
			\node[] (A) at (12.5,5.5) {$\operatorname{Gamma}^{-1}(\alpha+\theta)$};
			\node[] (B) at (16.2,2.5) {$\operatorname{Gamma}^{-1}(2\theta)$};
			\draw[-stealth] (12,3) -- (A);
			\draw[-stealth] (13,4) -- (A);	
			\draw[-stealth] (15,1) -- (B);
			\draw[-stealth] (16,1) -- (B);	
			\draw[-stealth] (14,1) -- (B);
		\end{tikzpicture}

		\caption{Vertex weights for the half-space log-gamma polymer and two possible paths (one marked in blue and the other in black) in $\Pi_{8,8}$. The anti-diagonal path $\{(8+k-1,8-k+1):1\leq k\leq 8\}$ is shown in red.}
		\label{fig000}
	\end{figure}
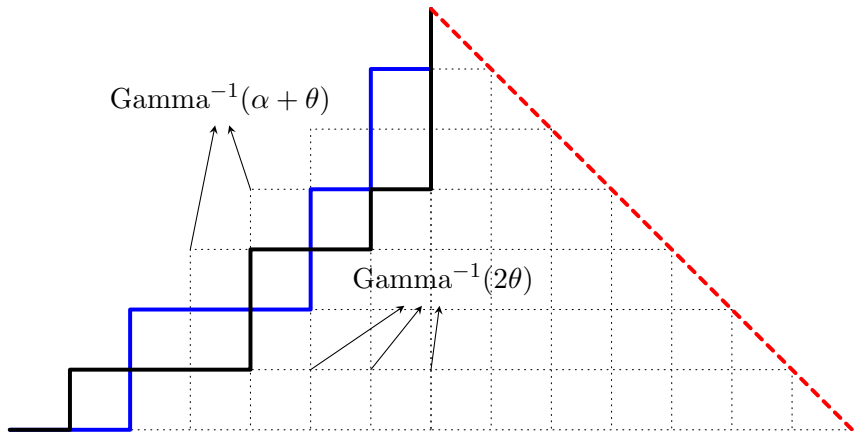

The parameter $\alpha$ controls the strength of the boundary weights, and there is a phase transition in the behavior of this model at $\alpha=0$. In this paper, we work in the \textit{unbound phase} $\alpha>0$, where the polymer delocalizes from the diagonal. We assume $\theta,\alpha>0$ are fixed parameters throughout, and we write $Z$ in place of $Z_{(\alpha,\theta)}$. We study the increments of the free energy $\log Z(m,n)$ on the anti-diagonal path $\{(N+k-1,N-k+1):1\leq k\leq N\}$ in this setting. Our main result is as follows:
	
	\begin{theorem} \label{t.main1} Fix $\theta, \alpha> 0$ and $r\in \Z_{\ge 1}$. As $N\to \infty$, we have the following multipoint convergence:
		\begin{align*}
			\Big(\log Z(N+k-1,N-k+1)-\log Z(N,N)\Big)_{k=1}^r \stackrel{d}{\longrightarrow} \big(S_1^{\uparrow}(k)\big)_{k=1}^{r},
		\end{align*}
  where the sequence of random variables $\big(S_1^{\uparrow}(k)\big)_{k\ge 1}$ is defined in Section \ref{sec:1.1}. 
	\end{theorem}

After discussing the significance of this result and its connections to previous literature below, in Section \ref{sec12} we will describe our method of proof, which involves a novel application of half-space Gibbsian line ensembles. In Section \ref{sec:ext} we will discuss how our techniques may be applied to related solvable models.

\subsubsection{Description of the limiting distribution} \label{sec:1.1}

Let $\m{h}_1, \m{h}_2$ be two measures on $\mathbb{R}$. Let $\fa$ denote the density of $\log A-\log B$, where $A,B$ are i.i.d.~$\operatorname{Gamma}(\theta)$ random variables. It has an explicit expression given by
\begin{align*}
    \fa(x)=\frac{\Gamma(2\theta)}{(\Gamma(\theta))^2}\Big(e^{-x/2}+e^{x/2}\Big)^{-2\theta}.
\end{align*}
Consider two independent random walks $(S_1(k),S_2(k))_{k=1}^n$ of length $n$  started from the initial data $S_1(1) \sim \m{h}_1,S_2(1) \sim \m{h}_2$, and with joint transition density $p\big((x_1,y_1),(x_2,y_2)\big):= \fa(x_2-x_1)\fa(y_2-y_1)$. In other words the increments of the random walks are distributed according to the density $\fa$. We denote the law of $(S_1(k),S_2(k))_{k=1}^n$ by $\Pr^{n;(\m{h}_1,\m{h}_2);\m{free}}$. If $\m{h}_i$ equals the Dirac delta measure $\delta_{x}$ at some $x\in\mathbb{R}$, we simply write $\m{h}_i=x$. If $\m{h}_i$ is absolutely continuous with respect to Lebesgue measure, we identify it with its density, which we denote by $\m{h}_i(\cdot)$.

\medskip

Consider the density {of a log-$\operatorname{Gamma}(\alpha)$ random variable}, $\ga(x):=\Gamma(\alpha)^{-1}e^{\alpha x-e^x}$, and define
\begin{align}\label{Vdef}
   V(z)  := \lim_{n\to\infty} \frac{\Ex^{n;(0,z);\m{free}}[W_n]}{\Ex^{n;(0,\ga);\m{free}}[W_n]},
\end{align}
where  
\begin{align}
    \label{wscintro}
    W_n:=\exp\bigg(-e^{\iise-\se{1}{2}}-\sum_{k=2}^{n-1} \left(e^{\iiks-\se{1}{k+1}}+e^{\iiks-\iks}\right)\bigg).
\end{align}
The weight $W_n$ can be viewed as a ``soft'' version of the indicator function for the non-intersection event $\{S_1(k) > S_2(k) \mbox{ for } 1\leq k\leq n\}$, as it associates a penalty of $e^{-e^{\delta}}$ whenever $S_2(k)-S_1(k)>\delta$ for some $k$. We shall show in Lemma \ref{V(x,y)} (see also Remark \ref{vxyw}) that the limit in \eqref{Vdef} exists and $V(z)>0$ for all $z\in\mathbb{R}$. 
Using this $V$, we consider a Markov chain $\big(S_1^{\uparrow}(k), S_2^{\uparrow}(k)\big)_{k\ge 1}$ with initial data $S_1^{\uparrow}(1) = 0$ and $S_2^{\uparrow}(1) \sim p_0^V(y):= V(y)\ga(y)$, and transition density given by
\begin{align}
\label{def:pv}
  p^V((x_1,y_1),(x_2,y_2))=\frac{V(y_2-x_2)}{V(y_1-x_1)}\, p((x_1,y_1),(x_2,y_2))\exp(-e^{y_1-x_2}-e^{y_2-x_2}).
\end{align}
  We will show in Lemma \ref{idenf} that $p_0^V(\cdot), p^V((x_1,y_1),\cdot)$ are indeed valid density functions on $\mathbb{R}$ and $\mathbb{R}^2$ respectively. The limiting law in Theorem \ref{t.main1} is the marginal distribution of $\big(S_1^{\uparrow}(k)\big)_{k\ge 1}$.

\medskip

Note that the dependence on $\alpha$ appears only in $p_0^V(y)$; the transition density only depends on $\theta$. The transition density \eqref{def:pv} is a Doob $h$-transform of a sub-Markov random walk process with $h=V$. It may be instructive to compare with the transition density for Dyson Brownian motion, i.e., Brownian motions conditioned never to intersect, cf. \cite[Section 4.2]{konig}. There, the analogue of $V$ is the Vandermonde determinant, which gives the long-time dependence of the non-intersection probability on the starting position, and the analogue of the exponential factor on the right of \eqref{def:pv} is the indicator for non-intersection. Thus, in light of the soft non-intersection interpretation of $W_n$, the Markov process $\big(S_1^{\uparrow}(k),S_2^{\uparrow}(k)\big)_{k\ge 1}$ may be viewed as two random walks  conditioned ``softly'' never to intersect. We refer to Section \ref{sec12} for an explanation of why this softly non-intersecting Markov process appears as the limiting distribution for the free energy. 

As we will explain in Section \ref{sec12}, the starting point of our proof is an integrable input coming from \cite{half1}. Other than that, our arguments are entirely probabilistic and do not yield an explicit expression for $V$. However, we remark that the work \cite{bk} describes stationary measures of the half-space KPZ equation (the intermediate disorder limit of $\hslg$ polymer) via an explicit Doob transform of a certain killed Brownian motion.

\medskip

Our main result, Theorem \ref{t.main1}, fits into a recent body of work attempting to construct and study \textit{stationary measures} for half-space models. The partition function $Z$ satisfies the recurrence relation
\begin{equation}
    \label{eq:hslgeqn}
    \begin{aligned}
     Z(m,n) & = \mathsf W_{m,n}\big(Z(m-1,n)+Z(m,n-1)\big), \quad m>n, \\
     Z(n,n) & = \mathsf{W}_{n,n}Z(n,n-1), \quad n\ge 1,
\end{aligned}
\end{equation}
with delta initial data $Z(k, 0)=\ind_{k=1}$. This may be viewed as a discrete version of the stochastic heat equation, and thus the free energy $\log Z$ is a discrete analogue of the half-space KPZ equation; see \cite{bc22}. In general, we say a process $(h(k))_{k\geq 1}$
is \textit{stationary} for the $\hslg$ polymer on the anti-diagonal path if the solution to \eqref{eq:hslgeqn} with
initial data $Z(k-1,-k+1) := e^{h(k)}$ for $k\ge 1$ has the property that the distribution of
$$\big(\log Z(N+k-1, N-k+1)-\log Z(N,N)\big)_{k\geq 1}$$
in the same for all $N\geq 0$. Assuming Theorem \ref{t.main1}, it is not hard to see that $\big(S_1^{\uparrow}(k)\big)_{k\ge 1}$ is indeed stationary for the $\hslg$ polymer on the anti-diagonal path. 

\smallskip

The study of stationary measures for the log-gamma polymer goes back to \cite{timo}, where the full space version of the model was first introduced (see also \cite{cosz14}). In \cite{timo}, the author found that the stationary measures for the full-space log-gamma polymer are given by a one-parameter family of random walks with log-gamma increments. Subsequently, it was proven in \cite{grasy15} (see also \cite{jra20}) that these are the only stationary measures which are ergodic with respect to translations in both directions of the lattice.

{Half-space polymer models, introduced in Kardar's work \cite{kar2}, have been extensively explored in physics literature due to their connection to wetting phenomena \cite{phy2,phy1,phy3}. They are of great interest because of the presence of a phase transition, as mentioned earlier. The half-space version of the log-gamma polymer model we study here was first introduced in \cite{osz14} (see also subsequent works \cite{bz19,bbc20,boz,bw,ims22, half1,dz23,victor}).}
For the half-space case, employing an approach similar to that of \cite{timo}, \cite{bkld2} constructed  a log-gamma random walk stationary measure for the $\hslg$ polymer. Recently, in \cite{dz23} it was demonstrated that in the bound phase ($\alpha<0$), the free energy increment process converges to the log-gamma random walk measure constructed in \cite{bkld2}. In the unbound phase $(\alpha > 0$), as illustrated by our main theorem, the stationary measures become notably more intricate. We remark that at $\alpha=0$, we expect the same limiting measure as in the bound phase; indeed, as $\alpha\downarrow 0$ the transition density in \eqref{def:pv} should converge to that of a single log-gamma random walk, as the density $\m{g}_\alpha$ pushes $S_2^\uparrow(1)$ to $-\infty$.

{In the work \cite{bc22}, the authors constructed a  family  of stationary measures $\{z_{v,P}(\cdot)\}_{v\in (-\theta,\min\{\alpha,0\}]}$ for the $\hslg$ polymer along each down-right path $P$ and conjectured that these constitute all extremal stationary measures for the  model. Our case of delta initial data for the partition function $Z$ corresponds to $v=0$, and we expect our measure to match with $z_{0,P_*}(\cdot)$ obtained in \cite{bc22} where $P_*$ denotes the anti-diagonal path, although we do not prove this here.} {Our techniques are notably distinct from those of \cite{bc22}, in that we exploit a certain Gibbsian line ensemble structure of the polymer, which we describe in the next section.} While our main result is stated only for the anti-diagonal path, we believe that it should be possible to extend the line ensemble structure and use our methods to access limits along any down-right path. The approach used here could also be modified by including one horizontal inhomogeneity (creating a non-trivial initial slope for the free energy) so as to obtain convergence to a one-parameter family of measures. We leave these two directions for future consideration. 



Finally, we remark that more recently in \cite{bcy}, the authors constructed a unique ergodic stationary measure for the log-gamma polymer on a strip, i.e., with two-sided boundaries, for any down-right path. It should be possible to obtain the half-space stationary measure described in Section \ref{sec:1.1} by taking an appropriate limit of their measures.


\subsection{Proof ideas} \label{sec12} In this section, we sketch the key ideas behind the proof of our main result. {The primary technical contribution of this paper is a novel method of obtaining convergence to stationary measures for models that can be embedded into Gibbsian line ensembles, of which the $\hslg$ polymer is just one example. A secondary contribution is a local convergence result for softly non-intersecting random walks, which we describe in Section \ref{sec:1.2.3}.} 

The starting point of our analysis here is the $\hslg$ Gibbsian line ensemble constructed in \cite{half1}. Namely, we may view the free energy process $(\log Z(N+k-1,N-k+1))_{k\ge 1}$ of the $\hslg$ polymer  as the  top curve of a Gibbsian line ensemble $(\L_i^N(\cdot))_{i\in \ll1,N\rr}$, consisting of log-gamma increment random walks interacting through a soft non-intersection condition and subject to a pairwise pinning at the left boundary. (See Figure \ref{figO1} and its caption.)  This remarkable fact comes from the geometric RSK correspondence \cite{cosz14, osz14, nz17,bz19} and the half-space Whittaker process \cite{bbc20}.

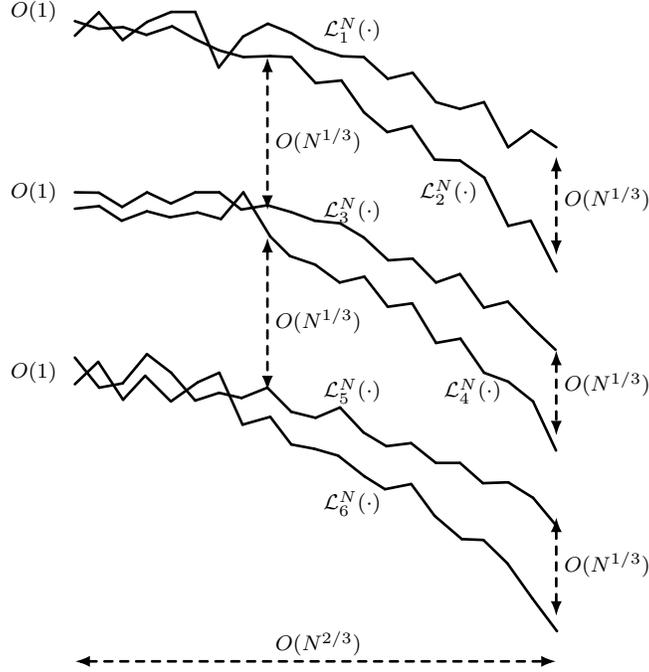
\begin{figure}[h!]
			\centering
		\begin{tikzpicture}[scale=.8,line cap=round,line join=round,>=triangle 45,x=4cm,y=1.5cm]
			\draw [line width=1pt] (0,2.7361969564619426)-- (0.09968303324610706,2.9960253171531424);
			\draw [line width=1pt] (0.09968303324610706,2.9960253171531424)-- (0.19830773702771506,2.6884965795913667);
			\draw [line width=1pt] (0.19830773702771506,2.6884965795913667)-- (0.30006302893410286,2.889974785777928);
			\draw [line width=1pt] (0.30006302893410286,2.889974785777928)-- (0.4,3);
			\draw [line width=1pt] (0.4,3)-- (0.5,3);
			\draw [line width=1pt] (0.5,3)-- (0.5973644533669638,2.3832141676196814);
			\draw [line width=1pt] (0.5973644533669638,2.3832141676196814)-- (0.698588596048139,2.7266568810878273);
			\draw [line width=1pt] (0.698588596048139,2.7266568810878273)-- (0.8017593568577984,2.869758011699555);
			\draw [line width=1pt] (0.8017593568577984,2.869758011699555)-- (0.8990902632820055,2.7648171825842884);
			\draw [line width=1pt] (0.8990902632820055,2.7648171825842884)-- (0.9990053878836391,2.600795293991784);
			\draw [line width=1pt] (0.9990053878836391,2.600795293991784)-- (1.1079547104749525,2.5089745438145457);
			\draw [line width=1pt] (1.1079547104749525,2.5089745438145457)-- (1.2,2.5);
			\draw [line width=1pt] (1.2,2.5)-- (1.300093597749738,2.259193187756184);
			\draw [line width=1pt] (1.300093597749738,2.259193187756184)-- (1.4032643585593976,2.32597371537499);
			\draw [line width=1pt] (1.4032643585593976,2.32597371537499)-- (1.5,2);
			\draw [line width=1pt] (1.5,2)-- (1.5998727895362959,1.925290549662153);
			\draw [line width=1pt] (1.5998727895362959,1.925290549662153)-- (1.7,2);
			\draw [line width=1pt] (1.7,2)-- (1.8,1.5);
			\draw [line width=1pt] (1.8,1.5)-- (1.895758745065885,1.6867886653092736);
			\draw [line width=1pt] (1.895758745065885,1.6867886653092736)-- (2,1.5);
			\draw [line width=1pt] (0,2.8983782378219005)-- (0.09903021247502392,2.812517559454864);
			\draw [line width=1pt] (0.09903021247502392,2.812517559454864)-- (0.20025435515619922,2.8315977102030945);
			\draw [line width=1pt] (0.20025435515619922,2.8315977102030945)-- (0.2975852615804062,2.745737031836058);
			\draw [line width=1pt] (0.2975852615804062,2.745737031836058)-- (0.40398128310420617,2.841199415267653);
			\draw [line width=1pt] (0.40398128310420617,2.841199415267653)-- (0.5000335469427568,2.698036654965482);
			\draw [line width=1pt] (0.5000335469427568,2.698036654965482)-- (0.5993110714954479,2.5740156751019847);
			\draw [line width=1pt] (0.5993110714954479,2.5740156751019847)-- (0.7,2.5);
			\draw [line width=1pt] (0.7,2.5)-- (0.8116756285443952,2.510788449351374);
			\draw [line width=1pt] (0.8116756285443952,2.510788449351374)-- (0.9,2.5);
			\draw [line width=1pt] (0.9,2.5)-- (1.0003144059631806,2.2114928108856082);
			\draw [line width=1pt] (1.0003144059631806,2.2114928108856082)-- (1.1073784030298084,2.2401130370079536);
			\draw [line width=1pt] (1.1073784030298084,2.2401130370079536)-- (1.200816073197047,1.8871302481656922);
			\draw [line width=1pt] (1.200816073197047,1.8871302481656922)-- (1.2981469796212541,1.6677085145610433);
			\draw [line width=1pt] (1.2981469796212541,1.6677085145610433)-- (1.3993711223024294,1.7344890421798496);
			\draw [line width=1pt] (1.3993711223024294,1.7344890421798496)-- (1.4947554105981522,1.3624261025893578);
			\draw [line width=1pt] (1.4947554105981522,1.3624261025893578)-- (1.5998727895362959,1.3528860272152425);
			\draw [line width=1pt] (1.5998727895362959,1.3528860272152425)-- (1.701096932217471,1.162084519732939);
			\draw [line width=1pt] (1.701096932217471,1.162084519732939)-- (1.8003744567701623,0.6278402987824895);
			\draw [line width=1pt] (1.8003744567701623,0.6278402987824895)-- (1.8938121269374009,0.6850807510271805);
			\draw [line width=1pt] (1.8938121269374009,0.6850807510271805)-- (2.0008761240040287,0.12221630395438526);
			\draw [line width=1pt] (0,1)-- (0.09968303324610683,0.9960253171531422);
			\draw [line width=1pt] (0.09968303324610683,0.9960253171531422)-- (0.19636111889923094,0.8377219570130232);
			\draw [line width=1pt] (0.19636111889923094,0.8377219570130232)-- (0.3,1);
			\draw [line width=1pt] (0.3,1)-- (0.3988094042615815,0.875882258509484);
			\draw [line width=1pt] (0.3988094042615815,0.875882258509484)-- (0.5,1);
			\draw [line width=1pt] (0.5,1)-- (0.6,1);
			\draw [line width=1pt] (0.6,1)-- (0.6912407305290592,0.808874500983061);
			\draw [line width=1pt] (0.6912407305290592,0.808874500983061)-- (0.7998127387293144,0.8568021077612535);
			\draw [line width=1pt] (0.7998127387293144,0.8568021077612535)-- (0.8990902632820055,0.7804815047683322);
			\draw [line width=1pt] (0.8990902632820055,0.7804815047683322)-- (0.9964211697062124,0.6850807510271805);
			\draw [line width=1pt] (0.9964211697062124,0.6850807510271805)-- (1.1015385486443559,0.656460524904835);
			\draw [line width=1pt] (1.1015385486443559,0.656460524904835)-- (1.2,0.5);
			\draw [line width=1pt] (1.2,0.5)-- (1.3020402158782223,0.2462372838178825);
			\draw [line width=1pt] (1.3020402158782223,0.2462372838178825)-- (1.4032643585593976,0.26531743456611284);
			\draw [line width=1pt] (1.4032643585593976,0.26531743456611284)-- (1.5,0);
			\draw [line width=1pt] (1.5,0)-- (1.5998727895362959,0.09359607783203974);
			\draw [line width=1pt] (1.5998727895362959,0.09359607783203974)-- (1.699150314088987,-0.27846686175845203);
			\draw [line width=1pt] (1.699150314088987,-0.27846686175845203)-- (1.798427838641678,-0.2116863341396458);
			\draw [line width=1pt] (1.798427838641678,-0.2116863341396458)-- (1.9,-0.5);
			\draw [line width=1pt] (1.9,-0.5)-- (1.9969828877470606,-0.7459305550900955);
			\draw [line width=1pt] (0,0.8186418062647929)-- (0.10097683060350805,0.8472620323871384);
			\draw [line width=1pt] (0.10097683060350805,0.8472620323871384)-- (0.1944145007707468,0.6850807510271805);
			\draw [line width=1pt] (0.1944145007707468,0.6850807510271805)-- (0.29953187970889034,0.7900215801424474);
			\draw [line width=1pt] (0.29953187970889034,0.7900215801424474)-- (0.3949161680046132,0.7232410525236411);
			\draw [line width=1pt] (0.3949161680046132,0.7232410525236411)-- (0.50890466407512,0.7773440285883928);
			\draw [line width=1pt] (0.50890466407512,0.7773440285883928)-- (0.6071213261663665,0.7042691302093282);
			\draw [line width=1pt] (0.6071213261663665,0.7042691302093282)-- (0.7,1);
			\draw [line width=1pt] (0.7,1)-- (0.8116756285443949,0.5107884493513744);
			\draw [line width=1pt] (0.8116756285443949,0.5107884493513744)-- (0.893250408896553,0.2939376606884584);
			\draw [line width=1pt] (0.893250408896553,0.2939376606884584)-- (0.9983677878346965,0.19853690694730664);
			\draw [line width=1pt] (0.9983677878346965,0.19853690694730664)-- (1.1,0);
			\draw [line width=1pt] (1.1,0)-- (1.2027626913255312,0.06497585170969422);
			\draw [line width=1pt] (1.2027626913255312,0.06497585170969422)-- (1.300093597749738,-0.2689267863843368);
			\draw [line width=1pt] (1.300093597749738,-0.2689267863843368)-- (1.4013177404309134,-0.23076648488787616);
			\draw [line width=1pt] (1.4013177404309134,-0.23076648488787616)-- (1.5005952649836045,-0.6696099520971741);
			\draw [line width=1pt] (1.5005952649836045,-0.6696099520971741)-- (1.6018194076647798,-0.6219095752265983);
			\draw [line width=1pt] (1.6018194076647798,-0.6219095752265983)-- (1.7,-1);
			\draw [line width=1pt] (1.7,-1)-- (1.8003744567701623,-1.0989133439323568);
			\draw [line width=1pt] (1.8003744567701623,-1.0989133439323568)-- (1.9035452175798215,-1.318335077537006);
			\draw [line width=1pt] (1.9035452175798215,-1.318335077537006)-- (1.9989295058755447,-1.8621193738615707);
			\draw [line width=1pt] (2.002822742132513,-3.865535202425757)-- (1.9,-3.5);
			\draw [line width=1pt] (1.9,-3.5)-- (1.796481220513194,-3.1118692478706587);
			\draw [line width=1pt] (1.796481220513194,-3.1118692478706587)-- (1.6972036959605028,-2.854287212769549);
			\draw [line width=1pt] (1.6972036959605028,-2.854287212769549)-- (1.6076592620502324,-2.8447471373954336);
			\draw [line width=1pt] (1.6076592620502324,-2.8447471373954336)-- (1.4947554105981522,-2.587165102294324);
			\draw [line width=1pt] (1.4947554105981522,-2.587165102294324)-- (1.3974245041739453,-2.2341823134520626);
			\draw [line width=1pt] (1.3974245041739453,-2.2341823134520626)-- (1.2884138889788335,-2.2914227656967534);
			\draw [line width=1pt] (1.2884138889788335,-2.2914227656967534)-- (1.2027626913255312,-2.148321635085026);
			\draw [line width=1pt] (1.2027626913255312,-2.148321635085026)-- (1.0937520761304194,-1.9193598261062619);
			\draw [line width=1pt] (1.0937520761304194,-1.9193598261062619)-- (0.9944745515777283,-1.8430392231133403);
			\draw [line width=1pt] (0.9944745515777283,-1.8430392231133403)-- (0.8990902632820055,-1.7953388462427646);
			\draw [line width=1pt] (0.8990902632820055,-1.7953388462427646)-- (0.8116756285443951,-1.4892115506486259);
			\draw [line width=1pt] (0.8116756285443951,-1.4892115506486259)-- (0.6966419779196549,-1.5759171126381155);
			\draw [line width=1pt] (0.6966419779196549,-1.5759171126381155)-- (0.6,-1);
			\draw [line width=1pt] (0.6,-1)-- (0.5113633117794375,-1.1045969746536137);
			\draw [line width=1pt] (0.5113633117794375,-1.1045969746536137)-- (0.3988094042615815,-1.318335077537006);
			\draw [line width=1pt] (0.3988094042615815,-1.318335077537006)-- (0.29203043996444555,-1.034112711146331);
			\draw [line width=1pt] (0.29203043996444555,-1.034112711146331)-- (0.20025435515619922,-1.2992549267887756);
			\draw [line width=1pt] (0.20025435515619922,-1.2992549267887756)-- (0.09708359434653978,-0.879491610327708);
			\draw [line width=1pt] (0.09708359434653978,-0.879491610327708)-- (0,-1.1275335700547024);
			\draw [line width=1pt] (0,-0.831791233457132)-- (0.10097683060350805,-1.165693871551163);
			\draw [line width=1pt] (0.10097683060350805,-1.165693871551163)-- (0.19830773702771506,-1.1179934946805872);
			\draw [line width=1pt] (0.19830773702771506,-1.1179934946805872)-- (0.29953187970889034,-0.7936309319606714);
			\draw [line width=1pt] (0.29953187970889034,-0.7936309319606714)-- (0.4,-1);
			\draw [line width=1pt] (0.4,-1)-- (0.5000335469427568,-1.3087950021628907);
			\draw [line width=1pt] (0.5000335469427568,-1.3087950021628907)-- (0.5993110714954479,-1.222934323795854);
			\draw [line width=1pt] (0.5993110714954479,-1.222934323795854)-- (0.6927487416626866,-1.2801747760405453);
			\draw [line width=1pt] (0.6927487416626866,-1.2801747760405453)-- (0.7998127387293144,-1.165693871551163);
			\draw [line width=1pt] (0.7998127387293144,-1.165693871551163)-- (0.8990902632820055,-1.432815982026388);
			\draw [line width=1pt] (0.8990902632820055,-1.432815982026388)-- (1,-1.5);
			\draw [line width=1pt] (1,-1.5)-- (1.1015385486443559,-1.385115605155812);
			\draw [line width=1pt] (1.1015385486443559,-1.385115605155812)-- (1.198869455068563,-1.661777791005152);
			\draw [line width=1pt] (1.198869455068563,-1.661777791005152)-- (1.2942537433642858,-1.814418996990995);
			\draw [line width=1pt] (1.2942537433642858,-1.814418996990995)-- (1.3967412169145335,-1.7803544108111566);
			\draw [line width=1pt] (1.3967412169145335,-1.7803544108111566)-- (1.5,-2);
			\draw [line width=1pt] (1.5,-2)-- (1.6,-2);
			\draw [line width=1pt] (1.6,-2)-- (1.6952570778320186,-2.2246422380779474);
			\draw [line width=1pt] (1.6952570778320186,-2.2246422380779474)-- (1.8003744567701623,-2.215102162703832);
			\draw [line width=1pt] (1.8003744567701623,-2.215102162703832)-- (1.9035452175798215,-2.3868235194379053);
			\draw [line width=1pt] (1.9035452175798215,-2.3868235194379053)-- (1.9989295058755447,-2.692105931409591);
			\draw [line width=1pt,dashed,{Latex[length=2mm]}-{Latex[length=2mm]}] (0,-4.2)-- (2,-4.2);
			\draw [line width=1pt,dashed,{Latex[length=2mm]}-{Latex[length=2mm]}] (0.8,0.5)-- (0.8,-1.2);
			\draw [line width=1pt,dashed,{Latex[length=2mm]}-{Latex[length=2mm]}] (0.8,2.5)-- (0.8,0.8);
			\draw [line width=1pt,dashed,{Latex[length=2mm]}-{Latex[length=2mm]}] (2,1.4)-- (2,0.3);
			\draw [line width=1pt,dashed,{Latex[length=2mm]}-{Latex[length=2mm]}] (2,-0.75)-- (2,-1.7);
			\draw [line width=1pt,dashed,{Latex[length=2mm]}-{Latex[length=2mm]}] (2,-2.6)-- (2,-3.7);
			\begin{scriptsize}
				\draw (0.8,-4) node[anchor=west] {$O(N^{2/3})$};
				\draw (0.8,-0.2) node[anchor=north west] {$O(N^{1/3})$};
				\draw (0.8,1.8) node[anchor=north west] {$O(N^{1/3})$};
				\draw (2,1.15) node[anchor=north west] {$O(N^{1/3})$};
				\draw (2,-0.9) node[anchor=north west] {$O(N^{1/3})$};
				\draw (2,-2.9) node[anchor=north west] {$O(N^{1/3})$};
				\draw (-0.3,3) node[anchor=west] {$O(1)$};
				\draw (-0.3,1) node[anchor=west] {$O(1)$};
				\draw (-0.3,-1) node[anchor=west] {$O(1)$};
				\draw (1,2.8) node[anchor=west] {$\L_1^N(\cdot)$};
				\draw (1.4,1) node[anchor=west] {$\L_2^N(\cdot)$};
				\draw (1,0.8) node[anchor=west] {$\L_3^N(\cdot)$};
				\draw (1.5,-1.2) node[anchor=west] {$\L_4^N(\cdot)$};
				\draw (1,-1.2) node[anchor=west] {$\L_5^N(\cdot)$};
				\draw (1,-2.45) node[anchor=west] {$\L_6^N(\cdot)$};
			\end{scriptsize}
		\end{tikzpicture}
		\caption{The half-space log-gamma line ensemble for large $N$ along with the typical scaling. The curves are softly non-intersecting, i.e., there is a super-exponential penalty incurred when they intersect. The pairwise pinning at the left boundary forces $\L_{2i-1}^N(1)-\L_{2i}^N(1)=O(1)$.}			
\label{figO1}
	\end{figure}

In light of the above line ensemble structure, proving our main theorem is equivalent to studying the increments of the top curve around the left boundary, i.e., $\big(\L_1^N(j)-\L_1^N(1)\big)_{j=1}^r$, as $N\to \infty$. The rest of the argument hinges on the following two (informally stated) results:

\smallskip

\begin{enumerate}[label=(\alph*)]
\setlength\itemsep{0.15cm}
    \item \label{itema} \textit{Separation of second and third curves}: With high probability, there is a separation (uniform in a window of size $\rho N^{2/3}$ for small enough $\rho>0$) of order $N^{1/3}$ between the second and third curves at the left boundary. (Proposition \ref{t.lsep23} and Corollary \ref{t.usep23})
    \item \label{itemb} \textit{Local convergence of softly non-intersecting pinned random walk bridges}: Consider two softly non-intersecting random walk bridges $(S_1(k),S_2(k))_{k\in \ll1,n\rr}$ of length $n$ started from (possibly random but fixed) initial data. For any fixed $r$, the law of $(S_1(k),S_2(k))_{k\in \ll1,r\rr}$ as $n\to \infty$ converges weakly to a Markov process with an explicit description. (Theorem \ref{thm:conv})
\end{enumerate}

\smallskip

\noindent We  will elaborate on the details of proving the two aforementioned items in Sections \ref{sec:1.2.2} and \ref{sec:1.2.3} respectively. Before doing so, note that the separation described in \ref{itema} allows us to show that, for small enough $\rho$, the law of $$\big(\L_1^N(k)-\L_1^N(1), \L_2^N(k)-\L_1^N(1)\big)_{k\in \ll1,\rho N^{2/3}\rr}$$  is close to the law of two log-gamma increment random walk bridges started from certain random initial data conditioned on soft non-intersection. Appealing to \ref{itemb}, the first $r$ pair of coordinates under the latter law can be shown to converge weakly to a Markov process $(S_1^{\uparrow}(k), S_2^{\uparrow}(k)\big)_{k=1}^r$ which matches with the one described in Section \ref{sec:1.1}. Thus we conclude Theorem \ref{t.main1}.

It is worth noting that the endpoints of the bridges are randomly distributed as well. Beyond establishing tightness of the endpoints at $N^{1/3}$ scale, little can be said about the exact nature of their distribution. Therefore, to ensure the weak convergence to the Markov process $(S_1^{\uparrow}(k), S_2^{\uparrow}(k)\big)_{k=1}^r$, we require our local convergence in \ref{itemb} to hold \textit{uniformly over all endpoints} of the random walk bridges. Much of the technical work in Sections \ref{sec31} and \ref{sec32} is dedicated to establishing uniform estimates for these random walk bridges.

We remark that in our proof we need to consider the odd points of $\L_1^N$ and even points of $\L_2^N$ due to the nature of Gibbs interaction. This is a technical detail that we will ignore in the introduction; the ideas and strategies illustrated above and in the rest of this subsection remain unchanged in the full proof.






\subsubsection{Uniform separation between second and third curves at the left boundary} \label{sec:1.2.2} 

To establish uniform separation between the second and third curves at the left boundary, we first claim pointwise separation precisely at the left boundary, i.e,
\begin{align}
    \label{deal2}
    \Pr\big( \L_2^N(1)-\L_3^N(1) \ge \delta N^{1/3}\big)
\end{align}
can be made arbitrarily close to $1$ by taking $\delta$ small enough. 
From here the uniform separation can be established by appealing to the process level tightness of the second and third curves. Although tightness was established only for the top curve in \cite{half1}, it is not hard to extend their arguments and techniques to prove a process-level tightness result for the top $k$ curves, for any $k$. We provide a complete proof of this result in Section \ref{sec6}. 

We now give a brief sketch of how we control \eqref{deal2}. Let us write $T=\lfloor N^{2/3} \rfloor$. The Gibbs property of the $\hslg$ line ensemble allows us to view the law of
\begin{align*}
    \big(\L_1^N(j), \L_2^N(j)\big)_{j=1}^{T}\mbox{ conditioned on } \L_1^N(T)=b_1,\,\L_2^N(T)=b_2, \mbox{ and }\L_3^N(j)=c_j, j\in \ll1, T\rr
\end{align*}
as two softly non-intersecting random walks started from $(b_1,b_2)$ conditioned to be pinned at the left boundary and with a soft non-intersection penalty for hitting the points $c_j$. 

We then make use of \textit{stochastic monotonicity} of the $\hslg$ line ensemble. Stochastic monotonicity is one of the standard tools in Gibbsian line ensemble arguments. It implies that if we decrease the boundary conditions, then the resulting measure is stochastically dominated by the original measure. Since we are interested in lower bounding the probability in \eqref{deal2}, we may decrease all $c_j$ with $j\ge 2$ down to $-\infty$ to decrease the probability (we need to keep $c_1=\L_3^N(1)$ fixed because of the nature of the event in \eqref{deal2}). Let us denote the random variables under this new conditional law by $(L_1(j), L_2(j))_{j=1}^{T}$. We then prove the following:

\begin{itemize}[leftmargin=15pt]
    \item Due to the soft non-intersection, the event $L_2(1)-c_1 \le -\delta N^{1/3}$ has small probability under the  conditional measure.
    
    \item Since there is no pinning between the second and third curves, any high probability event  remains a high probability event under the conditional law with boundary conditions $c_j=-\infty$ for $j\ge 2$. {We refer to this conditional law with boundary conditions $c_j=-\infty$ for $j\ge 2$ as \textit{interacting random walks} ($\m{IRW}$) with boundary conditions $(b_1,b_2)$ in the text (Definition \ref{def:IRW}). Again, the true definition is slightly different due to the parity structure of the Gibbs property.}  Using the explicit description of the $\m{IRW}$ law,  we show that under this law all limit points of $L_2(1)/N^{1/3}$ are non-atomic. This implies that under the conditional law the limit points of $L_2(1)/N^{1/3}$ are non-atomic as well; in other words, the  event $|L_2(1)-c_1| \le \delta N^{1/3}$ has small probability under the conditional measure. 
\end{itemize}
The above two bullet points imply the lower bound on \eqref{deal2}. The full details of this argument appear in Section \ref{sec4}.

\subsubsection{Convergence of softly non-intersecting random walks. } \label{sec:1.2.3} 

Our goal is to study the limiting law of
 \begin{align*}
     \big(L_1(j)-L_1(1), L_2(j)-L_1(1)\big)_{j=1}^{r},
 \end{align*}
 where $(L_1(j),L_2(j))_{r=1}^{n}$ are distributed as interacting random walks with some suitable boundary conditions. Here $n$ can be taken to be $\rho N^{2/3}$; it suffices that $n\to \infty$ as $N\to \infty$. Upon conditioning on $L_1(1)$, the shifted $\m{IRW}$ $(L_1(j)-L_1(1), L_2(j)-L_1(1))_{j=1}^{n}$, can be viewed as two softly non-intersecting random walk bridges started from certain random initial data which we describe now. 

\subsubsection*{Softly non-intersecting random walk bridges} Recall the random walk law $\Pr^{n;(0,\ga);\m{free}}$ introduced in Section \ref{sec:1.1}. Consider $(S_1(k),S_2(k))_{k=1}^n$ distributed as $\Pr^{n;(0,\ga);\m{free}}$. The softly non-intersecting random walk bridges are then {defined via
 \begin{equation} \label{wnlaw}
 \begin{split}
     \Pr_{W_n}^{n;(0,\ga);\m{free}}(\m{A}) &:= \frac{\Ex^{n;(0,\ga);\m{free}}[\ind_{\m{A}}W_n] }{\Ex^{n;(0,\ga);\m{free}}[W_n ]},\\
     \Pr_{W_n}^{n;(0,\ga);(b_1,b_2)}(\m{A}) &:= \frac{\Ex^{n;(0,\ga);\m{free}}[\ind_{\m{A}}W_n \mid S_1(n)=b_1, S_2(n)=b_2]}{\Ex^{n;(0,\ga);\m{free}}[W_n \mid S_1(n)=b_1, S_2(n)=b_2]},
\end{split}
 \end{equation}
 }
	where $W_n$ appears in \eqref{wscintro}. The soft non-intersection conditioning is expressed via the Radon--Nikodym derivative proportional to $W_n$. Indeed, note that there is a penalty of order $e^{-e^\delta}$ arising from $W_n$ whenever $S_1(k)-S_2(k) \le -\delta$. The right endpoints $(b_1,b_2)$ of the bridges are in principle random, but by a compactness argument it suffices to consider deterministic $(b_1,b_2)$ (varying with $n$) satisfying $b_i\in [-M\sqrt{n},M\sqrt{n}]$, $b_1-b_2\ge \frac{1}{M\sqrt{n}}$, and further $b_i/\sqrt{n}\to \beta_i\in\mathbb{R}$.  

 \medskip

 Let us write $\Ex^{n;(\m{h}_1,\m{h}_2);(b_1,b_2)}[\cdot]:=\Ex^{n;(\m{h}_1,\m{h}_2);\m{free}}[\cdot \mid S_1(n)=b_1,S_2(n)=b_2]$. We will give an outline of how to study the limit of the law in \eqref{wnlaw} with $W_n$ replaced by $\til{W}_n=\exp(-\sum_{k=2}^{n-1} e^{S_2(k)-S_1(k)})$, as the latter is easier to work with. 
 Take any event $\m{A}_r \in \sigma \{(S_1(k),S_2(k))_{k=1}^r\}$. Using the tower property of conditional expectation and the Markov property of random walk bridges, one has 
\begin{align}
\label{introiden}
    \Pr_{\til{W}_n}^{n;(0,\ga);(b_1,b_2)}(\m{A}_r)=\Ex^{n;(0,\ga);(b_1,b_2)}\left[\ind_{\m{A}_r}\cdot\til{W}_r\cdot V_n^{n-r+1;(b_1,b_2)}(S_1(r),S_2(r))\right],
\end{align}
	where
 \begin{align}\label{ratioimp}
     V_n^{m;(b_1,b_2)}(x,y):=\frac{A_n}{B_n}:=\frac{\Ex^{m;(x,y);(b_1,b_2)}[\til{W}_m]}{\Ex^{n;(0,\ga);(b_1,b_2)}[\til{W}_n]}.
 \end{align}
We then make the following deductions to compute the limit of \eqref{introiden}.

\begin{figure}[h!]
    \centering
    \includegraphics[scale=0.4]{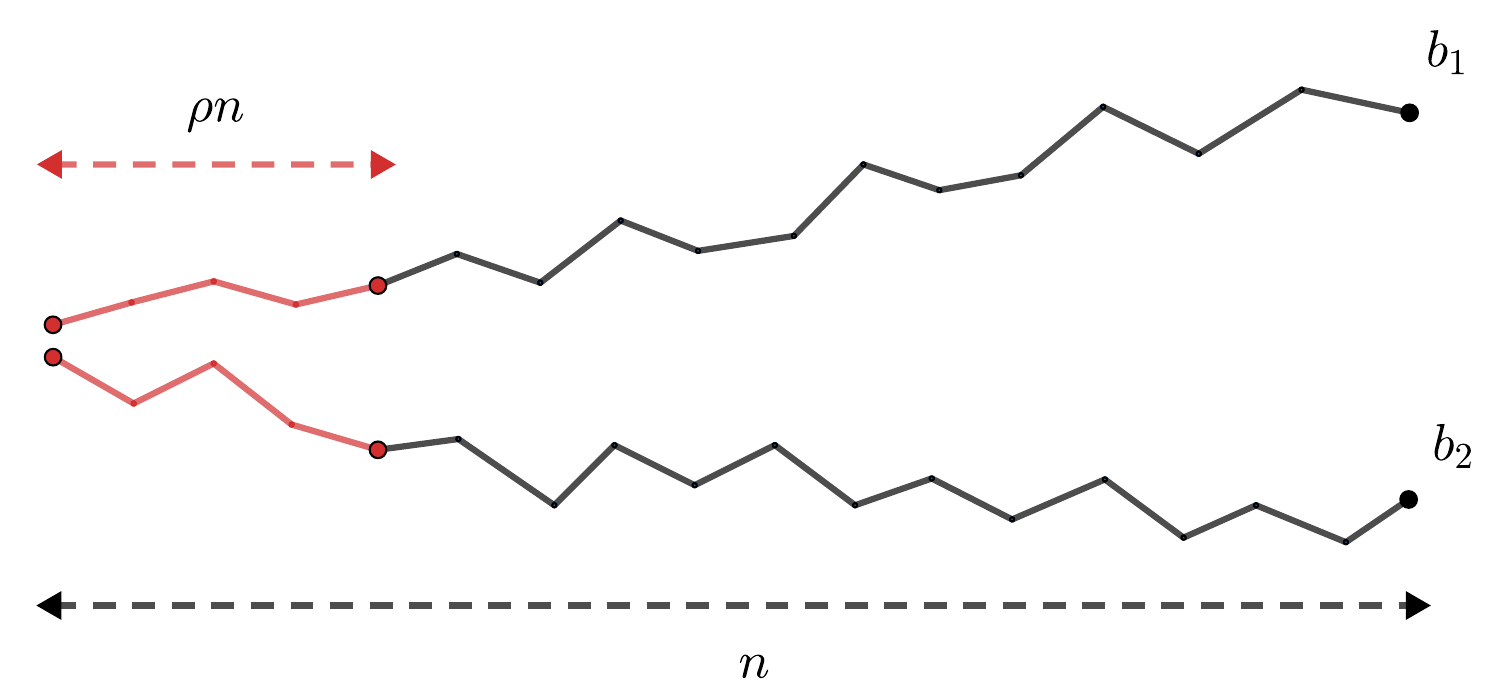}

    \caption{Two random walk bridges started at $O(1)$ distance apart and ending at $(b_1,b_2)$ are shown above. The distribution of the first $\rho n$ steps is close to that of free of random walks as we take $\rho$ small enough.}
    \label{fig:div}
\end{figure}

\begin{itemize}[leftmargin=15pt]
    \item As $n\to \infty$, the distribution of the initial part of the random walk bridge should not feel the effect of the endpoints $(b_1,b_2)$. In Lemma \ref{bridgewalk} we formalize this idea and show that $$\Ex^{n;(0,\ga);(b_1,b_2)}\big[g\big((S_1(k),S_2(k))_{k=1}^r\big)\big]=(1+o(1)) \cdot \Ex^{n;(0,\ga);\m{free}}\big[g\big((S_1(k),S_2(k))_{k=1}^r\big)\big]$$ for any integrable functional $g$. This allow us to replace $\Ex^{n;(0,\ga);(b_1,b_2)}$ appearing in \eqref{introiden} by $\Ex^{n;(0,\ga);\m{free}}$ with a $(1+o(1))$ multiplicative error.

    \item We next claim that as $m/n\to 1$ with $n\to\infty$, $V_n^{m;(b_1,b_2)}(x,y)$ converges to a limit which is independent of the endpoints $(b_1,b_2)$. Let us work with $m=n$ and write $U(k):=S_1(k)-S_2(k)$. It is known from \cite{spit} that the non-intersection probability of random walks $$\Pr^{n;(0,0);\m{free}}\big(U(k)\ge 0 \mbox{ for all }k\in \ll2,n\rr\big)$$ is asymptotically equal to $\mathfrak{c}\cdot n^{-1/2}$ for some explicit constant $\mathfrak{c}$. Comparing the soft and true non-intersection and using the above estimate, it is possible to show 
$cn^{-1/2}\le A_n , B_n \le Cn^{-1/2}.$
However the comparison technique is not powerful enough to produce the limit of $A_n/\sqrt{n}, B_n/\sqrt{n}$. 
Additionally, the limit $A_n/\sqrt{n}$ should depend on the initial data $(x,y)$ and it is not clear how to compute the limit for general initial data even in the true non-intersection case. We instead take the following route to argue the existence of the limit of the ratio.

\medskip

\noindent\textit{Our approach}. We consider the first $\rho n$ steps of the bridges (see Figure \ref{fig:div}). As in the first bullet point, this portion of the bridges does not feel the effect of the endpoints for $\rho$ small enough. We show that under the softly non-intersecting free law in \eqref{wnlaw}, $U(\rho n)/\sqrt{\rho  n}$ converges to the endpoint $R$ of a Brownian meander (Proposition \ref{meander}). It then follows that the remaining part of $U$ scaled diffusively, $(U(\rho n+tn)/\sqrt{n})_{t\in [0,1-\rho]}$, converges to a Brownian bridge $(B_t)_{t\in[0,1-\rho]}$ with $B_0=\sqrt{\rho} R$ and $B_{1-\rho}=\beta_1-\beta_2$ conditioned to stay positive (Proposition \ref{invarprin1}). Since $R$ is strictly positive a.s., so is $\mathsf Z^{(\beta_1,\beta_2)}:=\Pr(B_t >0 \mbox{ for all }t\in [0,1-\rho])$. {Factoring the weight $\widetilde{W}_n$ into two pieces corresponding to before and after time $\rho n$ and conditioning on the latter portion of the walk, the latter conditional expectation converges to $\m{Z}^{(\beta_1,\beta_2)}$, and we arrive at
\begin{align*}
    \Ex^{n;(x,y);(b_1,b_2)}[\til{W}_n] =(1+o(1)) \cdot \Ex^{\rho n;(x,y);\m{free}}[\til{W}_{\rho n}] \cdot \mathsf Z^{(\beta_1,\beta_2)}.
\end{align*}
}
This leads to
$V_n^{n;(b_1,b_2)}(x,y) = (1+o(1))V_{\rho n}^{\rho n; \m{free}}(x,y)$. From here a real analysis argument concludes the existence of the limit of the ratio in \eqref{ratioimp}. Let us call this limit $\til{V}(x,y)$; it is the analogue of $V$ defined in \eqref{Vdef}, with the Radon--Nikodym derivative $\til{W}_n$ instead of $W_n$.
\end{itemize}
Combining the above two bullet points, we eventually show that
\begin{align*}
     \lim_{n\to\infty}\Pr_{\til{W}_n}^{n;(0,\ga);(b_1,b_2)}(\m{A}_r)=\Ex^{r;(0,\ga);\m{free}}\left[\ind_{\m{A}_r}\cdot\til{W}_r\cdot \til{V}(S_1(r),S_2(r))\right].
\end{align*}
With a bit more technical work, one arrives at a Markov chain description of the above law. The details of the above argument are presented in Sections \ref{sec31} and \ref{sec32}. We mention that a version of the above problem was studied in \cite[Section 3.1]{cn16} for a single random walk conditioned to stay nonnegative, and a Doob $h$-transform formula analogous to \eqref{def:pv} was obtained. There the authors were able to express the corresponding function $V$ as a renewal function associated to a strict ascending ladder process of the random walk, which significantly simplified their argument. However, when dealing with soft non-intersection it is not clear how to obtain a similar description of $V$ and implement the same approach, so we perform the above analysis instead.

\subsection{Extension to other models}\label{sec:ext} 

The arguments described above can be used to extract convergence to stationary measures for other discrete half-space solvable models, such as the half-space geometric and exponential last passage percolation (LPP) models. These two models are solvable zero-temperature counterparts of the $\hslg$ polymer model, and they have been studied extensively from the early works \cite{rai00,br1,br01,br3,sis} to the more recent \cite{bbcs,bete,ale1,bz22}. Although the corresponding line ensemble is yet to be explicitly formulated, an analysis similar to ours should lead to a limiting measure of the form constructed in \cite{cn16} (the softly non-intersecting random walks should become truly non-intersecting at zero temperature). Since exponential LPP is a limit of the log-gamma polymer (see \cite[Section 3]{bc22} for example), a description of the stationary measure along the anti-diagonal path may be obtained directly by taking an appropriate limit of the law described in Section \ref{sec:1.1}.

There has been a large amount of recent work devoted to the construction of stationary measures for the half-space and open KPZ equations \cite{kpz86,it}, see \cite{bkld2,ck,bk,bkld,bd22,bc22,bcy} and the review \cite{cor22} for instance. Taking an intermediate disorder limit of the $\hslg$ polymer, one should be able obtain a line ensemble structure for the half-space KPZ equation. Once this is achieved, our approach can also be applied to study convergence of the increments of the half-space KPZ equation. We leave this for future consideration.

\subsection*{Organization} The remainder of the paper is organized as follows. In Section \ref{prelim} we collect some preliminary results related to the $\hslg$ line ensemble and various random walk models. In Section \ref{sec31}, we introduce the softly non-intersecting random walk bridges and establish various probability estimates. We then use these estimates to prove the convergence of softly non-intersecting random walk bridges around the left boundary in Section \ref{sec32}. In Section \ref{sec4}, we carry out the uniform separation argument illustrated in Section \ref{sec:1.2.2}. Finally, the proof of Theorem \ref{t.main1} appears in Section \ref{sec5}. One of the ingredients of our proof is the tightness of the full line ensemble, which we obtain by extending the arguments and techniques in \cite{half1}. Its proof is given in Section \ref{sec6}.

\subsection*{Acknowledgments} {We thank Ivan Corwin for conveying to us the idea regarding the separation of the first two curves from the rest of the line ensemble.} We thank Guillaume Barraquand and Ivan Corwin for their feedback on a draft of this paper, and we thank Amol Aggarwal and Amir Dembo for early discussions related to this project. {We thank the anonymous referee for their careful reading and useful comments
that helped improve the manuscript.} Part of this work was undertaken during the Columbia Probability Workshop at Columbia University in May 2023. We thank the organizers for their hospitality, and we acknowledge support from the Simons Foundation through Ivan Corwin's Investigator in Mathematics grant 929852.

\section{Preliminary results} \label{prelim}

In this section, we summarize a few results from \cite{half1}, referred to as \bcd\ in the
sequel, that form the toolbox of the proof of our main results. We also introduce various random walk models that arise in our analysis and explore their interconnections.

\subsection{$\hslg$ line ensemble and its Gibbs property} 
We start by defining the Gibbs property whose state space and associated weight function is  given by the following directed and colored (and labeled) graph. Define the graph $G$ with vertices $V(G):=\{(m,n): m\in \Z_{\ge 1}, n\in \Z_{<0}+\frac12\ind_{m\in 2\Z} \}$ and with the following directed colored (and labeled) edges:
	\begin{itemize}[leftmargin=15pt]
        \item For each $(m,n)\in \Z_{\ge 1}^2$, we put two {\color{blue}{\textbf{blue}}} edges from $$(2m+1,-n)\to(2m+2,-n+0.5)\mbox{ and }(2m+1,-n)\to(2m,-n+0.5).$$
         \item For each $(m,n)\in \Z_{\ge 1}^2$, we put two \black\ edges from $$(2m+2,-n-0.5) \to (2m+1,-n) \to\mbox{ and }(2m,-n+0.5) \to (2m+1,-n).$$
		\item For each $n\in \Z_{\ge 1}$, we put 
  \begin{itemize}[leftmargin=10pt]
      \item  a \textbf{black} edge: $(2,-2n-0.5)\to (1,-2n)$;  \ \ 
      a \purple\ (dashed) edge: $(2,-2n+0.5)\to (1,-2n+1).$
      \item a \blue\ edge $(1,-2n+1)\to (2,-2n+1.5)$; \ \ a \yellow\ edge $(1,-2n)\to (2,-2n+0.5).$
  \end{itemize}
	\end{itemize}
A portion of the corresponding graph is shown in Figure \ref{fig12}.	We write $E(G)$ for the set of edges of graph $G$ and $e=\{v_1\to v_2\}$ for a generic directed edge from $v_1$ to $v_2$ in $E(G)$ (the color of the edge is suppressed from the notation).

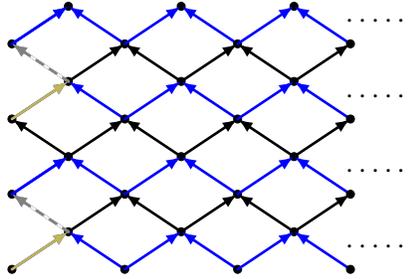
\begin{figure}[h!]
		\centering
			\begin{tikzpicture}[line cap=round,line join=round,>=triangle 45,x=1.5cm,y=1cm]
				\foreach \x in {0,1,2}
				{
					\draw [fill=black] (\x,0) circle (1.5pt);
					\draw [fill=black] (\x,-1) circle (1.5pt);
					\draw [fill=black] (\x,-2) circle (1.5pt);
					\draw [fill=black] (\x-0.5,-0.5) circle (1.5pt);
					\draw [fill=black] (\x-0.5,-1.5) circle (1.5pt);
					\draw [fill=black] (\x-0.5,-2.5) circle (1.5pt);
					\draw [fill=black] (\x-0.5,-3.5) circle (1.5pt);
					\draw [fill=black] (\x,-3) circle (1.5pt);
					\draw[line width=1pt,blue,{Latex[length=2mm]}-]  (\x,0) -- (\x-0.5,-0.5);
					\draw[line width=1pt,blue,{Latex[length=2mm]}-] (\x,0) -- (\x+0.5,-0.5);
					\draw[line width=1pt,black,{Latex[length=2mm]}-] (\x-0.5,-0.5) -- (\x,-1);
					\draw[line width=1pt,black,{Latex[length=2mm]}-] (\x+0.5,-0.5) -- (\x,-1);
					\draw[line width=1pt,blue,{Latex[length=2mm]}-]  (\x,-1) -- (\x-0.5,-1.5);
					\draw[line width=1pt,blue,{Latex[length=2mm]}-] (\x,-1) -- (\x+0.5,-1.5);
					\draw[line width=1pt,black,{Latex[length=2mm]}-] (\x-0.5,-1.5) -- (\x,-2);
					\draw[line width=1pt,black,{Latex[length=2mm]}-] (\x+0.5,-1.5) -- (\x,-2);
					\draw[line width=1pt,blue,{Latex[length=2mm]}-]  (\x,-2) -- (\x-0.5,-2.5);
					\draw[line width=1pt,blue,{Latex[length=2mm]}-] (\x,-2) -- (\x+0.5,-2.5);
					\draw[line width=1pt,black,{Latex[length=2mm]}-] (\x-0.5,-2.5) -- (\x,-3);
					\draw[line width=1pt,black,{Latex[length=2mm]}-] (\x+0.5,-2.5) -- (\x,-3);
					\draw[line width=1pt,blue,{Latex[length=2mm]}-]  (\x,-3) -- (\x-0.5,-3.5);
					\draw[line width=1pt,blue,{Latex[length=2mm]}-] (\x,-3) -- (\x+0.5,-3.5);
				}
			\draw[line width=1pt,white,{Latex[length=2mm]}-] (-0.5,-2.5) -- (0,-3);
				\draw[line width=1pt,white,{Latex[length=2mm]}-] (-0.5,-0.5) -- (0,-1);
				\draw[line width=1pt,gray,{Latex[length=2mm]}-,dashed] (-0.5,-2.5) -- (0,-3);
				\draw[line width=1pt,blue,{Latex[length=2mm]}-] (0,-2) -- (-0.5,-2.5);
				\draw[line width=1pt,yellow!70!black,{Latex[length=2mm]}-] (0,-3) -- (-0.5,-3.5);
				\draw[line width=1pt,gray,{Latex[length=2mm]}-,dashed] (-0.5,-0.5) -- (0,-1);
				\draw[line width=1pt,blue,{Latex[length=2mm]}-] (0,0) -- (-0.5,-0.5);
				\draw[line width=1pt,yellow!70!black,{Latex[length=2mm]}-] (0,-1) -- (-0.5,-1.5);
				\draw [fill=black] (2.5,-0.5) circle (1.5pt);
				\draw [fill=black] (2.5,-1.5) circle (1.5pt);
				\draw [fill=black] (2.5,-2.5) circle (1.5pt);
				\draw [fill=black] (2.5,-3.5) circle (1.5pt);
				\node at (2.8,-0.2) {$\cdots\cdots$};
				\node at (2.8,-1.2) {$\cdots\cdots$};
				\node at (2.8,-2.2) {$\cdots\cdots$};
				\node at (2.8,-3.2) {$\cdots\cdots$};
			\end{tikzpicture}
	\caption{The graph $G$ and its directed colored edges}
		\label{fig12}
	\end{figure}	
	
	We next define a bijection $\phi:V(G)\to \Z_{\ge 1}^2$ by $\phi((m,n))=(-\lfloor n \rfloor, m)$. This pushes the directed/colored edges in $G$ onto directed/colored edges on $\Z_{\ge 1}^2$ which we denote  by $E(\Z_{\ge 1}^2)$. We will always view $G$ as in Figure \ref{fig12} and will use the $\phi$-induced indexing when describing this graph. 
 
We associate to each $e\in E(\Z_{\ge 1}^2)$ a weight function based on its color defined as follows: 
	\begin{align}\label{def:wfn}
		W_{e}(x):=\begin{cases}
			\exp(\theta x-e^x) & \mbox{ if $e$ is \blue \color{blue}}, \\
			\exp(-e^x) & \mbox{ if $e$ is \black}, \\
			\exp(\alpha x-e^x) & \mbox{ if $e$ is \purple\,(dashed),} \\
			\exp((\theta+\alpha) x-e^x) & \mbox{ if $e$ is \yellow}.
		\end{cases}
	\end{align}

	\begin{theorem}[Half-space log-gamma line ensemble]\label{thm:conn} Fix $\alpha,\theta>0$, and $N\in \Z_{\ge 1}$.  Set $\mathcal{K}_N:=\{(i,j)\in \Z_{\ge 1}^2 : i\in [1,N], j\in [1,2N-2i+2]\}$. There exist random variables $\big(\L_i^N(j) : (i,j)\in \mathcal{K}_N\big)$, called here the $\hslg$ line ensemble, on a common probability space such that:
		\begin{enumerate}[label=(\roman*), leftmargin=20pt]
			\item \label{i01} We have the following equality in distribution:
			\begin{align}\label{impi}
				\big(\L_1^N(2j+1)\big)_{j\in \ll0,N-1\rr} \stackrel{(d)}{=} \big(\log Z(N+j,N-j)+2N\Psi(\theta)\big)_{j\in \ll0,N-1\rr},
			\end{align}
            where $\Psi(\theta) =  (\log\Gamma)'(\theta)$ is the digamma function.
			\item \label{i02} 
Let $\Lambda$  be any connected subset of $\{(i,j)\in \Z_{\ge 1}^2 : i\in [1,N-1], j\in [1,2N-2i+1]\}$. Set
		$$\partial \Lambda:=\big\{v\in \Z_{\ge 1}^2\cap \Lambda^c :\{v'\to v\}\in E(\Z_{\ge 1}^2) \mbox{ or }\{v\to v'\}\in E(\Z_{\ge 1}^2),\mbox{ for some } v'\in \Lambda\big\}.$$
  The law of $\big(\L_i^N(j) : (i,j)\in \Lambda\big)$ conditioned on $\big(\L_i^N(j): (i,j)\in \Lambda^c\big)$ is a measure on $\R^{|\Lambda|}$ with density  at $(u_{i,j})_{(i,j)\in \Lambda}$ proportional to
		\begin{align}\label{e:hsgibb}
			\prod_{e=\{v_1\to v_2\}\in E(\Lambda\cup \partial\Lambda)} W_{e}(u_{v_1}-u_{v_2}),
		\end{align}
		\end{enumerate}
where  $u_{i,j}=\L_i^N(j)$ for $(i,j)\in \partial \Lambda.$
	\end{theorem}
{The term $2N\Psi(\theta)$ appearing in \eqref{impi} encodes the law of large numbers of the log-partition function of the $\hslg$ polymer in the unbound phase (see Theorem 1.1 in \bcd).}
 We refer to the above measure on $\R^{|\Lambda|}$ with density  at $(u_{i,j})_{(i,j)\in \Lambda}$ proportional to \eqref{e:hsgibb} as the $\hslg$ \textit{Gibbs measure} with boundary conditions $(u_{i,j})_{(i,j)\in \partial \Lambda}$. 
 The precise description of the $\hslg$ line ensemble is given in \bcd. For the rest of the paper, we will take it as a black box, as we shall need only a few large-scale macroscopic properties of the object.
 
 We now state a few important properties of the $\hslg$ line ensemble and $\hslg$ Gibbs measures that were proven in \bcd.

\begin{lemma}[Translation invariance; Lemma 2.1 in \bcd]
    \label{traninv}
    Let $\big(L(v) : v\in \Lambda\big)$ be a collection of random variables distributed as the $\hslg$ Gibbs measure on the domain $\Lambda$ with some boundary conditions $\big(u_{i,j}:(i,j)\in \partial\Lambda\big)$.
 Then for any $c\in\mathbb{R}$, the law of $\big(L(v)+ c : v\in \Lambda\big)$ is the $\hslg$ Gibbs measure on the domain $\Lambda$ with boundary conditions $\big(u_{i,j}+c : (i,j)\in \partial\Lambda\big)$.	
\end{lemma}
\begin{proposition}[Stochastic monotonicity; Proposition 2.6 in \bcd]\label{p:gmc} Fix $k_1\le k_2$, $a_i\le b_i$ for $k_1\le i\le k_2$. Let
		\begin{align*}
			\Lambda:=\{(i,j): k_1\le i\le k_2, a_i\le j\le b_i\}.
		\end{align*}
		There exists a probability space that supports a collection of random variables
		\begin{align*}
			\big(L(v;(u_w)_{w\in \partial \Lambda}) : v\in \Lambda, (u_w)_{w\in\partial\Lambda} \in \R^{|\partial\Lambda|}\big)
		\end{align*}
		such that
		\begin{enumerate}[leftmargin=18pt]
			\item For each $(u_w)_{w\in\partial\Lambda}\in \R^{|\partial\Lambda|}$,  the marginal law of $\big(L(v;(u_w)_{w\in \partial \Lambda}): v\in \Lambda\big)$ is a measure on $\R^{|\Lambda|}$ with density  at $(u_{i,j})_{(i,j)\in \Lambda}$ proportional to
		\eqref{e:hsgibb}.
			\item With probability $1$, for all $v\in \Lambda$ we have
$$L\big(v;(u_w)_{w\in \partial \Lambda}\big) \le L\big(v;(u_w')_{w\in \partial \Lambda}\big) \mbox{ whenever }u_w\le u_w' \mbox{ for all }w\in \partial\Lambda.$$
Consequently, the probability of an increasing event under the $\hslg$ Gibbs measure increases if the boundary conditions are increased, and decreases if the boundary conditions are decreased.
		\end{enumerate}
	\end{proposition}

 The $\hslg$ line ensemble enjoys a soft non-intersection property which is captured in the following proposition.
 
\begin{proposition}[Theorem 3.1 in \bcd]
    \label{pp:order}  Fix any $k\in \Z_{\ge 1}$ and $\rho\in (0,1)$. There exists $N_0=N_0(\rho,k)>0$ such that for all $N\ge N_0$ we have $\Pr(\m{Ord}_{k,N}) \ge 1-\rho^N$, where
\begin{equation}
\label{t41}
    \begin{aligned}
        \m{Ord}_{k,N}& :=\bigcap_{i=1}^k \bigcap_{p=1}^{N-k-2} \left\{\max \big(\L_i^N(2p+1),\L_i^N(2p-1)\big) \le \L_i^N(2p)+(\log N)^{7/6}\right\} \\ & \hspace{3.5cm}\cap\left\{\L_{i+1}^N(2p) \le \min \big(\L_i^N(2p-1),\L_i^N(2p+1)\big)+(\log N)^{7/6}\right\}.
    \end{aligned}
    \end{equation}
\end{proposition}


The main result of \bcd\ demonstrates tightness of the top curve of the line ensemble. Here we extend their result to tightness of an arbitrary finite number of curves.

\begin{theorem}\label{p.tight3} For each $k\in \mathbb{Z}_{\geq 1}$ and $A>0$, the process $(\L_k^N(xN^{2/3})/N^{1/3})_{x\in [0,A]}$ is tight in the space $C([0,A])$ under the uniform topology.   
\end{theorem}

The proof of the above theorem is deferred to Section \ref{sec6}. It roughly mimics and generalizes the arguments present in  \bcd.

\subsection{Different random walk models and their properties}  In this section, we introduce various random walk laws that arise in the study of $\hslg$ Gibbs measures.

\begin{definition}\label{def:IRW} We define the \textit{interacting random walk} ($\m{IRW}$) law of length $T$ with boundary conditions $(a,b)$ to be the $\hslg$ Gibbs measure on the domain \begin{align}\label{k2t}
    \Phi=\big\{(i,j)\in\Z_{\ge 1}^2 : i\in \{1,2\}, j\in \ll 1,2T-1-\ind_{i=\operatorname{odd}}\rr\big\}
\end{align} with boundary conditions $u_{1,2T-1}=a$, $u_{2,2T}=b$, and $u_{3,2j}=-\infty$ for $j\in \ll1,T\rr$. We denote this measure by $\Pr_{\m{IRW}}^{T;(a,b)}$.
\end{definition}

\begin{figure}[h!]
    \centering
    \begin{tikzpicture}[line cap=round,line join=round,>=triangle 45,x=1.5cm,y=1cm]
				\foreach \x in {0,1}
    {
					\draw[line width=1pt,blue,{Latex[length=2mm]}-]  (\x,0) -- (\x-0.5,-0.5);
					\draw[line width=1pt,blue,{Latex[length=2mm]}-] (\x,0) -- (\x+0.5,-0.5);
					\draw[line width=1pt,black,{Latex[length=2mm]}-] (\x-0.5,-0.5) -- (\x,-1);
					\draw[line width=1pt,black,{Latex[length=2mm]}-] (\x+0.5,-0.5) -- (\x,-1);
					\draw[line width=1pt,blue,{Latex[length=2mm]}-]  (\x,-1) -- (\x-0.5,-1.5);
					\draw[line width=1pt,blue,{Latex[length=2mm]}-] (\x,-1) -- (\x+0.5,-1.5);
					\draw[line width=1pt,black,{Latex[length=2mm]}-] (\x-0.5,-1.5) -- (\x,-2);
					\draw[line width=1pt,black,{Latex[length=2mm]}-] (\x+0.5,-1.5) -- (\x,-2);
			}

    \draw[line width=1pt,blue,{Latex[length=2mm]}-]  (2,4-5) -- (1.5,4-5.5);
    \draw[line width=1pt,black,{Latex[length=2mm]}-]  (1.5,3-4.5) -- (2,3-5);
    \draw[line width=1pt,black,{Latex[length=2mm]}-]  (1.5,4-4.5) -- (2,4-5);

				\draw[line width=1pt,white,{Latex[length=2mm]}-] (-0.5,-0.5) -- (0,-1);

				\draw[line width=1pt,gray,{Latex[length=2mm]}-,dashed] (-0.5,-0.5) -- (0,-1);
				\draw[line width=1pt,blue,{Latex[length=2mm]}-] (0,0) -- (-0.5,-0.5);
				\draw[line width=1pt,yellow!70!black,{Latex[length=2mm]}-] (0,-1) -- (-0.5,-1.5);
    \foreach \x in {4,5}
				{
				
					\draw[line width=1pt,blue,{Latex[length=2mm]}-]  (\x,0) -- (\x-0.5,-0.5);
					\draw[line width=1pt,blue,{Latex[length=2mm]}-] (\x,0) -- (\x+0.5,-0.5);
					\draw[line width=1pt,black,{Latex[length=2mm]}-] (\x-0.5,-0.5) -- (\x,-1);
					\draw[line width=1pt,black,{Latex[length=2mm]}-] (\x+0.5,-0.5) -- (\x,-1);
					\draw[line width=1pt,blue,{Latex[length=2mm]}-]  (\x,-1) -- (\x-0.5,-1.5);
					\draw[line width=1pt,blue,{Latex[length=2mm]}-] (\x,-1) -- (\x+0.5,-1.5);

				}  
    \foreach \x in {2}
				{
    \foreach \y in {2,3,4}
    {
    \draw[line width=1pt,black,{Latex[length=2mm]}-]  (\y+5.5,\x-3.5) -- (\y+6,\x-4);
    }
    \foreach \y in {2,3}
    {
    \draw[line width=1pt,black,{Latex[length=2mm]}-]  (\y+6.5,\x-3.5) -- (\y+6,\x-4);
    }
    }
    \foreach \x in {4}
				{
    \draw[line width=1pt,blue,{Latex[length=2mm]}-]  (6,\x-5) -- (5.5,\x-5.5);
    \draw[line width=1pt,black,{Latex[length=2mm]}-]  (5.5,\x-4.5) -- (6,\x-5);
    }

				\draw[line width=1pt,white,{Latex[length=2mm]}-] (3.5,-0.5) -- (4,-1);

				\draw[line width=1pt,gray,{Latex[length=2mm]}-,dashed] (3.5,-0.5) -- (4,-1);
				\draw[line width=1pt,blue,{Latex[length=2mm]}-] (4,0) -- (3.5,-0.5);
				\draw[line width=1pt,yellow!70!black,{Latex[length=2mm]}-] (4,-1) -- (3.5,-1.5);
    \begin{scriptsize}
        \draw (1.55,-0.1) node[anchor=north] {$a$};
        \draw (0,-2.0) node[anchor=north] {$c_1$};
        \draw (1.0,-2.0) node[anchor=north] {$c_2$};
        \draw (2.05,-2.0) node[anchor=north] {$c_3$};
        \draw (2.05,-0.6) node[anchor=north] {$b$};
        \draw (5.55,-0.1) node[anchor=north] {$a$};
        \draw (6.05,-0.6) node[anchor=north] {$b$};
        \draw (8.0,-2.0) node[anchor=north] {$c_1$};
        \draw (9.0,-2.0) node[anchor=north] {$c_2$};
        \draw (10.05,-2.0) node[anchor=north] {$c_3$};
    \end{scriptsize}
    \draw (2.8,-0.8) node[anchor=north] {$=$};
    \draw (6.8,-0.8) node[anchor=north] {$\times$};
			\end{tikzpicture}
   \caption{Graphical structure for the $\hslg$ Gibbs measure $\Pr_\Phi^{T;(a,b,\vec{c})}$ (left) and $\Pr_{\m{IRW}}^{T;(a,b)}$ (middle) with $T=3$. The $\hslg$ Gibbs measure can be decomposed into the $\m{IRW}$ law along with a Radon--Nikodym derivative $\mathcal{H}(\vec{c})$ defined in \eqref{def:Gibbslaw} coming from the black edges shown on the right.}
    \label{fig:IRWgraph}
\end{figure}
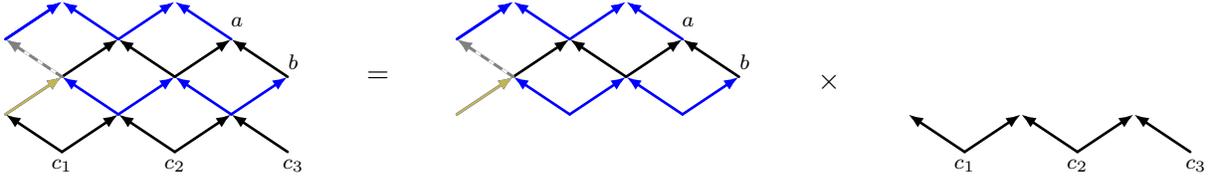

In the language of \bcd, $\m{IRW}$
precisely corresponds to the bottom-free measure on the domain $\mathcal{K}_{2,T}$ with boundary conditions $(a, b)$
(see Definition 2.4 in \bcd). The $\m{IRW}$ measures are useful in studying the $\hslg$ Gibbs measure on the same domain as $\Phi$ in \eqref{k2t} with general boundary conditions $(u_{i,j})_{(i,j)\in \partial\Phi}$. Indeed, if we denote the law of the latter as $\Pr_{\Phi}^{T;(a,b,\vec{c})}$, then $\Pr_{\Phi}^{T;(a,b,\vec{c})}$ is absolutely continuous w.r.t.~$\Pr_{\m{IRW}}^{T;(a,b)}$ with an explicit Radon--Nikodym derivative. More precisely, the probability of any event $\m{E}$ under $\Pr_{\Phi}^{(a,b,\vec{c})}$ can be written as
\begin{align}\label{def:Gibbslaw}
    \Pr_{\Phi}^{T;(a,b,\vec{c})}(\m{E})=\frac{\Ex_{\m{IRW}}^{T;(a,b)}[e^{-\mathcal{H}(\vec{c})}\ind_{\m{E}}]}{\Ex_{\m{IRW}}^{T;(a,b)}[e^{-\mathcal{H}(\vec{c})}]}, \quad \mathcal{H}(\vec{c}):=\sum_{j=1}^{T-1}\big[e^{c_j-L_2(2j-1)}+e^{c_j-L_2(2j+1)}\big]+e^{c_T-L_2(2T-1)},
\end{align}
where $a=u_{1,2T-1}$, $b=u_{2,2T}$, $c_j=u_{3,2j}$. The above formula will be very useful in transferring estimates from $\Pr_{\m{IRW}}^{T;(a,b)}$ to $\Pr_{\Phi}^{T;(a,b,\vec{c})}$. Note in particular that $\Pr_{\Phi}^{T;(a,b,(-\infty)^T)}$ is equal to $\Pr_{\m{IRW}}^{T;(a,b)}$.

\medskip

We next record a tightness result for $\m{IRW}$ on the diffusive scale.

\begin{lemma}\label{lem:IRWupperbd}For $i=1,2$, $M>0$, {and $a_1, a_2 \in \mathbb{R}$,} define the events
\begin{align}\label{def:diff}
\mathsf{Diff}_T^i(M) = \left\{\sup_{k\in\llbracket 1,2T-\mathbf{1}_{i=1}\rrbracket} \left|L_i(k) - a_i \right| \leq M\sqrt{T} \right\}.
\end{align}
For any $\e,K>0$, we can choose $M(\e,K)>0$ so that
\begin{align}\label{eq:diff}
\limsup_{T\to\infty} \sup_{|a_1-a_2| \le K\sqrt{T}} \Pr_{\mathsf{IRW}}^{T;(a_1,a_2)} \left(\neg\mathsf{Diff}_T^i(M)\right) < \e.
\end{align}
\end{lemma}

\begin{proof}  In Lemma 5.4 in \bcd\ it was shown that
for any $\e>0$, there exists $\mathfrak{M}=\mathfrak{M}(\e)>0$ such that
    \begin{align}\label{e:IRWupperbd}
\limsup_{T\to\infty}\Pr_{\m{IRW}}^{T;(0,-\sqrt{T})}\bigg(\sup_{i\in \ll1,2T-1\rr} |L_1(i)|+\sup_{j\in\ll1,2T\rr} |L_2(j)| \ge \mathfrak{M}\sqrt{T} \bigg) \le \e.
    \end{align}
From here, the proof of \eqref{eq:diff} follows from a straightforward application of stochastic monotonicity and translation invariance.  We illustrate this for $i=1$; the proof for $i=2$ is analogous.  Let us write $y = a_2-a_1$. We have $y \in [-K\sqrt{T},  K\sqrt{T}]$. By translation invariance, Lemma \ref{traninv}, we can write
    \begin{align*}
        \Pr_{\mathsf{IRW}}^{T;(a_1,a_2)} \left(\neg\mathsf{Diff}_T^1(M)\right) &= \Pr_{\mathsf{IRW}}^{T;(0,y)} \left(\sup_{k\in\llbracket 1,2T-1\rrbracket} L_1(k) \geq M\sqrt{T} \right) + \Pr_{\mathsf{IRW}}^{T;(0,y)} \left(\inf_{k\in\llbracket 1,2T-1\rrbracket} L_1(k) \leq -M\sqrt{T} \right). 
    \end{align*}
    To deal with the first term, we apply stochatisc monotonicity, Proposition \ref{p:gmc}, to shift the boundary data up from $(0,y)$ to $((K+1)\sqrt{T}, K\sqrt{T})$. This gives an upper bound of
    \[
    \Pr_{\mathsf{IRW}}^{T;(0,y)} \left(\sup_{k\in\llbracket 1,2T-1\rrbracket} L_1(k) \geq M\sqrt{T} \right) \leq \Pr_{\mathsf{IRW}}^{T;((K+1)\sqrt{T}, K\sqrt{T})} \left(\sup_{k\in\llbracket 1,2T-1\rrbracket} L_1(k) \geq M\sqrt{T} \right).
    \]
    Now applying translation invariance, Lemma \ref{traninv}, and shifting vertically by $-(K+1)\sqrt{T}$, the right-hand side is equal to
    \begin{align*}
    &\Pr_{\mathsf{IRW}}^{T;(0,-\sqrt{T})} \left(\sup_{k\in\llbracket 1,2T-1\rrbracket} \left(L_1(1)+(K+1)\sqrt{T}\right)\geq M\sqrt{T} \right) \\
    &\qquad\qquad\qquad\qquad\qquad \leq \Pr_{\mathsf{IRW}}^{T;(0,-\sqrt{T})} \left(\sup_{k\in\llbracket 1,2T-1\rrbracket} L_1(1) \geq (M-K-1)\sqrt{T} \right).
    \end{align*}
    Now \eqref{e:IRWupperbd} implies that the  probability on the right can be made less than $\e/2$ for large $T$ by choosing $M$ large enough depending on $K$ and $\e$.

    Similarly, we have
    \begin{align*}
        \Pr_{\mathsf{IRW}}^{T;(0,y)} \left(\inf_{k\in\llbracket 1,2T-1\rrbracket} L_1(k) \leq -M\sqrt{T} \right) &\leq \Pr_{\mathsf{IRW}}^{T;((K+1)\sqrt{T},K\sqrt{T})} \left(\inf_{k\in\llbracket 1,2T-1\rrbracket} L_1(k) \leq -M\sqrt{T} \right)\\
        &= \Pr_{\mathsf{IRW}}^{T;(0,-\sqrt{T})} \left(\inf_{k\in\llbracket 1,2T-1\rrbracket} L_1(k)  \leq -(M+K+1)\sqrt{T} \right).
    \end{align*}
    Choosing $M$ large enough depending on $\e$, \eqref{e:IRWupperbd} again implies that the last probability can be made less than $\e/2$ for large $T$.
\end{proof}

\medskip

We next recall the definition of paired random walks and weighted paired random walks from \bcd. Set $\Omega_n:=\R^n\times \R^n$ and let $\mathcal{F}_n$ be the Borel $\sigma$-algebra on $\Omega_n$. Write $\omega\in \Omega_n$ as $\omega=(\omega_1(1),\ldots,\omega_1(N),\omega_2(1),\ldots,\omega_2(N))$. As a slight abuse of notation, we will write 
	$$S_i(k)(\omega):=\omega_i(k), \quad k\in \ll1,n\rr, i\in \{1,2\}$$
	to denote the coordinate functions (i.e., random variables) in this space. 

 \begin{definition}[Paired random walks and weighted paired random walks] \label{prb}

Let $\fa(x)$ denote the density at $x\in \R$ of $\log Y_1-\log Y_2$ where $Y_1,Y_2$ are independent $\operatorname{Gamma}(\theta)$ random variables, and let \begin{equation}
    \label{defga}
    \ga(x)=\Gamma(\alpha)^{-1}e^{\alpha x-e^x}.
\end{equation} For $(x,y)\in \R^2$ and $n\in \mathbb{Z}_{\geq 2}$
the \textit{paired random walk} ($\m{PRW}$) law on $(\Omega_n,\mathcal{F}_n)$ is the probability measure $\Pr_{\m{PRW}}^{n;(x,y)}$ proportional to the product of two Dirac delta measures $\delta_{\omega_1(n)=x}\delta_{\omega_2(n)=y}$ and a density (against Lebesgue on $\R^{2(n-1)}$) given by
\begin{equation}
	\label{den}
	\begin{aligned} \ga\big(\omega_2(1)-\omega_1(1)\big)\prod_{k=2}^{n}\fa\big(\omega_1(k)-\omega_1(k-1)\big)\fa\big(\omega_2(k)-\omega_2(k-1)\big)\,d\omega_1(k)\, d\omega_2(k).
	\end{aligned}
\end{equation}
The \textit{weighted paired random walk} ($\m{WPRW}$) law $\Pr_{\m{WPRW}}^{n;(x,y)}$ on $(\Omega_n,\mathcal{F}_n)$ is absolutely continuous with respect to $\Pr_{\m{PRW}}^{n;(x,y)}$ and defined through a Radon--Nikodym derivative so that for all  $\m{A}\in \mathcal{F}_n$,
\begin{align}
	\label{wscint0} \Pr_{\m{WPRW}}^{n;(x,y)}(\m{A})=\frac{\Ex_{\m{PRW}}^{n;(x,y)}[\wsc\ind_{\m{A}}]}{\Ex_{\m{PRW}}^{n;(x,y)}[\wsc]},
\end{align}
where $\wsc=\wsc(\omega)$ is given by
\begin{align}\label{defw}
	\wsc:=\exp\bigg(-e^{\iise-\se{1}{2}}-\sum_{k=2}^{n-1} \left(e^{\iiks-\se{1}{k+1}}+e^{\iiks-\iks}\right)\bigg).
\end{align}
\end{definition}

Weighted paired random walks are connected to interacting random walks by the following lemma.

\begin{lemma}[Lemma 4.4 in \bcd] \label{wprwconn}
    Suppose $(L_1,L_2)$ are distributed as $\Pr_{\m{IRW}}^{T;(a,b)}$. Then the law of $(L_1(2i-1),L_2(2i))_{i=1}^T$ is $\Pr_{\m{WPRW}}^{T;(a,b)}$. 
\end{lemma}

Besides the above measures, we will make use of a variety of other $\hslg$ Gibbs measures and random walk type measures throughout the text. A summary of notation for many of the measures we will use is contained in the following table. Unless otherwise stated, the measures consist of two random walks.

{\renewcommand{\arraystretch}{1.2}
	\begin{longtable}[t]{l|p{0.5\textwidth}|l}
		\toprule	
                \multicolumn{3}{l}{Different $\hslg$ Gibbs measures used in the text} \\

            \midrule
            
			$\Pr_{\m{IRW}}^{T;(a,b)}$ & Interacting random walks of length $T$ with right boundary conditions $(a,b)$ & Def.~\ref{def:IRW} \\
   $\Pr_{\Phi}^{T;(a,b,\vec{c})}$ & $\hslg$ Gibbs measure on domain $\Phi$ in \eqref{k2t} with right/bottom boundary conditions $(a,b,\vec{c})$ & Eq.~\eqref{def:Gibbslaw} \\
   
   $\Pr_{\Upsilon}^{T;(a,b,\vec{c})}$ & $\hslg$ Gibbs measure on domain $\Upsilon$ in \eqref{fre} with right/top boundary conditions $(a,b,\vec{c})$& Eq.~\eqref{def:Glaw2} \\

   $\Pr_{m}^{T;\vec{x}}$ & $2m$ interacting random walks of length $T$ with right boundary conditions $\vec{x}$& Def.~\ref{mirw}\\

   \midrule

    \multicolumn{3}{l}{Different random walk measures on $(\Omega_n,\mathcal{F}_n)$ used in the text} \\

    \midrule
    
   $\Pr_{\m{PRW}}^{n;(a,b)}$ & Paired random walks of length $n$ with right boundary conditions $(a,b)$& Def.~\ref{prb}\\
$\Pr_{\m{WPRW}}^{n;(a,b)}$ &Weighted paired random walks of length $n$ with right boundary conditions $(a,b)$ & Def.~\ref{prb}\\
   $\Pr^{n;(\m{h}_1,\m{h}_2);\m{free}}$ & Independent random walks of length $n$ with left boundary conditions $(\m{h}_1,\m{h}_2)$& Def.~\ref{def:rwrb}\\
   $\Pr^{n;(\m{h}_1,\m{h}_2);(b_1,b_2)}$ &Independent random walk bridges of length $n$ with left/right boundary conditions $(\m{h}_1,\m{h}_2)$ and $(b_1,b_2)$ & Def.~\ref{def:rwrb}\\
   $\Pr_{\wsc}^{n;(\m{h}_1,\m{h}_2);(b_1,b_2)}$ or $\Pr_{\wsc}^{n;(\m{h}_1,\m{h}_2);\m{free}}$ & Softly non-intersecting random walks/bridges of length $n$ with corresponding boundary conditions& Def.~\ref{def:sni}\\
		\bottomrule
	\end{longtable}
}

\section{Softly non-intersecting random walks and bridges}    \label{sec31} 
 In this section, we introduce the general setup of softly non-intersecting random walks and bridges and prove various related estimates. We first introduce random walks and bridges with possibly random initial conditions. 
	\begin{definition}[Random walks and bridges with initial conditions] \label{def:rwrb} Suppose $\mathsf{h}_1,\mathsf{h}_2, \mathsf{f}$ are three densities on $\R$. We define the probability measure $\Pr^{n;(\mathsf{h}_1,\mathsf{h}_2);\m{free}}$ on $(\Omega_n,\mathcal{F}_n)$ with density (against Lebesgue on $\R^{2n}$) given by
	\begin{align}\label{freedensity}
		\prod_{i=1}^2 \left[ \mathsf{h}_i(\omega_i(1))\prod_{k=2}^n \mathsf{f}(\omega_i(k)-\omega_i(k-1))\prod_{k=1}^n d\omega_i(k)\right].
	\end{align}
	For $b_1,b_2\in \R$, we define $\Pr^{n;(\mathsf{h}_1,\mathsf{h}_2);(b_1,b_2)}$ on $(\Omega_n,\mathcal{F}_n)$ to be the probability measure proportional to the product of two Dirac delta functions $\delta_{\omega_1(n)=b_1}\delta_{\omega_2(n)=b_2}$ and a density (against Lebesgue on $\R^{2n-2}$) given by
	\begin{align}\label{bridgedensity}
		\prod_{i=1}^2 \left[ \mathsf{h}_i(\omega_i(1))\prod_{k=2}^n \mathsf{f}(\omega_i(k)-\omega_i(k-1))\prod_{k=1}^{n-1} d\omega_i(k)\right].
	\end{align}
	We may extend the definition of the above two measures to include Dirac delta functions: $\mathsf{h}_i(\omega_i(1))=\delta_{\omega_i(1)=a}$. In that case, we shall write $\Pr^{n;(a_1,a_2);\m{free}}$, $\Pr^{n;(a_1,a_2);(b_1,b_2)}$ for the above two measures. Note that the above laws depends on $\m{f}$ as well which we have suppressed from the notation.
\end{definition}	
We will be interested in a particular class of initial conditions defined below. 

\begin{definition}\label{def:ic} Fix  $M> 0$. We shall say 
		 $(\m{h}_1, \m{h}_2)\in \m{IC}(M)$ if for each $i\in \{1,2\}$, either $\m{h}_i(x)=\delta_{x=a_i}$ where $|a_i|\le M$ or $\m{h}_i(x)\le Me^{-|x|/M}$ for all $x\in \R$.
\end{definition}

Throughout this section, we shall also assume $\m{f}$ satisfies the following conditions. 

\begin{assumption}[Assumptions on the increments] \label{asp:in} The density $\m{f}$ satisfies the following properties:
	\begin{enumerate}
		\item $\ffa$ is symmetric and $\log \ffa$ is concave.
		
		\item Let $\psi$ denote the characteristic function corresponding to $\ffa$. Then $|\psi|$ is integrable. Given any $\delta>0$, there exists $\eta$ such that $\sup_{t\ge \delta} |\psi(t)|=\eta <1$.
		
		\item There exists a constant $\Con>0$ such that $\ffa(x)\le \Con e^{-|x|/\Con}$. In particular, this implies that if $X \sim \ffa$ then there exists $v>0$ such that
		\begin{align*}		\sup_{|t|\le v}\Ex[e^{tX}] <\infty.
		\end{align*}
		In other words, $X$ is a subexponential random variable.
	\end{enumerate}
\end{assumption}

The above assumptions originate from \bcd\ where the authors utilized these conditions to provide several estimates on non-intersection probabilities of random walks and bridges with increments from $\m{f}$. We shall use many of these estimates in our analysis. 

Random walks are generally much simpler to analyze than random bridges; the following lemma helps us transfer probability estimates of certain events under the free law to the same events under a bridge law. 

\begin{lemma}\label{lem:mrb} Fix $M>0$. Suppose $\m{A} \in \sigma \big\{(S_1(k),S_2(k))_{k=1}^{\lfloor n/4 \rfloor}\big\}$ or $\m{A} \in \sigma \big\{(S_1(k),S_2(k))_{k=\lfloor 3n/4\rfloor}^{n}\big\}$. Then there is a constant $\Con>0$ such that for all $a_1,a_2,b_1,b_2$ satisfying $|a_i-b_i|\le M\sqrt{n}$ we have
$$\Pr^{n;(a_1,a_2);(b_1,b_2)}(\m{A}) \le \Con \cdot\Pr^{n;(a_1,a_2);\m{free}}(\m{A}).$$
\end{lemma}

The proof follows easily by comparing the two densities, cf. Lemma 4.10 in \bcd.

The next lemma allows us to compare the first $r$ steps of a random walk bridge to those of a random walk, when $r/n$ is sufficiently small. In this lemma and in most of the paper, we will write $S_i(\rho n)$ for $S_i(\lfloor \rho n\rfloor)$ for brevity when $\rho n$ is not an integer.

\begin{lemma}\label{bridgewalk}
Fix any $M,K,\rho >0$, and let $g$ be any integrable functional of $(S_1(k),S_2(k))_{k=1}^{\rho n}$. Define the event $\m{C}_K = \{\max(|S_1(\rho n)|,|S_2(\rho n)|) \leq K\sqrt{\rho n} \}$. {Let $$B_M = \{(b_1,b_2)\in\mathbb{R}^2 : |b_1|+|b_2| \leq M\sqrt{n}, b_1-b_2 \geq \tfrac{1}{M}\sqrt{n}\}.$$} Then for any $\e>0$, we can find $n_0$ and $\rho_0$ depending on $M,K,\e$ so that for all $n\geq n_0$ and $\rho\leq\rho_0$,
\[
\sup_{(b_1,b_2)\in B_M} \left| \frac{\Ex^{n;(\m{h}_1,\m{h}_2);(b_1,b_2)}[g(S_1,S_2) \mathbf{1}_{\m{C}_K}]}{\Ex^{n;(\m{h}_1,\m{h}_2);\m{free}}[g(S_1,S_2)\mathbf{1}_{\m{C}_K}]} - 1 \right| < \e.
\]
\end{lemma}

\begin{proof}[Proof of Lemma \ref{bridgewalk}]
    The proof proceeds by estimating the Radon--Nikodym derivative between the two measures. {In the definition of the conditional density of the random walk bridges (via convolution of the log-gamma density $\m{f}$), we integrate separately over the trajectory before and after time $\rho n$. This leads to}
    \begin{align*}
        \Ex^{n;(\m{h}_1,\m{h}_2);(b_1,b_2)}\left[g(S_1,S_2)\mathbf{1}_{\m{C}_K}\right] &= \Ex^{n;(\m{h}_1,\m{h}_2);\m{free}}\left[g(S_1,S_2)\mathbf{1}_{\m{C}_K} \cdot D^{n;(b_1,b_2)}\right],
    \end{align*}
    where
    \begin{equation}\label{RNderiv}
        D^{n;(b_1,b_2)}  = \frac{\m{f}^{*(n-\rho n)}(S_1(\rho n)-b_1) \m{f}^{*(n-\rho n)}(S_2(\rho n)-b_2)}{\int_{\mathbb{R}^2} \m{h}_1(x_1)\m{h}_2(x_2) \m{f}^{*(n-1)}(x_1-b_1)\m{f}^{*(n-1)}(x_2-b_2)\,dx_1\,dx_2},
    \end{equation}
     where $\m{f}^{\ast k}$ denotes the $k$-fold convolution of $\m{f}$. It therefore suffices to show that for large $n$ and small $\rho$, on $\m{C}_K$ we have $|D^{n;(b_1,b_2)} - 1| < \e$ uniformly over $b_1,b_2$. We now seek estimates for the convolution. From \cite[Chapter XV.5, Theorem 2]{feller}, we have \begin{equation}\label{convolution}
 \sup_{z\in\R} \left|\sqrt{k}\,\m{f}^{\ast k}(z)-(2\pi \sigma^2)^{-1/2}e^{-z^2/2k\sigma^2}\right| \longrightarrow 0
 \end{equation} as $k\to \infty$, where $\sigma^2:=\int x^2\m{f}(x)dx$.  This implies, uniformly over $x_i$ and $S_i(r)$,
    \begin{align}
        \m{f}^{*(n-\rho n)}(S_i(\rho n) - b_i) &= \frac{1}{\sqrt{2\pi\sigma^2(n-\rho n)}}\exp\left(-\frac{(S_i(\rho n)-b_i)^2}{2(n-\rho n)\sigma^2}\right) + o_n\left(\frac{1}{\sqrt{n-\rho n}}\right),\label{RNnum}\\
        \m{f}^{*(n-1)}(x_i-b_i) &= \frac{1}{\sqrt{2\pi\sigma^2(n-1)}} \exp\left(-\frac{(x_i-b_i)^2}{2(n-1)\sigma^2}\right) + o_n\left(\frac{1}{\sqrt{n}}\right).\label{RNdenom}
    \end{align}
   Here $o_a(b)$ denotes any term that goes to zero 
 upon dividing by $b$ as $a\to \infty$. Integrating \eqref{RNdenom}, we can write the denominator in \eqref{RNderiv} as
    \begin{align*}
    \frac{1}{2\pi \sigma^2(n-1)} \int_{\mathbb{R}^2} \m{h}_1(x_1)\m{h}_2(x_2)\exp\left(-\frac{(x_1-b_1)^2}{2(n-1)\sigma^2} - \frac{(x_2-b_2)^2}{2(n-1)\sigma^2}\right)dx_1\,dx_2 + o_n\left(\frac{1}{n}\right).
    \end{align*}
    Combining with \eqref{RNnum}, we can write \eqref{RNderiv} as
    \begin{align}
        \frac{1}{D^{n;(b_1,b_2)}} &= \frac{n-\rho n}{n-1} \prod_{i=1}^2 \int_{\mathbb{R}} \m{h}_i(x_i)\exp\left(\frac{(S_i(\rho n)-b_i)^2}{2(n-\rho n)\sigma^2} - \frac{(x_i-b_i)^2}{2(n-1)\sigma^2}\right)dx_i + o_n(1)\\
        &= \frac{n-\rho n}{n-1} \prod_{i=1}^2 \bigg[  \exp\left(\frac{S_i(\rho n)^2}{2(n-\rho n)\sigma^2} - \frac{b_i^2}{2\sigma^2}\left(\frac{1}{n-\rho n} - \frac{1}{n-1}\right) - \frac{b_i S_i(\rho n)}{(n-\rho n)\sigma^2}\right)\label{RNexp}\\
        & \qquad \qquad\qquad \qquad \cdot \int_{\mathbb{R}} \m{h}_i(x_i) \exp\left(\frac{b_i x_i}{(n-1)\sigma^2} \right)dx_i \bigg] + o_n(1)\label{RNint}.
    \end{align}
    The prefractor $\frac{n-\rho n}{n-1}$ can of course be made arbitrarily close to 1 for small $\rho$. On the event  $\m{C}_K$, the first term in the exponential is of order $K^2\rho$. Since $b_i\in B_M$, the second term in the exponent is of order $M\rho$, and on $\m{C}_K$ the third term is of order $MK\sqrt{\rho}$. Thus all three of these terms are $o_\rho(1)$, and the exponential in \eqref{RNexp} is $1+o_\rho(1)$. Lastly, we claim that the integral in \eqref{RNint} is $1+o_{n}(1)$. Indeed, the tail bounds $\m{h}_i(x_i) \leq Me^{-|x_i|/M}$ from Definition \ref{def:ic} along with the fact that $|b_i| \leq M\sqrt{n}$ provide a lower bound of
    \begin{align*}
        \int_{-n^{1/4}}^{n^{1/4}} \m{h}_i(x_i) \exp\left(-\frac{M\sqrt{n}\cdot n^{1/4}}{(n-1)\sigma^2}\right)dx_i 
        &\geq e^{-Mn^{-1/4}/\sigma^2} \left( 1- 2\int_{n^{1/4}}^\infty Me^{-x_i/M}\,dx_i\right)\\
        &= e^{-Mn^{-1/4}/\sigma^2} \left( 1- 2M^2 e^{-n^{1/4}/M}\right) = 1-o_n(1).
    \end{align*}
    On the other hand we have an upper bound of
    \begin{align*}
        &\int_{-n^{1/4}}^{n^{1/4}} \m{h}_i(x_i)\exp\left(\frac{M\sqrt{n} \cdot n^{1/4}}{(n-1)\sigma^2}\right)dx_i + \int_{|x|>n^{1/4}} M\exp\left(-\frac{|x_i|}{M} + \frac{M\sqrt{n}\,|x_i|}{(n-1)\sigma^2}\right)dx_i \\
        &\qquad\qquad \leq 1+o_n(1) + \int_{|x|>n^{1/4}} M\exp\left(-|x_i|\left(\frac{1}{M}-o_n(1)\right)\right)dx_i\\
        &\qquad\qquad = 1+o_n(1) + (M^2+o_n(1))e^{-n^{1/4}/M + o_n(1)} = 1+o_n(1).
    \end{align*}
    Inserting the above estimates into \eqref{RNexp} and \eqref{RNint}, we obtain $D^{n;(b_1,b_2)} = 1+o_{n,\rho}(1)$ uniformly over $(b_1,b_2)\in B_M$, completing the proof.   
\end{proof}
\medskip

It is obvious that for any fixed $(\m{h}_1,\m{h}_2)\in \m{IC}(M)$ and fixed sequences $b_1(n),b_2(n) \in [-M\sqrt{n},M\sqrt{n}]$, we have tightness of the initial points $(S_1(1),S_2(1))$ under $\Pr^{n;(\m{h}_1,\m{h}_2),(b_1,b_2)}$ as we vary $n$. The following lemma shows that this tightness holds for any initial conditions in $\m{IC}(M)$, uniformly over all endpoints $b_1,b_2$ in the diffusive window.

\begin{lemma}[Uniform tightness of initial points]\label{l:tip} Fix $M>0$. Suppose  $(\m{h}_1, \m{h}_2)\in \m{IC}(M)$. There exist constants $\Con=\Con(M)>0$ and $n_0(M)>0$ such that for all $R>0$, $i\in \{1,2\}$, and $n\ge n_0$,
     \begin{align*}
          \sup_{|b_1|+|b_2|\le M\sqrt{n}} \Pr^{n;(\m{h}_1,\m{h}_2),(b_1,b_2)}\big(|\se{i}{1}| \ge R\big) \le \Con e^{-R/\Con}.
     \end{align*}
\end{lemma}
  \begin{proof} 
  The density of $\se{i}{1}$ is proportional to $\m{h}_i(x)\cdot \m{f}^{\ast (n-1)}(x-b_i),$
  where, as before, $\m{f}^{\ast k}$ denotes the $k$-fold convolution of $\m{f}$. Let us choose a rectangle $I=[-M_1,M_1]$ (depending on $M$) such that $\int_{I} \m{h}_i(x)dx \ge 1/2$. Note that for any Borel $A\subset \R$ we have 
  \begin{equation}\label{2.1}
       \begin{aligned}
     &  \Pr^{n;(\m{h}_1,\m{h}_2),(b_1,b_2)}\left(\se{i}{1}\in A\right) \le \frac{\int_{ A} \m{h}_i(x)\cdot  \m{f}^{\ast (n-1)}(x-b_i)dx}{\int_{ I} \m{h}_i(x)\cdot  \m{f}^{\ast (n-1)}(x-b_i)dx}.
  \end{aligned}
  \end{equation}
Using \eqref{convolution}, for large enough $n$, one can ensure $\sqrt{n-1}\, \m{f}^{\ast (n-1)}(x_i-b_i) \le 1+(2\pi \sigma^2)^{-1/2}$. Since $|b_i|\le M\sqrt{n}$, for all large enough $n$ we have $\sqrt{n-1}\,\m{f}^{\ast (n-1)}(x_i-b_i)  \ge (4\pi \sigma^2)^{-1/2}e^{-M^2/\sigma^2}$  for all $(x_1,x_2)\in I$. Plugging these bounds back in the right-hand side of \eqref{2.1} {and recalling that $\int_{I} \m{h}_i(x)dx \ge 1/2$} we get
  \begin{align*}
  \mbox{r.h.s.~of \eqref{2.1}} \le 2\sqrt{4\pi \sigma^2} \, e^{M^2/\sigma^2}(1+(2\pi \sigma)^{-1/2})\int_{A} \m{h}_i(x)dx.
  \end{align*}
for all large enough $n$. Taking  $R$ large enough, setting $A:=(-R,R)^c$ and utilizing the fact that $\m{h}_i$ have exponential tails, we get the desired result.
  \end{proof}

We now define the softly non-intersecting version of the above random walks and bridges. Some of our arguments will involve splitting the walks into two pieces, and towards this end we introduce the notation
\begin{equation}\label{wsplit}
{W}_{r\to n} := \exp\left(-e^{S_2(r)-S_1(r+1)} - \sum_{k=r+1}^{n-1} \left(e^{S_2(k)-S_1(k+1)}+e^{S_2(k)-S_1(k)}\right)\right),
\end{equation}
for any $r<n$. We write $W_n:=W_{1\to n}$, and note that this agrees with the definition of $\wsc$ in \eqref{defw}. We will also require a variant of $W_n$:
\begin{equation}\label{wrdef}
    \hat{W}_{r} := \exp\left(-\sum_{k=1}^{r-1} \left(e^{S_2(k)-S_1(k+1)} + e^{S_2(k+1)-S_1(k+1)}\right)\right) \mbox{ for }r\ge 2, {\quad\mbox{and } \hat{W}_1 := 1.}
\end{equation}
Note that $\hat{W}_r=W_r \cdot \exp\left(-e^{S_2(r)-S_1(r)}\right)$ and 
\begin{align}
    \label{decomp}
    W_{a+b}=\hat{W}_a \cdot W_{a\to a+b}
\end{align}
for any $a,b \in \Z_{\ge 1}$. Let $\mathcal{F}_a$ denote the $\sigma$-algebra generated by $(S_1(k),S_2(k))_{k=1}^a$. Then we observe that $\widehat{W}_a$ is $\mathcal{F}_a$-measurable, and the Gibbs property for random walks and bridges implies that
\begin{align}
\label{transfin}
    \Ex^{a+b;(\m{h}_1,\m{h}_2);\bullet}\left[{W}_{a\to a+b}\mid \mathcal{F}_a \right]=\Ex^{b+1;(S_1(a),S_2(a));\bullet}\left[{W}_{b+1}\right].
\end{align}
where $\bullet\in \{(b_1,b_2), \m{free}\}$.

\begin{definition}[Softly non-intersecting random walks and bridges] \label{def:sni} For $(\m{h}_1,\m{h}_2) \in \m{IC}(M)$ and $b_1,b_2\in \mathbb{R}$, we define a weighted probability measure $\Pr_{\wsc}^{n;(\m{h}_1,\m{h}_2),\bullet}$ on $(\Omega_n,\mathcal{F}_n)$ that is absolutely continuous with respect to $\Pr^{n;(\m{h}_1,\m{h}_2);\bullet}$ with Radon--Nikodym derivative $\wsc$ defined in \eqref{defw}. That is, for all $\m{A}\in \mathcal{F}_n$,
	\begin{align}
		\label{wscint}
		\Pr_{\wsc}^{n;(\m{h}_1,\m{h}_2);\bullet}(\m{A})=\frac{\Ex_{}^{n;(\m{h}_1,\m{h}_2);\bullet}[\wsc\ind_{\m{A}}]}{\Ex_{}^{n;(\m{h}_1,\m{h}_2);\bullet}[\wsc]}.
	\end{align}
 Here $\bullet \in \{\m{free}, (b_1,b_2)\}$.
\end{definition}

Note that there is a penalty of order $e^{-e^\delta}$ in the above Radon--Nikodym derivative whenever $S_1(k)-S_2(k) \le -\delta$, which justifies the ``soft non-intersection'' terminology.

 We now give a lemma which relates these softly non-intersecting bridge measures to the weighted paired random walk measures from Definition \ref{prb}.

 \begin{lemma}\label{lem:IRWshift}
 Fix any $n\in\mathbb{Z}_{\geq 1}$ and $a,b\in\mathbb{R}$. Recall the $\m{WPRW}$ law and the densities $\fa$ and $\ga$ from Definition \ref{prb}. Suppose $(S_1(k),S_2(k))_{k=1}^n$ has law $\Pr_{\m{WPRW}}^{n;(a,b)}$. Then the conditional law of $\big((S_1(k)-S_1(1),S_2(k)-S_1(1))\big)_{k=1}^n$ given $S_1(1)$ is $\Pr_{\wsc}^{n;(0,\m{h});(a',b')}$, where $\m{f}(x)=\fa(x)$,  $\m{h}(x) = \ga$ and $a':=a-S_1(1)$, $b':=b-S_1(1)$.
 \end{lemma}

 \begin{proof}
     This is a straightforward computation from the definitions. {Suppose $(S_1(k),S_2(k))_{k=1}^n$ has law $\Pr_{\m{WPRW}}^{n;(a,b)}$. The joint density of $(S_1(k),S_2(k))_{k=1}^n$ at $(\omega_1(k),\omega_2(k))_{k=1}^n$} is proportional to the product of two Dirac delta functions $\delta_{\omega_1(n)=x}\delta_{\omega_2(n)=y}$ and a density given by 
     \begin{align*}
         \ga\big(\omega_2(1)-\omega_1(1)\big)\wsc(\omega)\prod_{k=2}^{n}\fa\big(\omega_1(k)-\omega_1(k-1)\big)\fa\big(\omega_2(k)-\omega_2(k-1)\big)\,d\omega_1(k)\, d\omega_2(k).
     \end{align*}
    {By a change of variables, the joint density of $S_1(1),((S_1(k)-S_1(1))_{k=2}^n,(S_2(k)-S_1(1)))_{k=1}^n$ at $\omega_1(1),(\omega_1(k))_{k=2}^n, (\omega_2(k))_{k=1}^n$} is proportional to the product of two Dirac delta functions $\delta_{\omega_1(n)-\omega_1(1)=x}$ and $\delta_{\omega_2(n)-\omega_1(1)=y}$ times a density given by 
    \begin{equation}\label{denab}
         \begin{aligned}
& \ga\big(\omega_2(1)\big)\wsc(\omega) \fa\big(\omega_1(2)\big)\fa\big(\omega_2(2)-\omega_2(1)\big) \\ & \hspace{2cm} \cdot \prod_{k=3}^{n}\fa\big(\omega_1(k)-\omega_1(k-1)\big)\fa\big(\omega_2(k)-\omega_2(k-1)\big)\prod_{k=2}^n\,d\omega_1(k)\, d\omega_2(k).
     \end{aligned}
    \end{equation}
Note that in the above $\omega_1(1)$ only appears in the Dirac delta functions. The conditional law of $\big((S_1(k)-S_1(1),S_2(k)-S_1(1))\big)_{k=1}^n$ given $S_1(1)$ is thus given by $\delta_{\omega_1(1)=0}\delta_{\omega_1(n)=x+S_1(1)}\delta_{\omega_2(n)=y+S_1(1)}$ times the density in \eqref{denab}. Since we are multiplying by $\delta_{\omega_1(1)=0}$, we can change $\fa\big(\omega_1(2)\big)$ to $\fa\big(\omega_1(2)-\omega_1(1)\big)$ in \eqref{denab}. Then this precisely matches with the measure $\Pr_{\wsc}^{T;(0,\m{h});(a',b')}$ described in Definition \ref{def:rwrb}, as was to be shown.
 \end{proof}

\begin{remark}
While $\wsc$ is the relevant Radon--Nikodym derivative that arises in the analysis of $\hslg$ Gibbs measures, we believe all of the results below in this section are true and can be proven in a similar manner for a general Radon--Nikodym of the form
\begin{align*}
    \exp\bigg(-R_1\big(S_2(1)-S_1(2)\big)-\sum_{k=2}^{n-1} \left[R_1\big(S_2(k)-S_1(k+1)\big)+R_2\big(S_2(k)-S_1(k)\big)\right]\bigg),
\end{align*}
where $R_1$ and $R_2$ are functions satisfying  $R_i(x)/|x| \to \infty$, and $|x|R_i(-x) \to 0$ as $x\to \infty$.
\end{remark}

The goal of the rest of this section is to obtain some preliminary probability estimates for events involving softly non-intersecting random walks and bridges.  It is well known from \cite{igl} that if the starting points of two random walks are within an $O(1)$ window, then the probability of the two walks not intersecting up to time $n$ is of order $n^{-1/2}$. Since the Radon--Nikodym derivative $\wsc$ is an approximation of the indicator for non-intersection, we expect both the numerator and denominator of the fraction in \eqref{wscint} to be of order $n^{-1/2}$, in particular vanishing as $n\to\infty$. Thus, it is not straightforward to transfer probability estimates for random walks and bridges to their softly non-intersecting versions. To get around this difficulty, we instead provide estimates for $\sqrt{n} \cdot \Ex^{n;(\m{h}_1,\m{h}_2);\bullet}[\wsc\ind_{\m{A}}]$ for different events $\m{A}$ of interest in this subsection. Towards this end, we define the (weak) non-intersection event
	\begin{align}\label{def:nip}
		\ni_p\ll a,b\rr:=\{\se{1}{k}-\se{2}{k}\ge -p, \mbox{ for all } k\in \ll a,b\rr\}.
	\end{align}
	When $a=2, b=n-1$, we write $\ni_p:=\ni_p\ll 2,n-1\rr$. We write $\ni:=\ni_0$ for the true non-intersection event. We first recall a few estimates about non-intersecting probabilities of random walks and bridges from Appendix C of \bcd.
\begin{lemma}[Lemmas C.3, C.8, and C.9 in \bcd] \label{l:ni} There exist a constant $\Con>0$ such that
\begin{enumerate}[label=(\alph*),leftmargin=15pt]
    \item For all $(a_1,a_2)\in \R^2$, $\Pr^{n;(a_1,a_2);\m{free}}(\m{NI}) \le \frac{\Con\max\{a_1-a_2,1\}}{\sqrt{n}}$. If $a_1\ge a_2$, $\Pr^{n;(a_1,a_2);\m{free}}(\m{NI}) \ge \frac{1}{C\sqrt{n}}$.
    \item For all $(a_1,a_2), (b_1,b_2)\in \R^2$, $\Pr^{n;(a_1,a_2);(b_1,b_2)}(\ni_p) \le  e^{\Con p} \cdot   \Pr^{n;(a_1,a_2);(b_1,b_2)}(\ni)$.
    \item Fix $M>0$. There exists a constant $\Con'=\Con'(M)>0$ such that for all $(a_1,a_2), (b_1,b_2)\in \R^2$ satisfying $|a_i|\le {\sqrt{n}}(\log n)^{3/2}$ with $|a_1-a_2|\le (\log n)^{3/2}$ and $|b_i|\le M\sqrt{n}$ with $b_1\ge b_2$, we have
     \begin{align*}
        \Pr^{n;(a_1,a_2);(b_1,b_2)}(\ni) \le \tfrac{\Con'}{\sqrt{n}} \max\{a_1-a_2,1\} \cdot \max\left\{ \tfrac1{\sqrt{n}}|a_1-b_1|,2\right\}^{3/2}.
        \end{align*}
\end{enumerate}
\end{lemma}

We now begin with a lemma providing a generic upper bound on $\sqrt{n} \cdot \Ex^{n;(\m{h}_1,\m{h}_2);\bullet}[\wsc\ind_{\m{A}}]$.


\begin{lemma}\label{rlem} Fix any $\e, M>0$. Recall the collection $\m{IC}(M)$ from Definition \ref{def:ic}. Let $(b_1(n),b_2(n))$ be two sequences of terminal points satisfying $|b_1(n)|+|b_2(n)| \le M\sqrt{n}$. There exists $R=R(\e,M)>0$ such that for all events $\m{A}$, we have 
\begin{equation}
    \label{eq:l37}
    \begin{aligned}
		& \limsup_{n\to\infty} \sqrt{n} \cdot \Ex^{n;(\m{h}_1,\m{h}_2);\bullet}[\wsc \ind_{\m{A}}]  \le  \e \\ & \hspace{2cm}+  \limsup_{n\to\infty} \left[\sqrt{n} \cdot \sum_{p=0}^{\lfloor 2\log\log n\rfloor+1}e^{-e^{p-1}}\cdot \sup_{|a_i|\le R}\Pr^{n;(a_1,a_2);\bullet}\left({\m{A}}\cap \ni_p\right)\right]
	\end{aligned}
\end{equation}
where $\bullet \in \{(b_1(n),b_2(n)),\m{free}\}$.
\end{lemma}

\begin{proof} Observe that given any $q\in \Z_{\ge 1}$ we have 
	\begin{align*}
		\ind_{\ni_0}+\sum_{p=1}^{q}\ind_{\ni_{p}\cap \ni_{p-1}^c}+\ind_{\ni_q^c}=1.
	\end{align*}
	We have $\wsc\le 1$ and on $\ni_{p-1}^c$, we have $\wsc\le e^{-e^{p-1}}$. Thus taking $q=\lfloor 2\log\log n \rfloor +1$ we have    \begin{align}\label{wineq}
		\wsc \le \sum_{p=0}^{\lfloor 2\log\log n\rfloor+1} e^{-e^{p-1}}\ind_{\ni_p}+e^{-(\log n)^2}.
	\end{align}
	In the computation of the limit, the $e^{-(\log n)^2}$ term vanishes, leading to
	\begin{align}\label{eq:limsupineq}
		\limsup_{n\to\infty} \sqrt{n} \cdot \Ex^{n;(\m{h}_1,\m{h}_2);\bullet}[\wsc\ind_{\m{A}}] \le  \limsup_{n\to\infty} \left[\sqrt{n} \cdot \sum_{p=0}^{\lfloor 2\log\log n\rfloor+1}e^{-e^{p-1}}\cdot \Pr^{n;(\m{h}_1,\m{h}_2);\bullet}\left({\m{A}}\cap \ni_p\right)\right], 
	\end{align}
Let us write $\m{B}_R:=\big\{|S_1(1)|,|S_2(1)|\le R\big\}$. By a union bound we have
\begin{align*}
    \Pr^{n;(\m{h}_1,\m{h}_2);\bullet}\left({\m{A}}\cap \ni_p\right) 
 & \le \Pr^{n;(\m{h}_1,\m{h}_2);\bullet}\left({\m{A}}\cap \m{B}_R\cap \ni_p\right)+\Pr^{n;(\m{h}_1,\m{h}_2);\bullet}\left(\neg\m{B}_R \cap \ni_p\right) \\ & \le \sup_{|a_i|\leq R}\Pr^{n;(a_1,a_2);\bullet}\left({\m{A}}\cap \ni_p\right)+\Pr^{n;(\m{h}_1,\m{h}_2);\bullet}\left(\neg\m{B}_R \cap \ni_p\right)
\end{align*}
Plugging this estimate back into \eqref{eq:limsupineq}, we see that to arrive at \eqref{eq:l37} it suffices to show 
{\begin{align}
    \label{bd23}
    \limsup_{n\to\infty} \sqrt{n}\sum_{p=0}^{\lfloor 2\log\log n\rfloor+1}e^{-e^{p-1}}\cdot \Pr^{n;(\m{h}_1,\m{h}_2);\bullet}\left(\neg\m{B}_R \cap \ni_p\right)
\end{align}}
can be made arbitrarily small by taking $R$ large enough. Towards this end, note that from Lemma \ref{l:tip} we have for $i\in\{1,2\}$ that
\begin{equation}
    \label{onetg}
    \begin{aligned}
     \Pr^{n;(\m{h}_1,\m{h}_2);\bullet}\big(\big\{|S_i(1)|\ge (\log n)^{3/2}\big\} \cap \ni_p\big) 
  & \le \Pr^{n;(\m{h}_1,\m{h}_2);\bullet}\big(|S_i(1)|\ge (\log n)^{3/2}\big)  \\ & \le \Con e^{-(\log n)^{3/2}/\Con}. 
 \end{aligned}
\end{equation}
On the other hand for deterministic $a_i$ satisfying $|a_i|\le (\log n)^{3/2}$, from the non-intersection probability estimates in Lemma \ref{l:ni} we have
 \begin{align*}
   &  \Pr^{n;(a_1,a_2);(b_1,b_2)}(\ni_p) \le \tfrac{1}{\sqrt{n}}\cdot \Con e^{\Con p}(M+1)^{3/2}\max\{a_1-a_2,1\}, \\ &
     \Pr^{n;(a_1,a_2);\m{free}}(\ni_p) \le \tfrac{1}{\sqrt{n}}\cdot \Con \max\{a_1-a_2+p,1\},
 \end{align*}
This implies
\begin{align*}
    & \Pr^{n;(\m{h}_1,\m{h}_2);\bullet}\left(\big\{|S_1(1)|,|S_2(1)|\le (\log n)^{3/2}, |S_i(1)|\ge R\big\} \cap \ni_p\right) \\ & \le \tfrac{1}{\sqrt{n}}\cdot \Con e^{\Con p}(M+1)^{3/2} {Q_i}, 
\end{align*}
where
{\begin{align*}
    Q_i:=\Ex^{n;(\m{h}_1,\m{h}_2);\bullet}\left[\ind_{|S_i(1)|>R}\max\{S_1(1)-S_2(2),1\}\right].
\end{align*}}
{Combining the above bound with \eqref{onetg} leads to
\begin{align*}
    & \sqrt{n}\sum_{p=0}^{\lfloor 2\log\log n\rfloor+1}e^{-e^{p-1}}\cdot \Pr^{n;(\m{h}_1,\m{h}_2);\bullet}\left(\neg\m{B}_R \cap \ni_p\right) \\ & \hspace{0cm}\le \Con\left[ \sqrt{n}e^{-(\log n)^{3/2}/\Con}+(M+1)^{3/2}(Q_1+Q_2) \right]\sum_{p=0}^{\infty} e^{-e^{p-1}+\Con p} .
\end{align*}}
Thanks to Lemma \ref{l:tip}, $Q_1, Q_2$ can be made arbitrary small by taking $R$ large enough. This implies that \eqref{bd23} can be made arbitrarily small by taking $R$ large enough, completing the proof.
\end{proof}

Using the above lemma, we shall now produce refined versions of \eqref{eq:l37} with additional assumptions on the events involved. For the rest of this subsection we fix $M>0$ and assume $\m{h}_1,\m{h}_2 \in \m{IC}(M)$. We assume $(b_1(n),b_2(n))$ are two sequences of terminal points satisfying $b_1(n)-b_2(n) \ge \frac{1}{M}\sqrt{n}$ and $|b_i(n)|\le M\sqrt{n}$ for $i=1,2$. We shall write $b_i=b_i(n)$, suppressing the dependency on $n$.

	\begin{lemma}\label{smallupper} Fix any $r>0$. Suppose $\m{A}\in\sigma\{\se{1}{k},\se{2}{k} : k\in \ll1,r\rr\}$. For every $\e>0$ there exists a constant $\Con=\Con (r,\e,M)>0$ such that
		\begin{align}\label{eq:l310}
			\limsup_{n\to\infty} \sqrt{n} \cdot \Ex^{n;(\m{h}_1,\m{h}_2);\bullet}[\wsc\ind_{\m{A}}] \le  \e+ \Con\cdot \sqrt{\limsup_{n\to\infty} \Pr^{n;(\m{h}_1,\m{h}_2);\m{free}}(\m{A})}
		\end{align}
	where  $\bullet\in\{(b_1,b_2),\operatorname{free}\}$.
	\end{lemma}
\begin{proof}
{Recall $R$ from Lemma \ref{rlem}.} We shall provide a bound on $ \Pr^{n;(a_1,a_2);\bullet}(\m{A}\cap \ni_p)$ that is uniform over all $|a_i|\le R+M$. Note that applying Lemma \ref{lem:mrb}, we may find a constant $\Con$ depending on $R, M$ such that 
\begin{align} \label{eqty}
		\Pr^{n;(a_1,a_2);\bullet}\left({\m{A}}\cap \ni_p\right)  \le \Con \cdot \Pr^{n;(a_1,a_2),\m{free}}\left({\m{A}}\cap \ni_p\ll r+1, n/4 \rr\right).
	\end{align}
 By the tower property of conditional expectation and the estimates on non-intersection probabilities from Lemma \ref{l:ni} we have
	\begin{align*}
		\mbox{r.h.s.~of \eqref{eqty}} &  = \Con \cdot \Ex^{n;(a_1,a_2);\m{free}}\left[\ind_{{\m{A}}}\cdot \Ex^{n;(a_1,a_2);\m{free}}\left[\ind_{ \ni_p\ll r+1, n/4 \rr} \mid \sigma\{(\se{1}{k},\se{2}{k})_{k=1}^r\}\right]\right]  \\ & \le \tfrac{\Con}{\sqrt{n}}\cdot \Ex^{n;(a_1,a_2);\m{free}}\left[\ind_{\m{A}}\cdot (|\se{1}{r}-\se{2}{r}|+p)\right]  \\ & \le \tfrac{\Con}{\sqrt{n}}\cdot \sqrt{\Pr^{n;(a_1,a_2);\m{free}}(\m{A})}\cdot \sqrt{\Ex^{n;(a_1,a_2);\m{free}}\left[(|\se{1}{r}-\se{2}{r}|+p)^2\right]} \\ & \le \tfrac{\Con p}{\sqrt{n}}\cdot \sqrt{\Pr^{n;(a_1,a_2);\m{free}}(\m{A})}.
	\end{align*}
	In the second line we used translation invariance to shift $S_1$ vertically by $p$. In the third line we applied the Cauchy-Schwarz inequality, and the last inequality follows by noting that $|\se{1}{r}-\se{2}{r}|$ is exponentially tight for fixed $r$. Inserting this bound back in \eqref{eq:l37} we arrive at \eqref{eq:l310}. This completes the proof.
\end{proof}
\begin{lemma}\label{spread} Fix any $\rho>0$. Given any $\e>0$ there exists $\delta(\e,M)>0$ such that
	\begin{align}\label{eq:l311}
		\limsup_{n\to\infty} \sqrt{n} \cdot \Ex^{n;(\m{h}_1,\m{h}_2);\m{free}}\big[\wsc\big[\ind_{\m{A}_{0}(\delta)}+\ind_{\m{A}_{1}(\delta)}+\ind_{\m{A}_{2}(\delta)}\big]\big] \le  \e,
	\end{align}
 where $\m{A}_0(\delta):=\{ \se{1}{ \rho n}-\se{2}{ \rho n}\le \delta\sqrt{\rho n}\}$ and $\m{A}_i(\delta):=\{|\se{1}{ \rho n}|\ge \delta^{-1}\sqrt{\rho n}\}$ for $i=\{1,2\}$.
\end{lemma}
\begin{proof}  {We claim
\begin{align}
    \label{a1bd}
    \limsup_{n\to\infty} \sqrt{n} \cdot \Ex^{n;(\m{h}_1,\m{h}_2);\m{free}}\big[\wsc\ind_{\m{A}_{i}(\delta)}\big] \le  \e/3,
\end{align}} for $i=0,1,2$.  We omit floor functions for brevity. {We recall $R$ from Lemma \ref{rlem}. We shall first prove this lemma for deterministic initial data $(a_1,a_2)$ with $|a_i|\le R$ and then appeal to Lemma \ref{rlem}. Towards this end, we observe that} for  any $p\in [0,2\log\log n+1]$, by lifting the $S_1(\cdot)$ random walk by $p$ units we have
\begin{align}
  \nonumber  & \Pr^{n;(a_1,a_2);\m{free}}\big(\{|S_1(\rho n)| \ge \delta^{-1}\sqrt{\rho n}\} \cap \ni_p\big) \\ \nonumber &\qquad =\Pr^{n;(a_1+p,a_2);\m{free}}\big(\{|S_1(\rho n)-p| \ge \delta^{-1}\sqrt{\rho n}\} \cap \ni\big) \\ &\qquad =\Pr^{n;(a_1+p,a_2);\m{free}}(\ni) \cdot \Pr^{n;(a_1+p,a_2);\m{free}}\left(|S_1(\rho n)-p| \ge \delta^{-1}\sqrt{\rho n} \mid \ni\right) \label{eet}
\end{align}
From Lemma \ref{l:ni}, $\Pr^{n;(a_1+p,a_2);\m{free}}(\ni)\le \frac{\Con (p+R+1)}{\sqrt{n}}$. On the other hand, it is known that the non-intersecting random walks are tight under diffusive scaling, see Lemma C.12 in \bcd. Utilizing this fact, we see that  $\Pr^{n;(a_1+p,a_2);\m{free}}(|S_1(\rho n)-p| \ge \delta^{-1}\sqrt{\rho n} \mid \ni)$ can be made arbitrarily small uniformly over all $|a_i|\le R$ and $p\in [0,2\log\log n+1]$. Plugging these two estimates back in \eqref{eet}, we get an upper bound for $\Pr^{n;(a_1,a_2);\m{free}}\big(\{|S_1(\rho n)| \ge \delta^{-1}\sqrt{\rho n}\} \cap \ni_p\big)$, which upon inserting in \eqref{eq:l37}, leads to \eqref{a1bd} for $i=1$. {The argument for the event $\m{A}_2(\delta)$ is analogous. For $\mathsf{A}_0(\delta)$, the same trick as above leads to
\begin{align*}
  \nonumber  & \Pr^{n;(a_1,a_2);\m{free}}\big(\{\se{1}{ \rho n}-\se{2}{ \rho n}\le \delta\sqrt{\rho n}\} \cap \ni_p\big) \\ &\qquad \le C \cdot \frac{p+R+1}{\sqrt{n}}\cdot \Pr^{n;(a_1+p,a_2);\m{free}}\left(\se{1}{ \rho n}-\se{2}{ \rho n}\le \delta\sqrt{\rho n}+p \mid \ni\right)
\end{align*}
Thanks to Lemma C.2 in \bcd ~(Eq.~(C.1) specifically), the above probability can be made arbitrarily small uniformly over all $|a_i|\le R$ and $p\in [0,2\log\log n+1]$. Thus, inserting the above estimate into \eqref{eq:l37} leads to \eqref{a1bd} for $i=0$.}
\end{proof}
\begin{lemma}(Tightness near edges)\label{edgetight} Fix {$\rho\in (0,1/4)$}. Consider the event 
	\begin{align*}
		\m{A}_n(\gamma):=\left\{\sup_{i=1,2}\sup_{ k\in \ll1,\rho n\rr} |\se{i}{1}-\se{i}{k}| \ge \gamma\sqrt{\rho n}\right\}.
	\end{align*} 
	Given any $\e>0$ there exists $\gamma(\e,M)>0$ such that
	\begin{align*}
		\limsup_{n\to\infty} \sqrt{n} \cdot \Ex^{n;(\m{h}_1,\m{h}_2);\bullet}[\wsc\ind_{\m{A}_n(\gamma)}] \le  \e,
	\end{align*}
where  $\bullet\in\{(b_1,b_2),\operatorname{free}\}$.
\end{lemma}
\begin{proof} Fix any $|a_i|\le R$ and $p\in [0,2\log\log n+1]$. By Lemma \ref{lem:mrb} we have
\begin{align*}
    \Pr^{n;(a_1,a_2);\bullet}(\m{A}_n(\gamma)\cap \ni_p) & \le \Pr^{n;(a_1,a_2);\bullet}(\m{A}_n(\gamma)\cap \ni_p\ll 2,n/4\rr) \\ & \le \Con \cdot \Pr^{n;(a_1,a_2);\m{free}}(\m{A}_n(\gamma)\cap \ni_p\ll 2,n/4\rr) \\ & \le \Con \cdot \Pr^{n;(a_1+p,a_2);\m{free}}(\ni_0\ll 2,n/4\rr)\cdot \Pr^{n;(a_1+p,a_2);\m{free}}(\m{A}_n(\gamma)\mid \ni_0\ll 2,n/4\rr) 
\end{align*}
By Lemma \ref{l:ni} we have $\Pr^{n;(a_1+p,a_2);\m{free}}(\ni_0\ll 2,n/4\rr) \le \Con(p+R+1)/\sqrt{n}$. On the other hand, the second term can be viewed under the random walk law of length $\rho n$. Indeed, we have
$$\Pr^{n;(a_1+p,a_2);\m{free}}(\m{A}_n(\gamma)\mid \ni_0\ll 2,n/4\rr)=\Pr^{\rho n;(a_1+p,a_2);\m{free}}(\m{A}_n(\gamma)\mid \ni_0\ll 2,\rho n\rr).$$
Now, it is known that that $S_i(\cdot)$ are tight under diffusive scaling. Thus the probability on the right can be made arbitrary small by choosing $\gamma$ large enough. Note that here the choice of $\gamma$ does not depend on $\rho$ as we scale by $\sqrt{\rho n}$ since the walk length has been reduced to $\rho n$. Combining these estimates leads to an upper bound for $\Pr^{n;(a_1,a_2);\bullet}(\m{A}_n(\gamma)\cap \ni_p)$, which in view of \eqref{eq:l37} leads to the desired result.
\end{proof}
All of the above results provide estimates for the numerator on the r.h.s.~of \eqref{wscint} for certain types of events. We now record a lower bound for the denominator in \eqref{wscint}.
\begin{lemma}\label{EWlbd} There exists a constant $\Con=\Con(M)>0$ such that
\begin{align*}
	\liminf_{n\to\infty} \sqrt{n} \cdot \Ex^{n;(\m{h}_1,\m{h}_2);\bullet}[\wsc] \ge  \tfrac1{\Con},
\end{align*}
where  $\bullet\in\{(b_1,b_2),\operatorname{free}\}$.
\end{lemma}

The proof of this lemma is very similar to that of Corollary 4.12 in \bcd, so we omit the details. 

All of the above results will be useful in concluding that certain events have high probability under $\Pr_{\wsc}^{n;(\m{h}_1,\m{h}_2);\bullet}$. We end this section with two weak convergence results for this law under diffusive scaling. For $(S_1,S_2)\sim \Pr^{n;(\m{h}_1,\m{h}_2);\bullet}_{\wsc}$ we will write
\begin{equation}\label{Uk}
U_k := \frac{S_1(k)-S_2(k)}{\sqrt{n}}, \qquad 1\leq k\leq n,
\end{equation}
and we extend $U_\cdot$ to non-integer arguments by linear interpolation.

\begin{proposition}\label{invarprin1}
Suppose $(x_n,y_n)$ are sequences such that $n^{-1/2}(y_n-x_n) \to z > 0$ as $n\to\infty$. Then the law of $(U_{nt})_{t\in[0,1]}$ under $\Pr^{n;(x_n,y_n);\m{free}}_{\wsc}$ converges weakly as $n\to\infty$ to a Brownian motion $(B_t)_{t\in[0,1]}$ with $B_0 = z$ and variance $\sigma^2 = \int_{\R^2} x^2\fa(y)\fa(x-y)dxdy$, conditioned to remain positive on $[0,1]$.
\end{proposition}

\begin{proof}
    The proof of this lemma is similar to weak convergence results for non-intersecting random walks under diffusive scaling, e.g., \cite[Lemma 3.10]{serio}. The main difference here is the soft non-intersection condition. First note that under $\Pr^{n;(x_n,y_n);\m{free}}$, by the invariance principle $(U_{nt})_{t\in[0,1]}$ converges weakly to a Brownian motion $(B_t)_{t\in[0,1]}$ with $B_0 = z$ and variance $\sigma^2$. By the Skorohod representation theorem (since $C[0,1]$ with the uniform topology is a Polish space), we may pass to another probability space with measure $\Pr$ supporting random variables with the same laws as $(U_k)_{1\leq k\leq n}$ and $(B_t)_{t\in[0,1]}$ (for which we use the same notation for brevity), such that $U_{nt} \to B_t$ uniformly in $t\in[0,1]$, $\Pr$-a.s. Now in view of \eqref{wscint}, if $F$ is any bounded continuous functional on $C[0,1]$, then writing $F(U) := F((U_{nt})_{t\in[0,1]})$, we have
    \begin{equation}\label{URN}
    \Ex^{n;(x_n,y_n);\m{free}}_{\wsc} [ F(U)] = \frac{\Ex[F(U)\wsc]}{\Ex[\wsc]},
    \end{equation}
    where we recall the definition of $\wsc$ in \eqref{defw}. Now fix any $\delta>0$. The uniform convergence $U_{nt}\to B_t$ implies that $\Pr$-a.s.,
\begin{equation}\label{Wnonint}
\mathbf{1}\{ B_t > \delta, \, t\in [0,1]\} \leq \liminf_{n\to\infty} \wsc \leq \limsup_{n\to\infty} \wsc \leq \mathbf{1}\{ B_t > -\delta, \, t\in[0,1]\}.
\end{equation}
As $\delta\downarrow 0$, the indicators on the left and right both converge to $\mathbf{1}_{\{B_t>0, \, t\in[0,1]\}}$ and $\mathbf{1}_{\{B_t \geq 0, \, t\in[0,1]\}}$ respectively by continuity of the Brownian motion. These two indicators are equal $\Pr$-a.s., see e.g. \cite[Corollary 2.9]{CH14}. Since \eqref{Wnonint} holds for any $\delta$, it follows that $\Pr$-a.s.,
\begin{equation}\label{Wlim}
\lim_{n\to\infty} \wsc = \mathbf{1}\{B_t > 0 \mbox{ for all } t\in[0,1]\}.
\end{equation}
Using \eqref{Wlim} and the fact that $U_{nt}\to B_t$ uniformly in \eqref{URN}, the dominated convergence theorem implies that
\[
\Ex^{n;(x_n,y_n);\m{free}}_{\wsc}[F(U)] \longrightarrow \frac{\Ex[F(B)\mathbf{1}\{B_t>0 \mbox{ for all } t\in[0,1]\}]}{\Pr(B_t>0 \mbox{ for all } t\in[0,1])}.
\]
This proves the desired weak convergence.
\end{proof}

\begin{proposition}\label{meander} Under $\Pr_{\wsc}^{n;(\m{h}_1,\m{h}_2);\m{free}}$, 
$U_n$ converges weakly as $n\to\infty$ to the endpoint $B_1^{\m{me}}$ of the Brownian meander $(B_t^{\m{me}})_{t\in[0,1]}$ with variance $\sigma^2 = \int_{\R^2} x^2\fa(y)\fa(x-y)dxdy$. In particular, $B_1^{\m{me}}>0$ almost surely.
\end{proposition}

\begin{proof} Fix any $x>0$. Fix $\e>0$ and for $\rho,\delta>0$, define the event $\m{A}_\delta^\rho = \{U_{\rho n} \in [\delta\sqrt{\rho},\delta^{-1}\sqrt{\rho}]\}$.
Note that by Lemmas \ref{spread} and \ref{EWlbd}, we may choose $\delta = \delta(\e,M)>0$ independent of $\rho$ so that
\begin{equation}\label{Adp}
\limsup_{n\to\infty} \Pr^{n;(\m{h}_1,\m{h}_2);\m{free}}_{W_n}(\neg\m{A}_\delta^\rho) < \e.
\end{equation}
In the following we omit the $\m{free}$ superscript for brevity. Let $\mathcal{F}_{\rho n} = \sigma \{(S_1(k), S_2(k))_{k\leq \rho n}\}$. Using the identities \eqref{decomp} and \eqref{transfin} for the measure $\Pr^{n;(\m{h}_1,\m{h}_2)}$ along with the tower property of conditional expectation, noting that $\widehat{W}_{\rho n} \in \mathcal{F}_{\rho n}$, we write
    \begin{equation}\label{Unlbd}
    \begin{aligned}
        \Pr^{n;(\m{h}_1,\m{h}_2)}_{W_n} ( U_n \le x) & \geq  \frac{\Ex^{n;(\m{h}_1,\m{h}_2)}[W_n \mathbf{1}_{\m{A}_\delta^\rho} \cdot \mathbf{1}_{U_n\leq x}] }{\Ex^{n;(\m{h}_1,\m{h}_2)}[W_n]} \\
        &= \frac{\Ex^{n;(\m{h}_1,\m{h}_2)}\left[ \mathbf{1}_{\m{A}_\delta^\rho} \widehat{W}_{\rho n} \cdot \Ex^{n-\rho n+1;(S_1(\rho n),S_2(\rho n))}[W_{n-\rho n+1} \mathbf{1}_{U_{n-\rho n+1}\leq x} ] \right]}{\Ex^{n;(\m{h}_1,\m{h}_2)}[W_n]} \\
        &= \Ex^{n;(\m{h}_1,\m{h}_2)}\bigg[ \mathbf{1}_{\m{A}_\delta^\rho} \cdot \frac{\widehat{W}_{\rho n}\Ex^{n-\rho n+1;(S_1(\rho n),S_2(\rho n)}[W_{n-\rho n+1}]}{\Ex^{n;(\m{h}_1,\m{h}_2)}[W_n]} \\ & \hspace{3cm} \cdot \Pr^{n-\rho n+1;(S_1(\rho n),S_2(\rho n))}_{W_{n-\rho n+1}}(U_{n-\rho n+1}\leq x ) \bigg],
    \end{aligned}
    \end{equation}
where the last equality follows from the definition of softly non-intersecting random walk bridges (see \eqref{wscint}).

{ Now consider the probability inside the expectation in the last line, which we will write for the moment as a function $P(S_1(\rho n), S_2(\rho n))$ of the boundary conditions. Shifting both walks vertically by $S_2(\rho n)$, since $U_{n-\rho n+1}$ depends only on the difference of the two walks, we see that $P(S_1(\rho n), S_2(\rho n)) = P(S_1(\rho n) - S_2(\rho n), 0)$. Observe that $P(\cdot, 0)$ is a continuous function of the first boundary condition, as it can be written as an integral against the continuous density $\m{f}$ using \eqref{freedensity} and \eqref{bridgedensity}. Therefore for each $n$, $P(\cdot, 0)$ attains a minimum on the compact set $\sqrt{n}[\delta\sqrt{\rho},\delta^{-1}\sqrt{\rho}]$; let $u_n$ be a point where the minimum is attained. On the event $\mathsf{A}_\delta^\rho$, $S_1(\rho n)-S_2(\rho n)$ lies in this compact set, so $P(S_1(\rho n)-S_2(\rho n),0) \geq P(u_n, 0)$. It follows from \eqref{Unlbd} that 
\begin{equation}\label{Unlbd2}
\begin{aligned}
\Pr^{n;(\m{h}_1,\m{h}_2)}_{W_n} ( U_n \le x) & \geq  \Ex^{n;(\m{h}_1,\m{h}_2)}\bigg[ \mathbf{1}_{\m{A}_\delta^\rho} \cdot \frac{\widehat{W}_{\rho n}\Ex^{n-\rho n+1;(S_1(\rho n),S_2(\rho n)}[W_{n-\rho n+1}]}{\Ex^{n;(\m{h}_1,\m{h}_2)}[W_n]} \\ & \hspace{4.5cm} \cdot \Pr^{n-\rho n+1;(u_n, 0)}_{W_{n-\rho n+1}}(U_{n-\rho n+1}\leq x ) \bigg].
\end{aligned}
\end{equation}
Since $n^{-1/2} u_n \in [\delta\sqrt{\rho}, \delta^{-1}\sqrt{\rho}]$ for all $n$, by passing to a subsequence we may assume that $n^{-1/2} u_n \to U \in [\delta \sqrt{\rho}, \delta^{-1}\sqrt{\rho}]$. }

For $z>0$, let us write $\Pr^+_z$ for the law of a Brownian motion $(B_t)_{t\in[0,1-\rho]}$ with $B_0 = z$ and variance $\sigma^2$ conditioned to remain positive on $[0,1-\rho]$. It then follows from Proposition \ref{invarprin1} that
\begin{equation}\label{Uinvar}
    \Pr^{n-\rho n+1; (u_n, 0)}_{W_{n-\rho n+1}}(U_{n-\rho n+1}\leq x) \longrightarrow \Pr^+_U(B_{1-\rho} \leq x).
\end{equation}
Since $\{B_{1-\rho}\leq x\}$ is an increasing event and $U\in[\delta\sqrt{\rho},\delta^{-1}\sqrt{\rho}]$, we observe by stochastic monotonicity for Brownian bridges, e.g., \cite[Lemma 2.7]{CH14}, that
\begin{equation}\label{Bmonot}
    \Pr^+_{\delta^{-1} \sqrt{\rho}}(B_{1-\rho} \leq x) \leq \Pr^+_U(B_{1-\rho} \leq x) \leq \Pr^+_{\delta\sqrt{\rho}}(B_{1-\rho} \leq x) .
\end{equation}
By \eqref{Uinvar} and \eqref{Bmonot}, we may choose $n_0(\e,\delta,\rho)$ so that for all $n\geq n_0$,
\begin{equation*}
    \Pr^{n-\rho n+1; (u_n, 0)}_{W_{n-\rho n+1}}(U_{n-\rho n+1}\leq x) \geq \Pr^+_{\delta^{-1}\sqrt{\rho}}(B_{1-\rho} \leq x) - \e .
\end{equation*}
Inserting this bound into \eqref{Unlbd2} and using \eqref{Adp} implies that
\begin{equation}\label{Uliminf}
\begin{split}
    &\liminf_{n\to\infty} \Pr^{n;(\m{h}_1,\m{h}_2)}_{\wsc}(U_n\leq x)\\ &\qquad \geq \frac{\Ex^{n;(\m{h}_1,\m{h}_2)}\left[ \mathbf{1}_{\m{A}_\delta^\rho} \widehat{W}_{\rho n}\Ex^{n-\rho n+1;(S_1(\rho n),S_2(\rho n)}[W_{n-\rho n+1}]  \left(\Pr_{\delta^{-1}\sqrt{\rho}}^+(B_{1-\rho}\leq x) - \e \right) \right]}{\Ex^{n;(\m{h}_1,\m{h}_2)}[W_n]} \\
    &\qquad \geq \left(\Pr_{\delta^{-1}\sqrt{\rho}}^+(B_{1-\rho}\leq x) - \e\right) \cdot \Pr^{n;(\m{h}_1,\m{h}_2)}_{W_n} (\m{A}_\delta^\rho)\\
    &\qquad \geq (1-\e)\left(\Pr_{\delta^{-1}\sqrt{\rho}}^+(B_{1-\rho}\leq x) - \e\right).
\end{split}
\end{equation}
In the second line we used the Gibbs property and the tower property again, essentially reversing the steps in \eqref{Unlbd}, and in the last line, we used \eqref{Adp}. A very similar argument leads to an upper bound of
\begin{equation}\label{Ulimsup}
\limsup_{n\to\infty} \Pr^{n;(\m{h}_1,\m{h}_2);\m{free}}_{\wsc}(U_n\leq x) \leq \Pr_{\delta\sqrt{\rho}}^+(B_{1-\rho}\leq x) + 2\e.
\end{equation}
Now for any fixed $\delta>0$, it is known that as $\rho\downarrow 0$ the two measures $\Pr^+_{\delta\sqrt{\rho}}$ and $\Pr^+_{\delta^{-1}\sqrt{\rho}}$ both converge weakly to the unique law $\Pr^+_0$ of the Brownian meander $B^{\m{me}}$ on $[0,1]$, see \cite[Section 2]{DurrBM}. Sending $\rho\downarrow 0$ in \eqref{Uliminf} and \eqref{Ulimsup}, we thus obtain
\begin{align*}
(1-\e)\left(\Pr^+_0\big(B_1^{\m{me}} \leq x\big) - \e\right) &\leq \liminf_{n\to\infty} \Pr^{n;(\m{h}_1,\m{h}_2);\m{free}}_{\wsc}(U_n\leq x)\\ &\leq \limsup_{n\to\infty} \Pr^{n;(\m{h}_1,\m{h}_2);\m{free}}_{\wsc}(U_n\leq x) \leq \Pr_0^+\big(B_1^{\m{me}}\leq x\big) + 2\e. 
\end{align*}
Since $\e$ was arbitrary, we in fact have
\[
\lim_{n\to\infty} \Pr^{n;(\m{h}_1,\m{h}_2);\m{free}}_{\wsc}(U_n\leq x) = \Pr_0^+\big(B_1^{\m{me}} \leq x\big),
\]
which proves the weak convergence. 
\end{proof}

\section{Local convergence for softly non-intersecting random walk bridges} \label{sec32}
In this section, we prove the following local convergence result (around the left endpoint) for softly non-intersecting random walk bridges. 

\begin{theorem} \label{thm:conv} Fix any $M,\e>0$ and $r\in \mathbb{Z}_{\geq 1}$. Recall $\ga$ from \eqref{defga}. Consider the law $\Pr_{W_n}^{n;(0,\ga);\bullet}$ from \eqref{wscint} with increments distributed as $\m{f}=\fa$ of Definition \ref{prb}. Fix any Borel set $A\subset \mathbb{R}^r \times \mathbb{R}^r$ and consider the event $\m{A}:=\{(S_1(k),S_2(k))_{k=1}^r \in A\}$. Let $B_M = \{(b_1,b_2)\in\mathbb{R}^2 : |b_1|+|b_2| \leq M\sqrt{n}, b_1-b_2 \geq \frac{1}{M}\sqrt{n}\}$.  Then there exists $n_0=n_0(M,\e)>r$ such that for all $n\ge n_0$ we have 
\begin{align}
    \sup_{(b_1,b_2)\in B_M}\bigg|\Pr_{W_n}^{n;(0,\ga);(b_1,b_2)}(\m{A}) -\Pr\big((S_1^{\uparrow}(k),S_2^{\uparrow}(k))_{k=1}^r \in A\big)\bigg| \le \e,
\end{align}
where the random variables $\big(S_1^{\uparrow}(k),S_2^{\uparrow}(k)\big)_{k\ge 1}$ are defined in Section \ref{sec:1.1}.
\end{theorem}

\subsection{The function $V$ and its properties}  Throughout this section, we fix $M>0$ and two densities $\m{h}_1,\m{h}_2\in \m{IC}(M)$ as in Definition \ref{def:ic}. The results of this section are only needed with $(\m{h}_1,\m{h}_2) = (0,\ga)$ where $\ga$ is defined in \eqref{defga}, but we work in this slightly more general setting as it requires no extra work. {For notational convenience we fix the increment density to be $\fa$, but the results below apply with a general density $\m{f}$ satisfying Assumption \ref{asp:in}.} 
 For each $(x,y)\in \R^2$ we define
\begin{equation}\label{hatV}
  V_{\m{h}_1,\m{h}_2}(x,y)  = \lim_{n\to\infty} \frac{\Ex^{n;(x,y);\m{free}}[W_n]}{\Ex^{n;(\m{h}_1,\m{h}_2);\m{free}}[W_n]}.
\end{equation}
Because of the form of $W_n$, it is easy to see that $V_{\m{h}_1,\m{h}_2}(x,y)=V_{\m{h}_1,\m{h}_2}(0,y-x)$. Note that when $\m{h}_1=\delta_{0}$ and $\m{h}_2=\ga$, $V_{0,\ga}(0,z)$ agrees with $V(z)$ defined in \eqref{def:pv}. More generally, if $X,Y$ are random variables with densities $\m{f}_1,\m{f}_2\in\m{IC}(M)$ respectively for some $M>0$, we define
 \begin{equation}\label{hatVrandom}
     V_{\m{h}_1,\m{h}_2}(X,Y) = \lim_{n\to\infty} \frac{\Ex^{n;(\m{f}_1,\m{f}_2);\m{free}}[W_n]}{\Ex^{n;(\m{h}_1,\m{h}_2);\m{free}}[W_n]}.
 \end{equation}The next lemma is the main technical result of this section and it justifies the existence of the above limits.
 
\begin{lemma}\label{V(x,y)}
	Fix any $r\in\mathbb{Z}_{\geq 1}$, $M>0$, and $x,y\in\mathbb{R}$.
 Suppose $u_1,v_1,u_2,v_2 \in \m{IC}(M)$ and $b_1,b_2$ are two sequences with $|b_i|\le M\sqrt{n}$ and $b_1-b_2 \ge \frac1{M}\sqrt{n}$. Suppose $m_1, m_2$ are two sequences satisfying $m_1/n\to 1$ and $m_2/n\to 1$ as $n\to \infty$. Then for all $\bullet\in \{(b_1,b_2),\m{free}\}$,
 \begin{equation}\label{Vb2}
	\lim_{n\to\infty} \frac{\Ex^{{m_1}; (u_1,v_1); \bullet} [W_{m_1}]}{\Ex^{m_2; (u_2, v_2); \bullet} [W_{m_2}]}
	\end{equation}
 exists and is independent of $\bullet$.
\end{lemma}

In plain words, the above lemma implies that the limit in \eqref{hatV} exists, and $V$ can be obtained from bridge measures in the same way as from walk measures.

\begin{remark}\label{vxyw} Taking $(u_1,v_1)=(x,y), (u_2, v_2)=(0,\m{g}_\alpha), m_1=m_2=n, \bullet=\m{free}$ in Lemma \ref{V(x,y)}, we see that the limit in \eqref{Vdef} exists. Below, we will consider initial conditions $(u_1,v_1) = (S_1(r),S_2(r))$, $(u_2,v_2)=(\m{h}_1,\m{h}_2)$, where $(S_1,S_2)\sim \Pr^{r;(\m{h}_1,\m{h}_2);\m{free}}$ and $r\geq 1$ is fixed. Since $S_1(r),S_2(r)$ are both sums of finitely many independent random variables with exponential tails, their densities lie in $\m{IC}(M)$ for sufficiently large $M$, and Lemma \ref{V(x,y)} then guarantees that the limit $V_{\m{h}_1,\m{h}_2}(S_1(r),S_2(r))$ in \eqref{hatVrandom} with $X=S_1(r), Y=S_2(r)$ exists.
    
\end{remark}

\begin{proof} Without loss of generality, we may pass to a subsequence and assume $b_1/\sqrt{n} \to \beta_1, b_2/\sqrt{n} \to \beta_2$ where $\beta_1,\beta_2 \in [-M,M]$ and $\beta_1-\beta_2 \in [M^{-1},M]$. 
Fix $\varepsilon \in (0,1/2)$. For clarity, we divide the proof into three steps.

\medskip

\noindent\textbf{Step 1.} The goal of this step is to provide an upper bound for the ratio
\begin{align}
     \frac{\Ex^{m_1; (u_1,v_1); \bullet } [W_{m_1}]}{\Ex^{m_2; (u_2,v_2); \bullet } [W_{m_2}]}.
\end{align}
Towards this end, we shall proceed by providing upper and lower bounds on $\Ex^{m;(u,v);\bullet}[W_m]$ for large $n$, assuming $m$ is of order $n$ and $u,v\in \m{IC}(M)$. 

For $\rho, \gamma>0$ (to be fixed later depending on $\e,M$) define the event
	\[
	\m{B}_\gamma = \{ \max(|S_1(\rho n)|, |S_2(\rho n)|) < \gamma\sqrt{\rho n} \} \cap \{\max(|S_1(1)|, |S_2(1)|) < \gamma\}.
	\]
As before, we have omitted floor signs for brevity and written $S_i(\rho n)$ for $S_i(\lfloor \rho n\rfloor)$. Let $\mathcal{F}_{\rho n}$ denote the $\sigma$-algebra generated by $(S_1(k),S_2(k))_{k=1}^{\rho n}$, and recall $\hat{W}_r$ from \eqref{wrdef}. We note that $\mathbf{1}_{\m{B}_\gamma}\hat{W}_{\rho n}$ is $\mathcal{F}_{\rho n}$-measurable. Using \eqref{decomp} and conditioning on $\mathcal{F}_{\rho n}$, we write
\begin{equation}
    \begin{split}
	\Ex^{m; (u,v); \bullet} [W_{m}] &\geq  \Ex^{m; (u,v); \bullet} \left[ \mathbf{1}_{\m{B}_\gamma} \hat{W}_{\rho n} \cdot \Ex^{m; (u,v); \bullet}\left[ {W}_{\rho n\to m} \mid \mathcal{F}_{\rho n} \right] \right] \\
	&\geq  (1-\e) \cdot \Ex^{ \rho n ; (u,v); \m{free}} \left[ \mathbf{1}_{\m{B}_\gamma} \hat{W}_{\rho n} \cdot \Ex^{m; (u,v);\bullet}\left[ {W}_{\rho n\to m} \mid \mathcal{F}_{\rho n} \right] \right].\label{Vnum}
    \end{split}
\end{equation}
The last inequality follows from Lemma \ref{bridgewalk} by taking $\rho$ small enough depending on $\e$.
Let us define
\begin{equation}\label{Zdef}
\mathsf Z_{m}^\bullet := \exp\left(-e^{S_2(\rho n)-S_1(\rho n)}\right)\cdot \Ex^{m;(S_1(\rho n),S_2(\rho n));\bullet}[W_m].    
\end{equation}
This allows us to write
\begin{equation}
    \label{3eqs}
    \begin{aligned}
 & \Ex^{ \rho n ; (u,v); \m{free}} \left[ \mathbf{1}_{\m{B}_\gamma} \hat{W}_{\rho n} \cdot \Ex^{m; (u,v);\bullet}\left[ {W}_{\rho n\to m} \mid \mathcal{F}_{\rho n} \right] \right] \\ &\qquad = \Ex^{\rho n;(u,v);\m{free}}\left[\mathbf{1}_{\m{B}_\gamma} \hat{W}_{\rho n} \cdot \Ex^{m-\rho n+1;(S_1(\rho n),S_2(\rho n));\bullet}[W_{m-\rho n+1}]\right] \\ &\qquad = \Ex^{\rho n; (u,v); \m{free}}[W_{\rho n}] \cdot \Ex^{\rho n; (u,v); \m{free}}_{W_{\rho n}} \left[ \mathbf{1}_{\m{B}_\gamma} \cdot \mathsf Z_{m-\rho n+1}^\bullet\right].
\end{aligned}
\end{equation}
The first equality follows from \eqref{transfin}, and the second is by definition of $\Ex^{\rho n; (x,y); \m{free}}_{W_{\rho n}}$ from \eqref{wscint}. Plugging the above expression back into \eqref{Vnum}, we get the following lower bound:
\begin{align}\label{uvlbd}
    \Ex^{m; (u,v); \bullet} [W_{m}] 
	&\geq  (1-\e) \cdot \Ex^{\rho n; (u,v); \m{free}}[W_{\rho n}] \cdot \Ex^{\rho n; (u,v); \m{free}}_{W_{\rho n}} \left[ \mathbf{1}_{\m{B}_\gamma} \cdot \mathsf Z_{m-\rho n+1}^\bullet\right].
\end{align}
On the other hand for the upper bound, we first note that using Lemmas \ref{spread} and \ref{EWlbd} one can take $\gamma = \gamma(\varepsilon, M)$ large enough such that for all large $m$,
\begin{align}\label{newg}
    \Ex^{m; (u,v); \bullet } [W_m\ind_{\neg \m{B}_{\gamma}}] \le \frac{\e}{\sqrt{m}} \le C \e \cdot \Ex^{m; (u,v); \bullet} [W_m] .
\end{align}
Taking $\rho$ small enough, in view of  Lemma \ref{bridgewalk}, 
\begin{align*}
    \Ex^{m; (u,v); \bullet } [W_m\ind_{\m{B}_{\gamma}}] &= \Ex^{\rho n ; (u,v); \bullet} \left[ \mathbf{1}_{\m{B}_\gamma} \hat{W}_{\rho n} \cdot \Ex^{m; (u,v); \bullet}\left[ {W}_{\rho n\to m} \mid \mathcal{F}_{\rho n} \right] \right]\\
    &\le (1+\e)\cdot \Ex^{\rho n ; (u,v); \m{free}} \left[ \mathbf{1}_{\m{B}_\gamma} \hat{W}_{\rho n} \cdot \Ex^{m; (u,v); \bullet}\left[ {W}_{\rho n\to m} \mid \mathcal{F}_{\rho n} \right] \right] \\ & = (1+\e)\cdot \Ex^{\rho n; (x,y); \m{free}}[W_{\rho n}] \cdot \Ex^{\rho n; (x,y); \m{free}}_{W_{\rho n}} \left[ \mathbf{1}_{\m{B}_\gamma} \cdot \mathsf Z_{m-\rho n+1}^{\bullet}\right],
\end{align*}
where the first line is due to \eqref{decomp} and the third line is due to \eqref{3eqs}.
Adding the above inequality with \eqref{newg} and rearranging the terms we get
\begin{align}\label{uvubd}
    \Ex^{m; (u,v); \bullet } [W_m] & \le \frac{1+\e}{1-C\e}\cdot \Ex^{\rho n; (u,v); \m{free}}[W_{\rho n}] \cdot \Ex^{\rho n; (u,v); \m{free}}_{W_{\rho n}} \left[ \mathbf{1}_{\m{B}_\gamma} \cdot \mathsf Z_{m-\rho n+1}^{\bullet}\right].
\end{align}
Combining the upper bound from above with the lower bound from \eqref{uvlbd}, we thus have
\begin{equation}
    \label{ratiobd}
    \begin{aligned}
    \frac{\Ex^{m_1; (u_1,v_1); \bullet } [W_{m_1}]}{\Ex^{m_2; (u_2,v_2); \bullet } [W_{m_2}]} & \le \frac{(1+\e)}{(1-\e)(1-C\e)}\cdot \frac{\Ex^{\rho n; (u_1,v_1); \m{free}}[W_{\rho n}]}{\Ex^{\rho n; (u_2,v_2); \m{free}}[W_{\rho n}]} \\ & \hspace{4cm} \cdot \frac{\Ex^{\rho n; (u_1,v_1); \m{free}}_{W_{\rho n}} \left[ \mathbf{1}_{\m{B}_\gamma} \cdot \mathsf Z_{m_1-\rho n+1}^\bullet\right]}{\Ex^{\rho n; (u_2,v_2); \m{free}}_{W_{\rho n}} \left[ \mathbf{1}_{\m{B}_\gamma} \cdot \mathsf Z_{m_2-\rho n+1}^\bullet\right]}.
\end{aligned}
\end{equation}

\medskip

\noindent\textbf{Step 2.} We claim that if $m_1/n \to 1$ and $m_2/n\to 1$ as $n\to\infty$, then for all large enough $\gamma$ and $n$ we have
\begin{align}\label{tratio}
 \frac{1-\e}{1+\e} \le \frac{\Ex^{\rho n; (u_1,v_1); \m{free}}_{W_{\rho n}} \left[ \mathbf{1}_{\m{B}_\gamma} \cdot \mathsf Z_{m_1-\rho n+1}^\bullet\right]}{\Ex^{\rho n; (u_2,v_2); \m{free}}_{W_{\rho n}} \left[ \mathbf{1}_{\m{B}_\gamma} \cdot \mathsf Z_{m_2-\rho n+1}^\bullet\right]} \le \frac{1+\e}{1-\e}.
\end{align}
We shall prove \eqref{tratio} in the next step. Let us complete the proof of the lemma assuming it. Plugging the above bound back into \eqref{ratiobd}, we arrive at 
\begin{equation}
\label{trag}
    \begin{aligned}
    \frac{\Ex^{m_1; (u_1,v_1); \bullet } [W_{m_1}]}{\Ex^{m_2; (u_2,v_2); \bullet } [W_{m_2}]} & \le \frac{(1+\e)^2}{(1-\e)^2(1-C\e)}\cdot \frac{\Ex^{\rho n; (u_1,v_1); \m{free}}[W_{\rho n}]}{\Ex^{\rho n; (u_2,v_2); \m{free}}[W_{\rho n}]}.
\end{aligned}
\end{equation}
for all large enough $n$. Let us take $m_1=m_2=n$ and $\bullet=\m{free}$. Defining $$a_n:=\frac{\Ex^{n; (u_1,v_1); \m{free} } [W_{n}]}{\Ex^{n; (u_2,v_2); \m{free} } [W_{n}]},$$ 
the above inequality then translates to $a_n\le \frac{(1+\e)^2}{(1-\e)^2(1-C\e)}a_{\rho n}$ for all large $n$. Also note that $(a_n)_{n\geq 1}$ is bounded by Lemmas \ref{smallupper} and \ref{EWlbd}. From here it follows that $(a_n)_{n\geq 1}$ converges by the following real analysis lemma.

\begin{lemma} Suppose $f : \mathbb{Z}_{\ge 1}\to \mathbb{Z}_{\ge 1}$ is a non-decreasing function with $f(n) \to \infty$ as $n\to \infty$. Let $(a_n)_{n\ge 1}$ be a bounded sequence such that for every $\e>0$, there exists $n_0(\e)$ such that for all $n\ge n_0$ and $m\ge f(n)$ we have $ a_{n+m} \le (1+\e) a_{n}$. Then $\lim_{n\to \infty} a_n$ exists.
\end{lemma}
\begin{proof} Let us consider two convergent subsequences $(a_{m_k})_{k\ge 1}$ and $(a_{n_k})_{k\ge 1}$ with limits $a^{(1)}$ and $a^{(2)}$ respectively. Note that we may find a subsequence of $n_k$, say $n_{k_{\ell}}$, such that
$n_{{k}_{\ell}} \ge m_{\ell}+f(m_\ell)$ for all $\ell \ge 1$. Then for all $\ell$ large enough we have $a_{n_{{k}_{\ell}}} \le (1+\e) a_{m_\ell}$.
Taking $\ell \to \infty$, we get $a^{(2)}\le (1+\e)a^{(1)}$. As $\e$ is arbitrary, we have $a^{(2)} \le a^{(1)}$. Interchanging the roles of the subsequences, we get the reverse inequality, so $a^{(1)}=a^{(2)}$. Hence all convergent subsequences have the same limit, and since $(a_n)_{n\geq 1}$ is bounded this implies convergence. 
\end{proof}	

Now since $\lim_{n\to\infty} a_n$ exists and equals $\lim_{n\to\infty} a_{\rho n}$ for any fixed $\rho>0$, we may take the limsup on both sides of \eqref{trag} to get
\begin{align}\label{ragn}
   \limsup_{n\to\infty} \frac{\Ex^{m_1; (u_1,v_1); \bullet } [W_{m_1}]}{\Ex^{m_2; (u_2,v_2); \bullet } [W_{m_2}]} & \le \frac{(1+\e)^2}{(1-\e)^2(1-C\e)}\cdot \lim_{n\to\infty} a_n.
\end{align}
We may use \eqref{uvlbd}, \eqref{uvubd}, and \eqref{tratio} to get the following lower bound for the ratio:
\begin{align*}
    \frac{\Ex^{m_1; (u_1,v_1); \bullet } [W_{m_1}]}{\Ex^{m_2; (u_2,v_2); \bullet } [W_{m_2}]} \ge \frac{(1-\e)^2(1-C\e)}{(1+\e)^2}\cdot\frac{\Ex^{\rho n; (u_1,v_1); \m{free}}[W_{\rho n}]}{\Ex^{\rho n; (u_2,v_2); \m{free}}[W_{\rho n}]}.
\end{align*}
Upon taking the liminf we obtain
\begin{align}\label{ragn2}
    \liminf_{n\to \infty}\frac{\Ex^{m_1; (u_1,v_1); \bullet } [W_{m_1}]}{\Ex^{m_2; (u_2,v_2); \bullet } [W_{m_2}]} \ge \frac{(1-\e)^2(1-C\e)}{(1+\e)^2}\cdot\lim_{n\to\infty} a_n.
\end{align}
As $\e$ is arbitrary, \eqref{ragn} and \eqref{ragn2} together prove the lemma.

\medskip

\noindent\textbf{Step 3.} In this step we prove \eqref{tratio}. Let us study the numerator of the ratio in \eqref{tratio}. For convenience, we write $\Pr_n=\Pr^{\rho n; (u_1,v_1); \m{free}}_{W_{\rho n}}$. We first investigate the weak limit of $\mathsf Z_{m-\rho n+1}^\bullet$. 
By Proposition \ref{meander}, under $\Pr_n$, $n^{-1/2}(S_1(\rho n)-S_2(\rho n))$ converges weakly to $\sqrt{\rho} R$ where $R$ is the endpoint of the Brownian meander. By the Skorohod representation theorem, we may pass to a new probability space with measure $\Pr$ supporting random variables with the same laws as $S_1(\rho n), S_2(\rho n), R$ (for which we use the same notation for brevity) such that $n^{-1/2}(S_1(n\rho)-S_2(n\rho)) \longrightarrow \sqrt{\rho}R$, $\mathbb{P}$-a.s. We then observe that for $(\widetilde{S}_1,\widetilde{S}_2) \sim \Pr^{m-\rho n+1;(S_1(\rho n),S_2(\rho n));\bullet}$, by the invariance principle for random walks/bridges $\{n^{-1/2}(\widetilde{S}_1(nt) - \widetilde{S}_2(nt))\}_{t\in[0,1-\rho]}$ converges in law  to 
\begin{itemize}
    \item a Brownian motion $(
B_t)_{t\in[0,1-\rho]}$ with $B_0 = \sqrt{\rho}R$ if $\bullet=\m{free}$;
\item a Brownian bridge  $(
B_t)_{t\in[0,1-\rho]}$  with $B_0 = \sqrt{\rho}R$ and $B_{1-\rho} = \beta_1-\beta_2$ if $\bullet=(b_1,b_2)$. {Recall that we assumed $b_1/\sqrt{n} \to \beta_1, b_2/\sqrt{n} \to \beta_2$ where $\beta_1,\beta_2 \in [-M,M]$ and $\beta_1-\beta_2 \in [M^{-1},M]$.}  
\end{itemize}
 We denote these limiting laws by $\mathbb{P}_{R}^\star$, where $\star = \m{free}$ if $\bullet = \m{free}$ and $\star = (\beta_1,\beta_2)$ if $\bullet = (b_1,b_2)$. Since $R>0$ a.s. and $\beta_1-\beta_2>0$, it follows by the same arguments as in Proposition \ref{invarprin1} (see \eqref{Wnonint} in particular) that $\Pr$-a.s.,
\begin{equation}
\begin{split}\label{Wtildeconv}
\mathsf Z_{m_1-\rho n+1}^\bullet &\longrightarrow \mathsf Z_{\infty}^\star := \mathbb{P}^\star_R \left ( B_t > 0 \mbox{ for all } t\in[0,1-\rho]\right).
\end{split}
\end{equation}
This random variable $\mathsf Z_\infty^{\star}$ is strictly positive almost surely for all choices of $\star$ in question. When $\bullet=(b_1,b_2)$, by monotonicity of Brownian bridges w.r.t.~endpoints, $\mathsf Z_{\infty}^{(\beta_1,\beta_2)}$ is stochastically larger than $\mathsf Z_{\infty}^{(M^{-1},0)}$ as $\beta_1-\beta_2\ge \frac1M$.
Thus $\Ex[\mathsf Z_{\infty}^\star]>\delta$ where $\delta>0$ is independent of $\star$. By Lemmas \ref{spread} and \ref{EWlbd}, $\Pr_n(\m{B}_\gamma)$ can be made arbitrarily close to $1$ by taking $\gamma,n$ large. As $\mathsf Z_{m-\rho n+1}^\bullet \le 1$, in view of \eqref{Wtildeconv} and dominated convergence we have
\begin{align*}
  (1-\e)\Ex[\mathsf Z_{\infty}^\star] \le \Ex_n[\mathbf{1}_{\m{B}_\gamma} \cdot \mathsf Z_{m_i-\rho n+1}^\bullet] \le \Ex_n[\mathsf Z_{m_i-\rho n+1}^\bullet] \le  (1+\e)\Ex[\mathsf Z_{\infty}^{\star}] 
\end{align*}
for all large $\gamma, n$. Using this bound, we arrive at \eqref{tratio}. \end{proof}

\begin{definition}\label{def:LGmeas} Using $V_{\m{h}_1,\m{h}_2}$ from \eqref{hatVrandom}, we define an induced measure ${\Pr}_{\m{LG}}^{(\m{h}_1, \m{h}_2)}$ on $\Omega_n$, for Borel sets $A\subset\mathbb{R}^r\times\mathbb{R}^r$ and corresponding events $\m{A} := \{(S_1(k),S_2(k))_{k=1}^r \in A\}$, by 
\begin{align}\label{LGmeas}
\Pr_{\m{LG}}^{(\m{h}_1,\m{h}_2)}(\m{A})=\Ex^{r; (\m{h}_1,\mathsf{h}_2); \m{free}}\left[\ind_{\m{A}} \cdot \hat{W}_r V_{\m{h}_1,\m{h}_2}(S_1(r),S_2(r))\right].
\end{align}
Here $V_{\m{h}_1,\m{h}_2}(S_1(r),S_2(r))$ is well-defined by Remark \ref{vxyw}.
\end{definition}

The fact that \eqref{LGmeas} defines a probability measure is encoded in the following lemma.

\begin{lemma}\label{expec1}
We have
\[
\Ex^{r;(\m{h}_1,\m{h}_2);\m{free}}\left[\hat{W}_r V_{\m{h}_1,\m{h}_2}(S_1(r),S_2(r))\right] = 1.
\]
\end{lemma}

\begin{proof} For clarity, we divide the proof into two steps.

\medskip

\noindent\textbf{Step 1.} We claim that
\begin{align}
    \label{ratio11}
    \lim_{n\to\infty}\frac{\Ex^{n+r;(\m{h}_1,\m{h}_2);\m{free}}[W_{n+r}]}{\Ex^{n;(\m{h}_1,\m{h}_2);\m{free}}[W_n]} = 1.
\end{align}
We postpone its proof to \textbf{Step 2}. Let us complete the proof assuming it.
Let $\mathcal{F}_r$ denote the $\sigma$-algebra generated by $(S_1(k),S_2(k))_{k=1}^{r}$. We introduce the ratio
\[
T_r^n = \frac{\Ex^{n+r;(\m{h}_1,\m{h}_2);\m{free}}[W_{n+r} \mid \mathcal{F}_{r+1}]}{\Ex^{n;(\m{h}_1,\m{h}_2);\m{free}}[W_n]}.
\]
Note that $W_{n+r}=\hat{W}_{r+1}\cdot W_{r+1\to n+r}$ where $W_{r+1\to n+r}$ is defined in \eqref{wsplit}. Thus,
\begin{align*}
T_r^n &= \hat{W}_{r+1}\cdot \frac{\Ex^{n+r;(\m{h}_1,\m{h}_2);\m{free}}\left[W_{r+1\to n+r}\mid \mathcal{F}_{r+1}\right]}{\Ex^{n;(\m{h}_1,\m{h}_2);\m{free}}[W_n]}
= \hat{W}_{r+1} \cdot \frac{\Ex^{n;(S_1(r+1),S_2(r+1));\m{free}}[W_n]}{\Ex^{n;(\m{h}_1,\m{h}_2);\m{free}}[W_n]}.
\end{align*}
The second equality above follows from the Gibbs property for random walks. So by definition of $V_{\m{h}_1,\m{h}_2}(x,y)$, we have
\[
\lim_{n\to\infty} T_r^n = \hat{W}_{r+1} V_{\m{h}_1,\m{h}_2}(S_1(r+1),S_2(r+1)).
\]
By \eqref{ratio11}, we have $\Ex^{n+r;(\m{h}_1,\m{h}_2);\m{free}} [T_r^n] \to 1$ as $n\to\infty$.  
We now claim that 
\begin{align}\label{clafi}
    T_r^n \le C[|S_1(r+1)-S_2(r+1)|+1],
\end{align}
for some constant $C>0$ dependent only on $M$. Note that under the law $\Pr^{n+r;(\m{h}_1,\m{h}_2);\m{free}}$, for $i=1,2$, $S_i(r+1)$ is just a sum of $r+1$ independent random variables each with exponential tails. This implies that $\Ex^{n+r;(\m{h}_1,\m{h}_2);\m{free}}|S_1(r+1)-S_2(r+1)|$ is finite. Then applying dominated convergence  with the claim in \eqref{clafi} gives the result. We thus focus on proving \eqref{clafi}. Towards this end, applying the inequality in \eqref{wineq} we see that
\begin{align*}
    \Ex^{n;(S_1(r+1),S_2(r+1));\m{free}}[W_n] \le e^{-(\log n)^2}+\sum_{p=0}^{\lfloor 2\log\log n\rfloor+1} e^{-e^{p-1}}\Pr^{n;(S_1(r+1),S_2(r+1));\m{free}}(\ni_p).
\end{align*}
Thanks to Lemma \ref{l:ni}, we can estimate the above non-intersection probability:
\begin{align*}
    \Pr^{n;(S_1(r+1),S_2(r+1));\m{free}}(\ni_p) = \Pr^{n;(S_1(r+1)+p,S_2(r+1));\m{free}}(\ni) \le \frac{|S_1(r+1)-S_2(r+1)|+p+1}{\sqrt{n}}.
\end{align*}
Using the trivial inequality $\hat{W}_r\le 1$ and the lower bound on $\Ex^{n;(\m{h}_1,\m{h}_2);\m{free}}[W_n]$ of the order $1/\sqrt{n}$ from Lemma \ref{EWlbd} we thus have
\begin{align*}
    T_r^n & \le \Con\left[\sqrt{n}e^{-(\log n)^2}+\sum_{p=0}^{\lfloor 2\log\log n\rfloor+1} e^{-e^{p-1}} [|S_1(r+1)-S_2(r+1)|+p+1]\right] 
\end{align*}
which is clearly less than $\Con[|S_1(r+1)-S_2(r+1)|+1]$, completing the proof of \eqref{clafi}.

\medskip

  \noindent\textbf{Step 2.} In this step, we prove \eqref{ratio11}.  First observe that we have a trivial lower bound $W_{n+r} \leq W_n$, so the limsup of the ratio is at most 1. For a lower bound, fix $\varepsilon>0$. For $\delta>0$, we introduce two events
    \[
    \m{C}_\delta = \left\{ \inf_{k\in\llbracket n,n+r\rrbracket} S_1(k) - \sup_{k\in\llbracket n,n+r\rrbracket} S_2(k)  > \tfrac{\delta}{4}\sqrt{n}\right\}, \quad \m{D}_\delta  = \left\{ S_1(n)-S_2(n) > \tfrac{\delta}{2}\sqrt{n}\right\}
    \]
    Then we have
    \begin{align*}
        \Ex^{n+r;(\m{h}_1,\m{h}_2);\m{free}} [ W_{n+r}] &\geq \Ex^{n+r;(\m{h}_1,\m{h}_2);\m{free}} [ W_{n+r} \mathbf{1}_{\m{C}_\delta}]\\
        &= \Ex^{n+r;(\m{h}_1,\m{h}_2);\m{free}} \left[ W_n \mathbf{1}_{\m{C}_\delta} \exp\left(-\sum_{k=n+1}^{n+r}(e^{S_2(k-1)-S_1(k)} + e^{S_2(k)-S_1(k)})\right)\right]\\
        &\geq \exp(-2r e^{-\delta\sqrt{n}/4}) \cdot \Ex^{n+r;(0,\mathsf{h});\m{free}}[W_n \mathbf{1}_{\m{C}_\delta}].
    \end{align*}
    If $n$ is large enough depending on $\delta,\varepsilon$, this implies
    \begin{align}\label{xfr0}
    \Ex^{n+r;(\m{h}_1,\m{h}_2);\m{free}} [ W_{n+r}] \geq (1-\varepsilon)\left\{ \Ex^{n;(\m{h}_1,\m{h}_2);\m{free}} [ W_{n}] - \Ex^{n+r;(\m{h}_1,\m{h}_2);\m{free}} [ W_{n}\mathbf{1}_{\neg \m{C}_\delta}]\right\}.
    \end{align}
    We now seek to lower bound the second term on the right, {i.e., $\Ex^{n+r;(\m{h}_1,\m{h}_2);\m{free}} [ W_{n}\mathbf{1}_{\neg \m{C}_\delta}]$. We write
    \begin{align}\label{xfr}
        \Ex^{n+r;(\m{h}_1,\m{h}_2);\m{free}} [ W_{n}\mathbf{1}_{\neg \m{C}_\delta}] &\leq \Ex^{n+r;(\m{h}_1,\m{h}_2);\m{free}} [ W_{n}\mathbf{1}_{\neg \m{C}_\delta}\mathbf{1}_{\m{D}_\delta}] + \Ex^{n+r;(\m{h}_1,\m{h}_2);\m{free}} [ W_{n}\mathbf{1}_{\neg \m{D}_\delta}].
    \end{align}
    By Lemma \ref{spread}, we can choose $\delta$ depending on $\varepsilon,M$ so that the second term on the right-hand side of \eqref{xfr} is at most $\varepsilon/\sqrt{n}$ for all large $n$, and by Lemma \ref{EWlbd} we know $\Ex^{n;(\m{h}_1,\m{h}_2);\m{free}}[W_n]\ge \frac1{C \sqrt{n}}$. Thus we get 
    \begin{align} \label{xfr1}
    \Ex^{n+r;(\m{h}_1,\m{h}_2);\m{free}} [ W_{n}\mathbf{1}_{\neg \m{D}_\delta}] \le C\e \cdot \Ex^{n;(\m{h}_1,\m{h}_2);\m{free}}[W_n].
    \end{align}
    For the first term on the right-hand side of \eqref{xfr}, we condition on the $\sigma$-algebra $\mathcal{F}_n = \sigma\{(S_1(k),S_2(k))_{k=1}^n\}$ and write
    \begin{align*}
        \Ex^{n+r;(\m{h}_1,\m{h}_2);\m{free}} [ W_{n}\mathbf{1}_{\neg \m{C}_\delta}\mathbf{1}_{\m{D}_\delta}] &= \Ex^{n+r;(\m{h}_1,\m{h}_2);\m{free}} [ W_{n}\mathbf{1}_{\m{D}_\delta} \cdot \Ex^{n+r;(\m{h}_1,\m{h}_2);\m{free}}[\mathbf{1}_{\neg\m{C}_\delta} \mid \mathcal{F}_n]]\\
        &= \Ex^{n;(\m{h}_1,\m{h}_2);\m{free}}[W_n \mathbf{1}_{\m{D}_\delta} \cdot \Pr^{r;(S_1(n),S_2(n));\m{free}}(\neg \m{C}_\delta)].
    \end{align*}
    It follows from tail bounds on independent walks of length $r$, i.e., Kolmogorov's inequality, that
    $\mathbf{1}_{\m{D}_\delta} \cdot \Pr^{r;(S_1(n),S_2(n));\m{free}}(\neg \m{C}_\delta) < \varepsilon$
    for large enough $n$ depending on $\delta,\varepsilon$. This leads to
    \begin{align*}
        \Ex^{n+r;(\m{h}_1,\m{h}_2);\m{free}} [ W_{n}\mathbf{1}_{\neg \m{C}_\delta}\mathbf{1}_{\m{D}_\delta}] \le \e \cdot \Ex^{n;(\m{h}_1,\m{h}_2);\m{free}} [ W_{n}].
    \end{align*}
    Plugging the above bound along with the bound in \eqref{xfr1} back in \eqref{xfr} gives us
     \begin{align*}
        \Ex^{n+r;(\m{h}_1,\m{h}_2);\m{free}} [ W_{n}\mathbf{1}_{\neg \m{C}_\delta}] &\leq (C\e +\e) \cdot \Ex^{n;(\m{h}_1,\m{h}_2);\m{free}} [ W_{n}].
    \end{align*}
    Finally, plugging the above bound back in \eqref{xfr0} leads to} 
    \[
    \Ex^{n+r;(\m{h}_1,\m{h}_2);\m{free}} [ W_{n+r}] \geq (1-\varepsilon)(1 - C\varepsilon-\e) \cdot \Ex^{n;(\m{h}_1,\m{h}_2);\m{free}} [ W_{n}].
    \]
    Since $\e$ was arbitrary the liminf of the ratio in \eqref{ratio11} is at least 1, finishing the proof.
\end{proof}

\subsection{Proof of Theorem \ref{thm:conv}} In this subsection we prove Theorem \ref{thm:conv}. We first argue that the densities involved in the limiting distributions are indeed valid density functions, and that the limit distribution can be viewed as the measure $\Pr_{\m{LG}}^{(\m{h}_1,\m{h}_2)}$ in Definition \ref{def:LGmeas} with $\m{h}_1=\delta_0$, $\m{h}_2=\ga$, where $\ga$ is defined in \eqref{defga}.

\begin{lemma}\label{idenf} The functions $p_0^V, p^V$ defined in Section \ref{sec:1.1} are valid density functions. The process $\big(S_1^\uparrow(k),S_2^\uparrow(k)\big)_{k=1}^r$ defined in Section \ref{sec:1.1} is equal in distribution to $(S_1(k),S_2(k))_{k=1}^r$ under $\Pr_{\m{LG}}^{(0,\ga)}$, where $\ga$ is defined in \eqref{defga}.
\end{lemma}

\begin{proof} Note that $V$ defined in \eqref{Vdef} equals $V_{0,\ga}(0,\cdot)$. Thus taking $r=1, \m{h}_1=\delta_0, \m{h}_2=\ga$, in view of Lemma \ref{expec1} we have
\begin{align*}
    \int_{\R} p_0^V(y) dy =\Ex^{1;(0,\ga);\m{free}} \left[V_{0,\ga}(0,S_2(1))\right] = \Ex^{1;(0,\ga);\m{free}} \left[V_{0,\ga}(S_1(1),S_2(1))\right]=1.
\end{align*}
Thus $p_0^V$ is a density. Now taking $r=2, \m{h}_1=\delta_{x_1}, \m{h}_2=\delta_{y_1}$, Lemma \ref{expec1} yields
\begin{align*}
    \int_{\R^2} p^V\big((x_1,y_1),(x_2,y_2)\big)\, dx_2\, dy_2 = \Ex^{2;(x_1,y_1);\m{free}}\left[\hat{W}_2V_{x_1,y_1}(S_1(2),S_2(2))\right]=1,
\end{align*}
and $p^V$ is also a density. Let us take any Borel set ${A}\subset \R^r$. Then
\begin{align*}
 \Pr\big((S_1^\uparrow(k),S_2^\uparrow(k))_{k=1}^r\in A\big) & =  \Ex^{r;(0,\ga);\m{free}}\bigg[\ind_{(S_1(k),S_2(k))_{k=1}^r\in A}\cdot\hat{W}_r {V_{0,\ga}(S_1(r),S_2(r))}\bigg] \\ & =\Pr_{\m{LG}}^{(0,\ga)}\big((S_1(k),S_2(k))_{k=1}^r\in A\big). 
\end{align*}
The last equality follows from the definition. 
\end{proof}

\begin{proof}[Proof of Theorem \ref{thm:conv}] Fix any $\m{h}_1,\m{h}_2 \in \m{IC}(M)$. For $K>0$ let $\m{C}_K = \{|S_1(r)|, |S_2(r)| \leq K\}$. We have the trivial lower bound
\begin{equation*}\label{AMlbd}
\Pr^{n; (\m{h}_1, \m{h}_2); (b_1,b_2)}_{W_n} (\m{A}) \geq \Pr^{n; (\m{h}_1, \m{h}_2); (b_1,b_2)}_{W_n} (\m{A}\cap\m{C}_K).
\end{equation*}
Let us write $\mathcal{F}_r$ for the $\sigma$-algebra generated by $(S_1(k),S_2(k))_{k=1}^r$. Using \eqref{decomp} and \eqref{transfin}, noting that $\ind_{\m{A}\cap \m{C}_K}\hat{W}_r$ is $\mathcal{F}_r$-measurable, we write
\begin{align*}
    \Pr^{n; (\m{h}_1, \m{h}_2); (b_1,b_2)}_{W_n} (\m{A}\cap\m{C}_K) &= \frac{\Ex^{n;(\m{h}_1, \m{h}_2);(b_1,b_2)}[\mathbf{1}_{\m{A}\cap\m{C}_K} W_n ]}{\Ex^{n;(\m{h}_1,\m{h}_2);(b_1,b_2)}[W_n]}\\
    &= \frac{\Ex^{n;(\m{h}_1, \m{h}_2);(b_1,b_2)}\left[\mathbf{1}_{\m{A}\cap\m{C}_K} \hat{W}_r \cdot \Ex^{n;(\m{h}_1,\m{h}_2);(b_1,b_2)}\left[{W}_{r\to n}\mid \mathcal{F}_r\right]\right]}{\Ex^{n;(\m{h}_1,\m{h}_2);(b_1,b_2)}[W_n]}\\
    &\geq \frac{\Ex^{n;(\m{h}_1, \m{h}_2);(b_1,b_2)}\left[\mathbf{1}_{\m{A}\cap\m{C}_K} \hat{W}_r\cdot \Ex^{n-r+1;(S_1(r),S_2(r));(b_1,b_2)}\left[W_{n-r+1} \right]\right]}{\Ex^{n;(\m{h}_1, \m{h}_2);(b_1,b_2)}[W_n]}\\
&= \Ex^{n;(\m{h}_1, \m{h}_2);(b_1,b_2)}\left[\mathbf{1}_{\m{A}\cap\m{C}_K} \hat{W}_r \cdot \frac{\Ex^{n-r+1;(S_1(r),S_2(r));(b_1,b_2)}[W_{n-r+1}]}{\Ex^{n;(\m{h}_1, \m{h}_2);(b_1,b_2)}[W_n]}\right]\\
&\geq (1-\e) \cdot \Ex^{r;(\m{h}_1, \m{h}_2);\m{free}}\left[\mathbf{1}_{\m{A}\cap\m{C}_K} \hat{W}_r \cdot \frac{\Ex^{n-r+1;(S_1(r),S_2(r));(b_1,b_2)}[W_{n-r+1}]}{\Ex^{n;(\m{h}_1, \m{h}_2);(b_1,b_2)}[W_n]}\right].
\end{align*}
The last line follows from Lemma \ref{bridgewalk}.  Now observe that for each $n$, {the last line is a continuous function of $(b_1,b_2)$ in the compact set $B_M$ defined in Theorem \ref{thm:conv}} and thus attains a minimum at some $(\beta_1,\beta_2) \in B_M$. Indeed, the expectations can be written as integrals against the continuous density $\m{f}$ using \eqref{freedensity} and \eqref{bridgedensity}, and the denominator is always nonzero. Therefore, for each $n$,  
\begin{align*}
    \Pr^{n; (\m{h}_1, \m{h}_2); (b_1,b_2)}_{W_n} (\m{A}) &\geq (1-\e) \cdot \Ex^{r;(\m{h}_1, \m{h}_2);\m{free}}\left[\mathbf{1}_{\m{A}\cap\m{C}_K} \hat{W}_r\cdot \frac{\Ex^{n-r+1;(S_1(r),S_2(r));(\beta_1,\beta_2)}[W_{n-r+1}]}{\Ex^{n;(\m{h}_1, \m{h}_2);(\beta_1, \beta_2)}[W_n]}\right].
\end{align*}
Now using Fatou's lemma and Lemma \ref{V(x,y)} {(with $(b_1,b_2) = (\beta_1,\beta_2) \in B_M$)}, we may choose $n_0$ (independent of $b_1,b_2$) so that for $n\geq n_0$ we have
\begin{align*}
    \Pr^{n; (\m{h}_1, \m{h}_2); (b_1,b_2)}_{W_n} (\m{A}) &\geq (1-\e)^2 \cdot \liminf_{n\to\infty} \Ex^{r;(\m{h}_1, \m{h}_2);\m{free}}\left[\mathbf{1}_{\m{A}\cap\m{C}_K} \hat{W}_r\cdot \frac{\Ex^{n-r+1;(S_1(r),S_2(r));(\beta_1,\beta_2)}[W_{n-r+1}]}{\Ex^{n;(\m{h}_1, \m{h}_2);(\beta_1, \beta_2)}[W_n]}\right]\\
    &\geq (1-\e)^2 \cdot
\Ex^{r;(\m{h}_1, \m{h}_2);\m{free}}\left[\mathbf{1}_{\m{A}}\cdot \hat{W}_r V_{\m{h}_1,\m{h}_2}(S_1(r),S_2(r)) \cdot \mathbf{1}_{\m{C}_K}\right].
\end{align*}
Note that passing to $(\beta_1,\beta_2)$ was necessary to ensure that we can choose $n_0$ independent of $b_1,b_2$. To remove the indicator $\mathbf{1}_{\m{C}_K}$ on the right side of the above equation, note by Lemma \ref{expec1} that $\Ex^{r;(\m{h}_1,\m{h}_2);\m{free}}[\hat{W}_rV(S_1(r),S_2(r))] = 1$, and $\hat{W}_r V_{\m{h}_1,\m{h}_2}(S_1(r),S_2(r))\mathbf{1}_{\neg\m{C}_K} \to 0$ a.s. as $K\to\infty$. By dominated convergence we can choose $K$ large enough so that for $n\geq n_0(M,\e)$ and for all $(b_1,b_2)$,
\begin{align}\label{Eflbd}
\Pr^{n; (\m{h}_1, \m{h}_2); (b_1,b_2)}_{W_n} (\m{A}) &\geq 
(1-\e)^3 \cdot \Ex^{r;(\m{h}_1, \m{h}_2);\m{free}}\left[\mathbf{1}_{\m{A}}\cdot \hat{W}_r V_{\m{h}_1,\m{h}_2}(S_1(r),S_2(r))\right].
\end{align}
For the upper bound, applying \eqref{Eflbd} to $\neg\m{A}$ in place of $\m{A}$ gives
\begin{equation}
\begin{split}
    \Pr^{n; (\m{h}_1, \m{h}_2); (b_1,b_2)}_{W_n} (\m{A}) &= 1- \Pr^{n; (\m{h}_1, \m{h}_2); (b_1,b_2)}_{W_n} (\neg\m{A})\\
    &\leq 1 - (1-\e)^3 \cdot \Ex^{r;(\m{h}_1, \m{h}_2);\m{free}}\left[\mathbf{1}_{\neg\m{A}}\cdot \hat{W}_r V_{\m{h}_1,\m{h}_2}(S_1(r),S_2(r))\right]\\
    &= 1 - (1-\e)^3 \cdot \Ex^{r;(\m{h}_1, \m{h}_2);\m{free}}[\hat{W}_r V_{\m{h}_1,\m{h}_2}(S_1(r),S_2(r))] \\
    &\qquad\qquad + (1-\e)^3 \cdot \Ex^{r;(\m{h}_1, \m{h}_2);\m{free}} \left[\mathbf{1}_{\m{A}}\cdot \hat{W}_r V_{\m{h}_1,\m{h}_2}(S_1(r),S_2(r))\right]\\
    &= (1-\e)^3 \cdot \Ex^{r;(\m{h}_1, \m{h}_2);\m{free}} \left[\mathbf{1}_{\m{A}} \cdot \hat{W}_r V_{\m{h}_1,\m{h}_2}(S_1(r),S_2(r))\right] + 1 - (1-\e)^3, \label{Efubd}
\end{split}
\end{equation}
with the last line following from Lemma \ref{expec1}. In view of Lemma \ref{idenf}, the desired bound follows from \eqref{Eflbd} and \eqref{Efubd} by taking $\m{h}_1=\delta_0, \m{h}_2=\ga$, and readjusting $\e$.
\end{proof}

\section{Separation between second and third curves} \label{sec4}

The main goal of this section is to show that with high probability there is a positive separation between the second and third curves of the $\hslg$ line ensemble at the the left boundary: Theorem \ref{t.lsep23} and Corollary \ref{t.usep23}. The proof of the separation result relies on the fact that the limit points of the left boundary value of the second curve are non-atomic. In the following proposition, we shall prove this non-atomicity result in the case of interacting random walk laws. As alluded to in the introduction, this will eventually translate into non-atomicity for the left boundary of the second curve of the line ensemble via the Gibbs property.

\begin{proposition}[Limit points are non-atomic]\label{p.nonatomic}
    Fix any $\e\in (0,1)$. There exists $\delta=\delta(\e)>0$ such that
    \begin{align*}
        \liminf_{T\to \infty} \sup_{r\in \mathbb{R}} \Pr_{\m{IRW}}^{T;(0,-\sqrt{T})}\left(|L_2(1)-r| \ge \delta \sqrt{T}\right) \ge 1-\e.
    \end{align*}
\end{proposition}

\begin{proof}
    Note that under the $\m{IRW}$ law,
    $e^{L_2(1)-L_2(2)}$ is distributed as $\operatorname{Gamma}(\alpha+\theta)$. {Indeed, in Figure \ref{fig:IRWgraph}, $L_2(1)-L_2(2)$ is given by the yellow arrow, whose weight prescribed in \eqref{def:wfn} is log-Gamma with parameter $\alpha+\theta$.} Hence it suffices to show the non-atomicity for $L_2(2)$ instead. 
    By Lemma \ref{wprwconn}, $L_2(2)$ is equal in distribution to $S_2(1)$ where $(S_1(k),S_2(k))_{k=1}^T$ are distributed as $\Pr_{\m{WPRW}}^{T;(0,-\sqrt{T})}$ defined in Definition \ref{prb}. 

    Fix $r\in\mathbb{R}$. Let $\mathsf{A} = \{|S_2(1)-r| < \delta\sqrt{T}\}$. Our goal is to show
    \begin{enumerate}[leftmargin=20pt]
        \item[($\star$)] for all large enough $T$, $\mathbb{P}^{T;(0,-\sqrt{T})}_{\m{WPRW}}(\mathsf{A} )$ can be made arbitrarily small (uniformly in $r\in \mathbb{R}$) by taking $\delta$ small enough. 
    \end{enumerate}
    We use \eqref{wscint} to write
    \[
    \mathbb{P}^{T;(0,-\sqrt{T})}_{\m{WPRW}}(\mathsf{A} ) = \frac{\mathbb{E}^{T;(0,-\sqrt{T})}_{\m{PRW}}[  \wsc\ind_\mathsf{A}]}{\mathbb{E}^{T;(0,-\sqrt{T})}_{\m{PRW}}[\wsc]}.
    \]
    We now provide lower and upper bounds for the denominator and the numerator of the r.h.s.~of the above equation respectively. Thanks to Corollary 4.12 from \bcd, we have $ \mathbb{E}^{T;(0,-\sqrt{T})}_{\m{PRW}}[\wsc] \geq \frac{1}{C\sqrt{T}}$. On the other hand, by Lemma 4.11 in \bcd, we have
    \begin{align*}
        \mathbb{E}^{T;(0,-\sqrt{T})}_{\m{PRW}}[ \wsc\mathbf{1}_\mathsf{A} ] &\leq  \frac{C}{T} + \frac{C}{\sqrt{T}} \mathbb{E}^{T;(0,-\sqrt{T})}_{\m{PRW}}\left[\mathbf{1}_\mathsf{A}  [(S_2(1)-S_1(1) + 1) \vee 1] \big[\tfrac{|S_1(1)|}{\sqrt{T}}\vee 2\big]^{3/2}\right]\\
        &\leq \frac{C}{T} + \frac{C}{\sqrt{T}} \sqrt{\mathbb{P}^{T;(0,-\sqrt{T})}_{\m{PRW}}(\mathsf{A}) \cdot\mathbb{E}^{T;(0,-\sqrt{T})}_{\m{PRW}}\left[[(S_2(1)-S_1(1) + 1) \vee 1]^2 \big[\tfrac{|S_1(1)|}{\sqrt{T}}\vee 2\big]^3\right]},
    \end{align*}
    where in the second line we have used Cauchy-Schwarz inequality.  It is known from Lemma 4.7 in \bcd\ that $S_2(1)-S_1(1)$ and $S_1(1)/\sqrt{T}$ have exponential tails {(independent of $T$)} under the $\m{PRW}$ law. This implies, with another application of Cauchy-Schwarz, that the expectation factor in the last line is bounded by a constant uniformly in $T$. Thus, to conclude the proof, it now suffices to show $(\star)$ under the $\m{PRW}$ law.

    Recall the density of the $\m{PRW}$ law from \eqref{den}. We may write
    
    \[
    \mathbb{P}^{T;(0,-\sqrt{T})}_{\m{PRW}}(\mathsf{A} ) = \frac{\mathbb{E}^{T;(0,-\sqrt{T})}_{\m{RW}}[\mathbf{1}_\mathsf{A}  \cdot \mathsf{g}_\alpha(S_2(1)-S_1(1))]}{\mathbb{E}^{T;(0,-\sqrt{T})}_{\m{RW}}[ \mathsf{g}_{\alpha}(S_2(1)-S_1(1))]},
    \]
    where $\mathbb{P}^{T;(0,-\sqrt{T})}_{\m{RW}}$ denotes the measure proportional to  $\delta_{\omega_1(T)=0}\delta_{\omega_2(T)=-\sqrt{T}}$ times a density
    \[
    \prod_{k=2}^T \mathsf{f}_\theta(\omega_1(k) - \omega_1(k-1)) \mathsf{f}_\theta(\omega_2(k) - \omega_2(k-1)).
    \]
    By (4.18) in \bcd, there is a constant $C>0$ so that for all $T$,
    \begin{equation}\label{glbd}\mathbb{E}^{T;(0,-\sqrt{T})}_{\m{RW}}[ \mathsf{g}_\alpha(S_2(1)-S_1(1))] \geq \tfrac{1}{C\sqrt{T}}.
    \end{equation}
    From the precise expression of $\ga$, we see that it has exponential tails. In particular, $\mathsf{g}_\alpha(x) \leq Ce^{-|x|/C}$. Using this we have
    \begin{align} \nonumber
        & \mathbb{E}^{T;(0,-\sqrt{T})}_{\m{RW}}[\mathbf{1}_\mathsf{A}  \cdot \mathsf{g}_\alpha(S_2(1)-S_1(1))] \\ \nonumber &\leq \mathbb{E}^{T;(0,-\sqrt{T})}_{\m{RW}}\big[\mathbf{1}_\mathsf{A}  \mathbf{1}_{|S_2(1)-S_1(1)|>\sqrt{T}} \cdot \mathsf{g}_\alpha(S_2(1)-S_1(1))\big]\\ \nonumber
        &\quad \hspace{2cm}+ \sum_{0\leq p\leq \sqrt{T}} \mathbb{E}^{T;(0,-\sqrt{T})}_{\m{RW}}\big[\mathbf{1}_\mathsf{A}  \mathbf{1}_{|S_2(1)-S_1(1)|\in [p,p+1]} \cdot \mathsf{g}_\alpha(S_2(1)-S_1(1))\big]\\
        &\leq Ce^{-\frac{\sqrt{T}}C} \cdot \mathbb{P}^{T;(0,-\sqrt{T})}_{\m{RW}}(\mathsf{A} ) + \sum_{p\leq\sqrt{T}} Ce^{-\frac{p}C} \cdot\mathbb{P}^{T;(0,-\sqrt{T})}_{\m{RW}}\big(\mathsf{A}  \cap \{p \leq |S_2(1)-S_1(1)| \leq p+1\}\big). \label{Esplit}
    \end{align}
Note that under $\mathbb{P}^{T;(0,-\sqrt{T})}_{\m{RW}}$, $S_1(1)$ and $S_2(1)+\sqrt{T}$ are i.i.d.~variables each distributed as the sum of $T$ i.i.d. variables with density $\fa$.  Thus, for the first term in \eqref{Esplit}, the Berry-Esseen theorem implies that the distribution function $F$ of $S_2/\sqrt{T}$ satisfies $\lVert F - \Phi\rVert_\infty \leq C\rho/\sigma^3\sqrt{T}$, where $\Phi$ is the cdf of a Gaussian random variable $Z_2$ with mean $-1$ and variance $\sigma^2:=\int x^2\fa (x)$. It follows that
    \begin{equation}\label{Eubd}
    \mathbb{P}^{T;(0,-\sqrt{T})}_{\m{RW}}(\mathsf{A} ) \leq \mathbb{P} \left(\big|Z_2 - \tfrac{r}{\sqrt{T}} \big| < \delta \right) + \frac{4C\rho}{\sigma^3\sqrt{T}} \leq C\delta,
    \end{equation}
        uniformly over $r\in\mathbb{R}$ once $T$ is large enough depending on $\delta$. To handle the sum in \eqref{Esplit}, we can write
    \begin{equation}\label{Esum}\begin{split}
        &\mathbb{P}^{T;(0,-\sqrt{T})}_{\m{RW}}\big(\mathsf{A}  \cap \{p \leq |S_2(1)-S_1(1)| \leq p+1\}\big)\\ &\qquad = \int_{|y-r|<\delta\sqrt{T}} \mathbb{P}^{T;(0,-\sqrt{T})}_{\m{RW}}\big(p \leq |S_1(1)-y| \leq p+1\big)\mathbb{P}^{T;(0,-\sqrt{T})}_{\m{RW}}(S_2(1)=dy).
    \end{split}
    \end{equation}
    To estimate the first probability in the sum, we rely on a convergence result for the density of the random walk endpoint to a Gaussian density,  Lemma B.3 in \bcd. This lemma implies that uniformly over $p\in[0,\sqrt{T}]$,
    \[
    \mathbb{P}^{T;(0,-\sqrt{T})}_{\m{RW}}(p \leq |S_1(1)-y| \leq p+1) = (1 + o_T(1)) \cdot \mathbb{P}\left(\big|Z_1-\tfrac{y}{\sqrt{T}}\big| \in [\tfrac{p}{\sqrt{T}},\tfrac{p+1}{\sqrt{T}}]\right) \leq \frac{C}{\sqrt{T}},
    \]
    where $Z_1$ is a Gaussian with mean 0 and variance $\sigma^2$ independent from $Z_1$. Inserting this bound in \eqref{Esum} and using \eqref{Eubd}, we find
    \begin{align*}
        \mathbb{P}^{T;(0,-\sqrt{T})}_{\m{RW}}(\mathsf{A}  \cap \{p \leq |S_2(1)-S_1(1)| \leq p+1\}) &\leq \frac{C}{\sqrt{T}}\int_{|y-r|<\delta\sqrt{T}} \mathbb{P}^{T;(0,-\sqrt{T})}_{\m{RW}}(S_2(1)=dy)\\
        &= \frac{C}{\sqrt{T}} \cdot \mathbb{P}^{T;(0,-\sqrt{T})}_{\m{RW}}(\mathsf{A} ) \leq \frac{C\delta}{\sqrt{T}}. 
    \end{align*}
    Plugging the above estimate and the bound from \eqref{Eubd} into \eqref{Esum}, we find
    \[
    \mathbb{E}^{T;(0,-\sqrt{T})}_{\m{RW}}[\mathbf{1}_\mathsf{A} \cdot \mathsf{g}_\alpha(S_2(1)-S_1(1))] \leq C\delta e^{-\sqrt{T}/C} + \sum_{p\leq\sqrt{T}} Ce^{-p/C}\tfrac{\delta}{\sqrt{T}} \leq\frac{C\delta}{\sqrt{T}}.
    \]
    In combination with \eqref{glbd}, this implies $(\star)$ for the $\m{PRW}$ law, completing the proof.
\end{proof}

With the aid of the above proposition, we now establish the high probability separation at the left boundary of the second and third curves.

\begin{theorem}[Left boundary separation] \label{t.lsep23}
    Fix any $\e\in (0,1)$. There exists $\delta=\delta(\e)>0$ such that
    \begin{align*}
        \liminf_{N\to \infty} \Pr\big(\min\{\L_2^N(1),\L_2^N(3)\}-\L_3^N(2) \ge \delta N^{1/3}\big) \ge 1-\e.
    \end{align*}
\end{theorem}

\begin{proof}
First note that by tightness of the second curve (Theorem \ref{p.tight3}),  for any $\delta>0$ we have
\[
\liminf_{N\to\infty} \Pr\left(\L_2^N(3) - \L_2^N(1) \geq - \delta N^{1/3}\right) \geq 1-\tfrac{\e}{2}.
\]
It follows that for sufficiently large $N$,
\begin{equation*}
\Pr\left(\min\{\L_2^N(1), \L_2^N(3)\} - \L_3^N(2) \geq \delta N^{1/3}\right) \geq \Pr\left(\L_2^N(1) - \L_3^N(2) \geq 2\delta N^{1/3}\right) - \tfrac{\e}{2}.
\end{equation*}
Now let us set $T = \lceil N^{2/3}\rceil$. By tightness of the $\hslg$ line ensemble (Theorem \ref{p.tight3}), we can choose $M$ large enough depending on $\varepsilon$ so that for all large $N$, 
\begin{align*}
    \Pr(\m{A}) \ge 1-\tfrac\e4, \mbox{ where } \m{A}:=\bigg\{|\mathcal{L}_1^N(2T-1)|+ |\mathcal{L}_2^N(2T)|+|\mathcal{L}_3^N(2)| \le M\sqrt{T}\bigg\}.
\end{align*}
Thus it suffices to show
\begin{align*}
    \Pr\left(\m{A}\cap\big\{\L_2^N(1) - \L_3^N(2) \leq 2\delta N^{1/3}\big\}\right) \le \tfrac{\e}4.
\end{align*}
Let us consider the $\sigma$-field $$\mathcal{F}=\big\{\L_1^N\ll 2T-1,2N\rr, \L_2^N\ll2T,2N\rr, \L_i^N\ll 1,2N-2i+2\rr \mbox{ for } i\in \ll3,N\rr\big\}.$$
Observe that $\m{A} \in \mathcal{F}$. Invoking the Gibbs property, we see that the conditional measure given  $\mathcal{F}$ can be viewed as a $\hslg$ Gibbs measure on the domain $\Phi$ defined in \eqref{k2t}. By the tower property of conditional expectation we have
\begin{align*}
    \Pr\left(\m{A}\cap\big\{\L_2^N(1) - \L_3^N(2) \leq 2\delta N^{1/3}\big\}\right) = \Ex\left[\ind_{\m{A}} \Pr_{\Phi}^{(a,b,\vec{c})}\big(L_2(1) - \L_3^N(2) \leq 2\delta N^{1/3}\big)\right],
\end{align*}
where $\Pr_{\Phi}^{(a,b,\vec{c})}$ was defined in \eqref{def:Gibbslaw}. Here $a=\L_1^N(2T-1)$, $b=\L_2^N(2T)$, and $c_j=\L_3^N(2j)$ for $j\in \ll1,T\rr$.
However, on the event $\m{A}$, we have $a,b,c_1 \in [-M\sqrt{T},M\sqrt{T}]$. Thus it suffices to prove that uniformly over $a,b,c_1\in[-M\sqrt{T},M\sqrt{T}]$ and $c_2,\dots,c_T \in\mathbb{R}$, we can find $\delta>0$ so that
\begin{equation}\label{Eabc}
\limsup_{T\to\infty} \Pr_{\Phi}^{(a,b,\vec{c})}\left(L_2(1) - c_1 \leq 2\delta \sqrt{T}\right) \leq \tfrac{\e}{4}.
\end{equation}
Observe that the above event is increasing. Thus by stochastic monotonicity (Proposition \ref{p:gmc}) and the relation \eqref{def:Gibbslaw},
\begin{align}\nonumber
    \Pr_{\Phi}^{(a,b,\vec{c})}\left(L_2(1) - c_1 \leq 2\delta \sqrt{T}\right) & \le \Pr_{\Phi}^{(a',b',\vec{c}^{\,'})}\left(L_2(1) - c_1 \leq 2\delta \sqrt{T}\right) \\ & = \frac{\Ex_{\m{IRW}}^{T;(a',b')}[\exp(-e^{c_1-L_2(1)}-e^{c_1-L_2(3)})\ind_{L_2(1) - c_1 \leq 2\delta \sqrt{T}}]}{\Ex_{\m{IRW}}^{T;(a',b')}[\exp(-e^{c_1-L_2(1)}-e^{c_1-L_2(3)})]} \label{Eabc2}
\end{align}
where $a'=-M\sqrt{T}, b'=-(M+1)\sqrt{T}$, and $\vec{c}^{\,'}=(c_1,c_2',\ldots,c_T')$ with $c_j'=-\infty$ for $j\in \ll2,T\rr$. We now provide upper and lower bounds for the numerator and the denominator in \eqref{Eabc2} respectively.
To lower bound the denominator, note that by Proposition 4.1 in \bcd\ there exists $\phi>0$ depending on $M$ so that
\[
\Pr_{\m{IRW}}^{T;(a',b')}\left(\min\{L_2(1),L_2(3)\} \geq 2M\sqrt{T}\right) \geq \phi
\]
for large enough $T$.
Since $c_1\leq M\sqrt{T}$, it follows that 
\begin{equation}\label{denomlbd}
\Ex_{\m{IRW}}^{T;(a',b')}\left[\exp(-e^{c_1 - L_2(1)} - e^{c_1 - L_2(3)}) \right] \geq \phi \cdot \exp(-2e^{-M\sqrt{T}}) \geq \tfrac{\phi}{2}
\end{equation}
for sufficiently large $T$ depending on $M$. As for the numerator in \eqref{Eabc2}, invoking translation invariance (Lemma \ref{traninv}) yields
\begin{align*}
   & \Ex_{\m{IRW}}^{T;(a',b')}\left[\exp(-e^{c_1-L_2(1)}-e^{c_1-L_2(3)})\ind_{L_2(1) - c_1 \leq 2\delta \sqrt{T}}\right] \\ & \hspace{3cm} = \Ex_{\m{IRW}}^{T;(0,-\sqrt{T})}\left[\exp(-e^{r-L_2(1)}-e^{r-L_2(3)})\ind_{L_2(1) - r\leq 2\delta \sqrt{T}}\right]
\end{align*}
where $r=c_1+M\sqrt{T}$.  By splitting the indicator and bounding the exponential by 1 on the second part, we get an upper bound of
\begin{align*}
&\Ex_{\m{IRW}}^{T;(0,-\sqrt{T})}\left[\mathbf{1}_{L_2(1) - r < -2\delta\sqrt{T}} \, \exp(-e^{r - L_2(1)} - e^{r - L_2(3)}) \right] + \Pr_{\m{IRW}}^{T;(0,-\sqrt{T})}\left( \left|L_2(1) - r\right| \le 2\delta\sqrt{T} \right)\\
&\qquad \leq \exp(-e^{2\delta\sqrt{T}}) + \Pr_{\m{IRW}}^{T;(0,-\sqrt{T})}\left( \left|L_2(1) - r\right| \le 2\delta\sqrt{T} \right).
\end{align*}
The first term vanishes as $T\to\infty$ for any $\delta>0$, and by Proposition \ref{p.nonatomic} the second term can be made less than $\frac{\e}{4}\cdot\frac{\phi}{2}$ uniformly over $r$ (hence over $a,b,\vec{c}$) for large $T$ by choosing $\delta$ small enough. Plugging this estimate and the bound in \eqref{denomlbd} back into \eqref{Eabc2} verifies \eqref{Eabc}, completing the proof.
\end{proof}

Using the tightness of the $\hslg$ line ensemble, the above separation can be extended into a small $N^{2/3}$ window around the left boundary.

\begin{corollary}[Uniform separation] \label{t.usep23}
    Fix any $\e\in (0,1)$. There exist $\delta=\delta(\e), \rho=\rho(\e)>0$ such that
    \begin{align}\label{eq:t.usep23}
        \liminf_{N\to \infty} \Pr\bigg(\inf_{k\in \ll1,\rho N^{2/3}\rr}\big[\min\{\L_2^N(2k-1),\L_2^N(2k+1)\}-\L_3^N(2k)\big] \ge \delta N^{1/3}\bigg) \ge 1-\e.
     \end{align}
\end{corollary}

\begin{proof}
    For $\delta,\rho>0$, define the three events
    \begin{align*}
        \mathsf{A}_{\delta} &= \left\{ \L_2^N(1) - \L_3^N(2) \geq 2\delta N^{1/3}\right\}, \quad \mathsf{B}_{\delta}^{\rho} = \left\{\inf_{k\in\llbracket 1,3\rho N^{2/3}\rrbracket} \L_2^N(k) - \L_2^N(1) \geq -\tfrac{\delta}{2}N^{1/3}\right\},\\
        \mathsf{C}_{\delta}^\rho &= \left\{\sup_{k\in\llbracket 1,\rho N^{2/3}\rrbracket} \L_3^N(2k) - \L_3^N(2) \leq \tfrac{\delta}{2}N^{1/3}\right\},
    \end{align*}
    and let $\mathsf{D}_\delta^\rho =  \mathsf{A}_{\delta} \cap \mathsf{B}_{\delta}^{\rho} \cap \mathsf{C}_{\delta}^{\rho}$.
    For $\ell \in \{2k-1, 2k+1\}$ where $k\in\llbracket 1,\rho N^{2/3}\rrbracket$, on the event $\mathsf{D}_\delta^\rho$ we have
    \begin{align*}
    \L_2^N(\ell) - \L_3^N(2k) &= \big(\L_2^N(\ell) - \L_2^N(1)\big) - \big(\L_3^N(2k) - \L_3^N(2)\big) + \big(\L_2^N(1) - \L_3^N(2)\big)\\
    &\geq -\tfrac{\delta}{2}N^{1/3} - \tfrac{\delta}{2}N^{1/3} + 2\delta N^{1/3} = \delta N^{1/3}.
    \end{align*}
    Thus the event in \eqref{eq:t.usep23} is satisfied on $\mathsf{D}_\delta^\rho$. By Theorem \ref{t.lsep23}, we may first choose $\delta(\e)>0$ so that $\Pr(\m{A}_{\delta}) \ge 1-\frac\e3$ for large enough $N$. By the tightness of the line ensemble (Theorem \ref{p.tight3}), we may then choose $\rho(\e)>0$ such that $\Pr(\m{B}_{\delta}^{\rho})\ge 1-\frac\e3$ and $\Pr(\m{C}_{\delta}^{\rho}) \ge 1-\frac\e3$ for large enough $N$. This leads to $\Pr(\mathsf{D}_\delta^\rho) \geq 1-\e$ by a union bound.
\end{proof}

\section{Proof of the main theorem} \label{sec5}

In this section, we prove our main theorem, Theorem \ref{t.main1}, about convergence for increments of the first curve of the $\hslg$ line ensemble at the left boundary. Before delving into the proof of Theorem \ref{t.main1} we require one final ingredient, which establishes diffusive separation between the first and second curves of the line ensemble at any mesoscopic scale near the left boundary.

\begin{proposition}[Diffusive separation between first and second curve] \label{p.sep12}
    Fix any $\e\in (0,1)$ and $\gamma\in (0,1)$. There exists $\delta=\delta(\e,\gamma)>0$ such that
    \begin{align*}
        \liminf_{N\to \infty}  \Pr\big(\L_1^N(2\lfloor N^{2\gamma/3}\rfloor -1)-\L_2^N(2\lfloor N^{2\gamma/3}\rfloor) \in [\delta N^{\gamma/3},\delta^{-1}N^{\gamma/3}]\big) \ge 1-\e.
    \end{align*}
\end{proposition}

\begin{proof} Fix  any $\e\in (0,1)$ and $\gamma\in (0,1/3)$. Set $T=N^{1/3}$. We assume $N$ (and hence $T$) is large enough throughout the proof and for convenience we also assume $T$  and $T^{2\gamma}$ are integers. For each $\delta>0$, consider the event
	\begin{align*}
		\m{A}_{\delta} := \big\{\L_1^N(2T^{2\gamma} -1)-\L_2^N(2T^{2\gamma}) \notin [\delta T^{\gamma},\delta^{-1}T^{\gamma}]\big\}.
	\end{align*}		
	By Theorem \ref{p.tight3} and Proposition \ref{pp:order}, there exists $M=M(\e)>0$ such that $\Pr(\m{B}_{M}) \ge 1-\e/3,$ where $\m{B}_M=\m{B}_{1,M}\cap \m{B}_{2,M}$ with
	\begin{align*}
	\m{B}_{1,M} & := \left\{|\L_1^N(1)|+|\L_2^N(2)|\le M\sqrt{T}\right\}, \\
        \m{B}_{2,M} & :=\Big\{ |\L_1^N(2T-1)|+|\L_2^N(2T)| \le M\sqrt{T}, \,\L_1^N(2T-1) \ge \L_1^N(2T)-(\log T)^{7/6}\Big\}.
	\end{align*}
Consider the $\sigma$-field
	\begin{align*}
		\mathcal{F}:=\sigma\left\{(\L_i^N\ll 1,2N-2i+2\rr)_{i\ge 3}, (\L_i^N(j))_{j\ge 2T+i-2, i\in \ll1,2\rr}\right\}.
	\end{align*}
 Recall the domain $\Phi$ from \eqref{k2t}. Using Theorem \ref{thm:conn} and the relation \eqref{def:Gibbslaw} we have 
\begin{align*}
    \Ex\left[\ind_{\m{B}_{1,M}\cap \m{A}_{\delta}}\mid \mathcal{F}\right] = \Pr_{\Phi}^{\vec{u}}(\m{B}_{1,M}\cap \m{A}_{\delta}) = \frac{\Ex_{\m{IRW}}^{T;(a,b)}[e^{-\mathcal{H}(\vec{u})}\ind_{\m{B}_{1,M}\cap \m{A}_{\delta}}]}{\Ex_{\m{IRW}}^{T;(a,b)}[e^{-\mathcal{H}(\vec{u})}]}
\end{align*}
where $a=\L_1^N(2T-1)$, $b=\L_2^N(2T)$, and $\vec{u}\in \mathbb{R}^{|\partial \Phi|}$ with $u_{i,2T+i-2}=\L_i^N(2T+i-2)$ for $i\in \{1,2\}$ and $u_{3,2j}=\L_3^N(2j)$ for $j\in \ll1,T\rr$. For each $\beta>0$, let us set $\m{C}(\beta):=\{\Ex_{\m{IRW}}^{T;(a,b)}[e^{-\mathcal{H}(\vec{u})}] \ge \beta\}.$ It was shown in  \bcd\ that given any $\e>0$ there exists $\beta(\e)>0$ such that $\Pr(\neg \m{C}(\beta))\le \e/3$ (see eq.~(5.12) in \bcd). We work with this choice of $\beta$ for the rest of the proof.
	Note that $\m{B}_{2,M}\cap \m{C}(\beta) \in \mathcal{F}$. By a union bound and the tower property of conditional expectation we have
	\begin{equation}
		\label{3terms}
		\begin{aligned}
			\Pr(\m{A}_{\delta}) &  \le \Pr(\neg \m{B}_{M})+\Pr\left(\neg \m{C}(\beta)\right) +\Ex\left[\ind_{\m{B}_{2,M}\cap\m{C}_{}(\beta)}\Ex\left[\ind_{\m{B}_{1,M}\cap \m{A}_{\delta}}\mid \mathcal{F}\right]\right] \\ & \le 2\e/3+ \Ex\left[\ind_{\m{B}_{2,M}\cap\m{C}_{}(\beta)}\Ex\left[\ind_{\m{B}_{1,M}\cap \m{A}_{\delta}}\mid \mathcal{F}\right]\right] \\ & \le 2\e/3+ \beta^{-1}\Ex\left[\ind_{\m{B}_{2,M}}\Pr_{\m{IRW}}^{T;(a,b)}({\m{B}_{1,M}\cap \m{A}_{\delta}})\right].
		\end{aligned}
	\end{equation}
We claim that 
\begin{itemize}[leftmargin=15pt]
    \item[$(\star)$] there exists $\delta(\e)>0$ such that for all $(u,v) \in D_M:=\{(a,b) : |a|+|b|\le M\sqrt{T}, a-b \ge -(\log T)^{7/6}\}$ we have
\begin{align} \label{irwbd2}
    \Pr_{\m{IRW}}^{T;(u,v)}(\m{B}_{1,M}\cap \m{A}_\delta) \le \beta\e/3.
\end{align}
\end{itemize}
for $T$ large enough. Plugging this bound back into \eqref{3terms} yields $\Pr(\m{A}_\delta) \le \e$ for all large enough $T$, which is precisely what we want to show. We thus focus on proving \eqref{irwbd2}. Towards this end, recall the $\m{WPRW}$ law from Definition \ref{prb} and its connection to the $\m{IRW}$ law from Lemma \ref{wprwconn}. Thanks to this connection, it suffices to show $(\star)$ under the $\m{WPRW}$ law where the events are now interpreted as
\begin{align*}
    \m{B}_{1,M}:=\left\{|S_1(1)|+|S_2(1)| \le M\sqrt{T}\right\}, \qquad \m{A}_{\delta}=\left\{S_1(T^{2\gamma})-S_2( T^{2\gamma}) \in [\delta T^{\gamma},\delta^{-1}T^{\gamma}]\right\}.
\end{align*}
From eq.~(5.34) in \bcd\  we know that for any event $\m{E} \in \sigma\{S_1(k),S_2(k) : k\in \ll1,T/4\rr\}$, there exists $M'>0$ depending on $\beta, \e, M$ such that 
    \begin{align}\label{rrte}
        \sup_{(u,v)\in D_M} \Pr_{\m{WPRW}}^{T;(u,v)}\big(\m{B}_{1,M} \cap \m{A}_\delta\big)  \le  \tfrac{\beta\e}{6}+ M' \cdot \sup_{\substack{p \in [0,2\log\log T] \\ |a_1-a_2|\le M_2+2\log\log T}} \Pr_{\m{RW}}^{T/4;(a_1,a_2)}\left(\m{B}_{1,M}\cap \m{A}_{\delta,p} \mid \m{NI}\right), 
    \end{align}
where $\m{A}_{\delta,p}:= \{S_1(T^{\gamma})-p-S_2(T^{\gamma}) \not\in [\delta T^{\gamma},\delta^{-1}T^{\gamma}]\}$ and $\m{NI}:= \bigcap_{i\in \ll2,T/4\rr} \{S_1(k)>S_2(k)\}.$ Under the non-intersection condition, it is well known (see \cite{igl} for example) that $S_1(\cdot)-S_2(\cdot)$ under diffusive scaling converges to a Brownian meander, whose endpoint is strictly positive with probability $1$. Using this result, it is not hard to obtain estimates for the probability on the right of \eqref{rrte} uniform over the starting points, as was done in Appendix C in \bcd. In particular, invoking Lemma C.5 in \bcd, we can make the supremum on the r.h.s.~of \eqref{rrte} arbitrary small by taking $\delta$ small enough. This proves $(\star)$ for the $\m{WPRW}$ law.
\end{proof}

\begin{proof}[Proof of Theorem \ref{t.main1}]
By Theorem \ref{thm:conn} it suffices to show that for any Borel set $A\subset\mathbb{R}^r$, we have
\begin{align}\label{eq:mainthm}
    \lim_{N\to\infty}{\Pr}(\m{A}_N) = \Pr\left(\big(S_1^\uparrow(k)\big)_{k=1}^r \in A \right),
\end{align}
where $\mathsf{A}_N = \left\{(\mathcal{L}_1^N(2k-1) - \mathcal{L}_1^N(1))_{k=1}^r \in A\right\}$ and $\big(S_1^\uparrow(k)\big)_{k\geq 1}$ is defined in Section \ref{sec:1.1}. We write $T = \lfloor N^{1/4}\rfloor$ and define the $\sigma$-algebra
\begin{align*}
    \mathcal{F} &= \sigma\left(\{\mathcal{L}_1^N(j) : j\geq 2T-1 \} \cup \{\mathcal{L}_2^N(j) : j\geq 2T\} \cup \{\mathcal{L}_i^N(j) : i\geq 3, j\geq 1\}\right).
\end{align*}
For $N\in\mathbb{Z}_{\geq 1}$ and $\delta>0$, we define the events
\begin{align*}
        \mathsf{Gap}^{2,3}_N(\delta) &:= \left\{\mathcal{L}_2^N(2T) - \mathcal{L}_3^N(2T) \geq \delta N^{\frac13}\right\}, \quad
         \mathsf{Sep}^{1,2}_N(\delta) := \left\{\mathcal{L}_1^N(2T-1) - \mathcal{L}_2^N(2T) \in [\delta \sqrt{T}, \delta^{-1}\sqrt{T}]\right\},\\
       \mathsf{Low}^3_N(\delta) & := \left\{\sup_{k\in\llbracket 1,T\rrbracket}\left(\mathcal{L}_3^N(2k) - \mathcal{L}_3^N(2T)\right) \leq \tfrac{\delta}{4}N^{\frac13}\right\}, \quad \mathsf{Good}_N(\delta) := \mathsf{Gap}^{2,3}_N(\delta) \cap \mathsf{Low}^3_N(\delta) \cap \mathsf{Sep}^{1,2}_N(\delta).
    \end{align*}
    Note that all the above events are measurable with respect to $\mathcal{F}$.
Fix any $\e>0$. Combining Corollary \ref{t.usep23}, Theorem \ref{p.tight3}, and Proposition \ref{p.sep12} we get that
    there exists $\delta(\e)>0$ so that
    \begin{equation}\label{good>1-e}
\mathbb{P}\left(\mathsf{Good}_N(\delta)\right) \geq 1-\varepsilon
    \end{equation}
for all large enough $N$. This implies    
\begin{equation}\label{AGood}
\big|{\Pr}(\mathsf{A}_N) - {\Pr}(\mathsf{A}_N \cap \mathsf{Good}_N(\delta))\big|\le \e,
\end{equation}
for sufficiently large $N$. Thus it suffices to show  $\Pr(\mathsf{A}\cap\mathsf{Good}_N(\delta))$ can be made arbitrarily close to by taking $N$ large enough and $\delta$ small enough. Towards this end, we recall the Gibbs property from Theorem \ref{thm:conn} and the law $\Pr_{\Phi}^{(a,b,\vec{c})}$ from \eqref{def:Gibbslaw}. Using the tower property of conditional expectation followed by the Gibbs property, we can write
\begin{align}\label{atower}
    \Pr\big(\mathsf{A}_N\cap\mathsf{Good}_N(\delta)\big) = \Ex\left[\ind_{\mathsf{Good}_N(\delta)}\cdot\mathbb{E}[\mathbf{1}_{\mathsf{A}_N} \mid \mathcal{F}]\right] = \Ex\left[\ind_{\mathsf{Good}_N(\delta)}\cdot\mathbb{P}_{\Phi}^{(a,b,\vec{c})}(\mathsf{A}_N)\right]
\end{align}
where $a=\mathcal{L}_1^N(2T-1), b=\mathcal{L}_2^N(2T)$, $c_j=\L_3^N(2j)$ for $j\in \ll1,T\rr$. In the r.h.s.~of the above equation, we interpret the event $\m{A}$ as $\{(L_1(2k-1)-L_1(1))_{k=1}^r \in A\}$. 

For the rest of the proof we work with deterministic $(a,b,\vec{c}) \in D_N(\delta)$ where
\begin{align*}
    D_N(\delta):= \bigg\{(a,b,\vec{c}) : b-c_T \ge \delta N^{\frac13}, \, \delta\sqrt{T}\le a-b \le \delta^{-1}\sqrt{T}, \, \sup_{k\in \ll1,T\rr} (c_k-c_T) \le \tfrac\delta4N^{\frac13}\bigg\}.
\end{align*}
Clearly on the event $\m{Good}_N(\delta)$, the random boundary data of the Gibbs measure in \eqref{atower} always lies in $D_N(\delta)$. Thus it suffices to obtain estimates that are uniform over all choices in $D_N(\delta)$. This will allow us to use the same estimates  for the random boundary conditions in \eqref{atower}. We now claim that
\begin{align}\label{eq:LEIRW}
\sup_{(a,b,\vec{c})\in D_N(\delta)} \left|\mathbb{P}_{\Phi}^{(a,b,\vec{c})}(\mathsf{A}_N)- \mathbb{P}_{\m{IRW}}^{(a,b)}(\mathsf{A}_N)\right| \le \e.
\end{align}
We postpone the proof of \eqref{eq:LEIRW}. 
Consider the event $\m{Left}_T(M) = \left\{|L_1(1) - a| \leq M\sqrt{T} \right\} \supset \m{Diff}^1_T(M)$, where $\m{Diff}_T^1(M)$ is defined in \eqref{eq:diff}.  
By Lemma \ref{lem:IRWupperbd} we may choose $M(\e,\delta)>0$ such that
\begin{align}\label{LRevent1}
    \sup_{\delta\sqrt{T}\le a-b \le \delta^{-1}\sqrt{T}} \Pr_{\m{IRW}}^{T;(a,b)}\big(\m{Left}_T(M)\big) \ge 1-\e.
\end{align}
for all large enough $T$. Hence with this choice of $M$ we have
\begin{align}\label{leftt}
    \sup_{\delta\sqrt{T}\le a-b \le \delta^{-1}\sqrt{T}}\left|{\Pr}_{\m{IRW}}^{T;(a,b)}(\mathsf{A}_N)-{\Pr}_{\m{IRW}}^{T;(a,b)}(\mathsf{A}_N\cap \m{Left}_T(M))\right| \le \e,
\end{align}
for sufficiently large $T$.  Recall from Lemma \ref{lem:IRWshift}  that the conditional law of $(L_1(2k-1)-L_1(1))_{k=1}^T$ under $\Pr_{\m{IRW}}^{T;(a,b)}$ given $L_1(1)$ is $\Pr_{W_T}^{T;(0,\m{h});(a',b')}$ (defined in \eqref{wscint}), where we have abbreviated $a'=a - L_1(1), b'=b-L_1(1)$. Thus by the tower property,
\begin{align*}
    {\Pr}_{\m{IRW}}^{T;(a,b)}(\mathsf{A}_N\cap \m{Left}_N(M)) & = {\Ex}_{\m{IRW}}^{T;(a,b)}\left[\ind_{\m{Left}_T(M)} \cdot {\Ex}_{\m{IRW}}^{T;(a,b)}\left[\ind_{\m{A}_N}\mid L_1(1)\right]\right] \\ & = \Ex_{\m{IRW}}^{T;(a,b)}\left[\ind_{\m{Left}_T(M)}\cdot\Pr_{W_T}^{T;(0,\m{h});(a',b')}\big((S_1(k))_{k=1}^r \in A\big)\right]
\end{align*}
Observe that on the event $\m{Left}_T(M)$ we have $a'-b' \ge \delta \sqrt{T}$ and $|a'|+|b'| \le (2M+\delta^{-1})\sqrt{T}$. Furthermore, by Theorem \ref{thm:conv}, for large enough $N$ we have 
\begin{align*}
  \sup_{a'-b' \ge \delta \sqrt{T}, |a'|+|b'| \le (2M+\delta^{-1})\sqrt{T}}  \left|\Pr_{W_T}^{T;(0,\m{h});(a',b')}\big((S_1(k))_{k=1}^r \in A\big)-\Pr\left(\big(S_1^\uparrow(k)\big)_{k=1}^r \in A\right)\right|\le \e.
\end{align*}
 Let us write $\rho_{A}:=\Pr\big(\big(S_1^\uparrow(k)\big)_{k=1}^r \in A\big)$. Combining the above estimate with \eqref{leftt}, \eqref{AGood}, and \eqref{eq:LEIRW} we get that
\begin{align*}
    \left|\Pr(\m{A}_N)- \rho_{A} \cdot \Ex\left[\ind_{\m{Good}_N(\delta)} \Pr_{\m{IRW}}^{T;(a,b)}(\m{Left}_T(M))\right]\right|\le 4\e.
\end{align*}
for large enough $N$. Combining the above estimate with the probability estimates in \eqref{good>1-e} and \eqref{LRevent1} leads to \eqref{eq:mainthm}. All we are left to show is \eqref{eq:LEIRW}. Towards this end, by the relation in \eqref{def:Gibbslaw}, we have 
    \begin{equation}\label{lem5.1gibbs}
     \mathbb{P}_{\Phi}^{(a,b,\vec{c})}(\mathsf{A}_N) = \frac{\Ex_{\mathsf{IRW}}^{T;(a,b)}[e^{-\mathcal{H}(\vec{c})}\mathbf{1}_{\mathsf{A}_N}]}{\Ex_{\mathsf{IRW}}^{T;(a,b)}[e^{-\mathcal{H}(\vec{c})}]},
    \end{equation}
    where $\mathcal{H}(\vec{c})$ is defined in \eqref{def:Gibbslaw}.
    Define a new event
    \[
    \mathsf{High}^2_N(\delta) = \left\{\inf_{k\in\llbracket 1,2T\rrbracket} \left(L_2(k)-b\right) \geq -\tfrac{\delta}{4}N^{1/3}\right\}.
    \]
    Lemma \ref{lem:IRWupperbd} implies that for sufficiently small $\delta$ we have
    \begin{equation*}
    \inf_{\delta\sqrt{T} \le a-b \le \delta^{-1}\sqrt{T}} \Pr_{\mathsf{IRW}}^{T;(a,b)}\left(\mathsf{High}^2_N(\delta)\right) > 1-\e.
    \end{equation*}
   for all large enough $T$. Note that on the event $\mathsf{High}^2_N(\delta)$, we have
    \[
    \inf_{k\in\llbracket 1,T-1\rrbracket } \left(\min\{L_2(2k-1), L_2(2k+1)\} - c_k\right) \geq \tfrac{\delta}{2}N^{1/3},
    \]
    for all $(a,b,\vec{c})\in D_N(\delta)$. In particular, this implies all exponents in the weight $\mathcal{H}$ are bounded above by $-\frac{\delta}{2}N^{1/3}$, which forces $e^{-\mathcal{H}(\vec{c})} \geq \exp(-2N^{1/3}e^{-\frac{\delta}{2}N^{1/3}}) = 1-o_N(1)$ on the event $\m{High}^2_N(\delta)$. Thus, for any $\eta>0$, 
    \[
    \sup_{(a,b,\vec{c})\in D_N(\delta)} \Pr_{\mathsf{IRW}}^{T;(a,b)}\left(|e^{-\mathcal{H}(\vec{c})}-1| < \eta \right) \geq 1-C\varepsilon,
    \]
    for large enough $N$. In other words, $e^{-\mathcal{H}(\vec{c})}$ is close to $1$ with high probability. Using this inequality in \eqref{lem5.1gibbs}, a straightforward computation leads to \eqref{eq:LEIRW} by readjusting $\e$. 
\end{proof}

\section{Tightness of the half-space log-gamma line ensemble}\label{sec6}

In this section, we prove Theorem \ref{p.tight3}. Towards this end, we first establish endpoint tightness in Section \ref{seca1} and then conclude process-level tightness in Section \ref{seca2}. Before going into the details, we introduce certain multilevel versions of softly non-intersecting random walk bridges and $\m{IRW}$ that will appear in the proof.

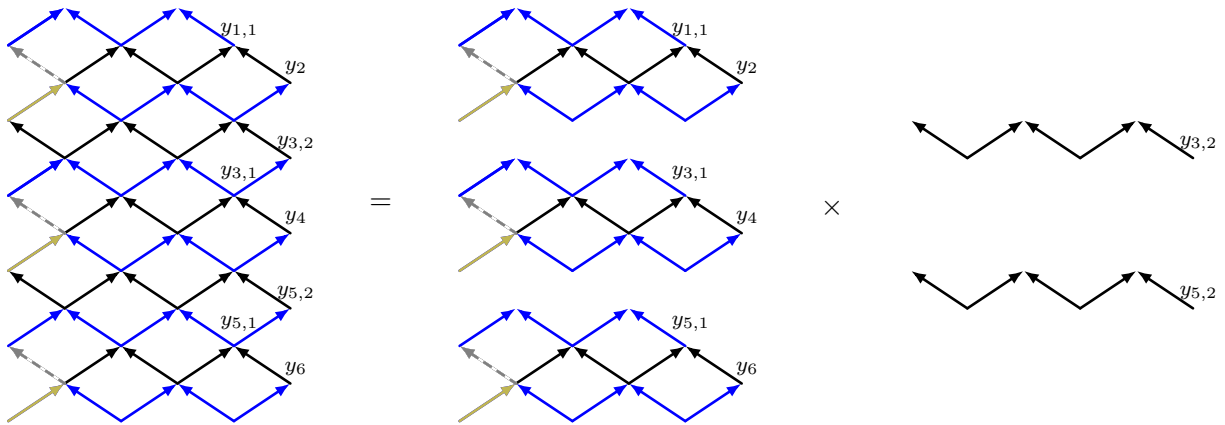
\begin{figure}[h!]
    \centering
    \begin{tikzpicture}[line cap=round,line join=round,>=triangle 45,x=1.5cm,y=1cm]
				\foreach \x in {0,1}
				{
				
					\draw[line width=1pt,blue,{Latex[length=2mm]}-]  (\x,0) -- (\x-0.5,-0.5);
					\draw[line width=1pt,blue,{Latex[length=2mm]}-] (\x,0) -- (\x+0.5,-0.5);
					\draw[line width=1pt,black,{Latex[length=2mm]}-] (\x-0.5,-0.5) -- (\x,-1);
					\draw[line width=1pt,black,{Latex[length=2mm]}-] (\x+0.5,-0.5) -- (\x,-1);
					\draw[line width=1pt,blue,{Latex[length=2mm]}-]  (\x,-1) -- (\x-0.5,-1.5);
					\draw[line width=1pt,blue,{Latex[length=2mm]}-] (\x,-1) -- (\x+0.5,-1.5);
					\draw[line width=1pt,black,{Latex[length=2mm]}-] (\x-0.5,-1.5) -- (\x,-2);
					\draw[line width=1pt,black,{Latex[length=2mm]}-] (\x+0.5,-1.5) -- (\x,-2);
					\draw[line width=1pt,blue,{Latex[length=2mm]}-]  (\x,-2) -- (\x-0.5,-2.5);
					\draw[line width=1pt,blue,{Latex[length=2mm]}-] (\x,-2) -- (\x+0.5,-2.5);
					\draw[line width=1pt,black,{Latex[length=2mm]}-] (\x-0.5,-2.5) -- (\x,-3);
					\draw[line width=1pt,black,{Latex[length=2mm]}-] (\x+0.5,-2.5) -- (\x,-3);
					\draw[line width=1pt,blue,{Latex[length=2mm]}-]  (\x,-3) -- (\x-0.5,-3.5);
					\draw[line width=1pt,blue,{Latex[length=2mm]}-] (\x,-3) -- (\x+0.5,-3.5);
					\draw[line width=1pt,blue,{Latex[length=2mm]}-]  (\x,-4) -- (\x-0.5,-4.5);
					\draw[line width=1pt,blue,{Latex[length=2mm]}-] (\x,-4) -- (\x+0.5,-4.5);
                    \draw[line width=1pt,blue,{Latex[length=2mm]}-]  (\x,-5) -- (\x-0.5,-5.5);
					\draw[line width=1pt,blue,{Latex[length=2mm]}-] (\x,-5) -- (\x+0.5,-5.5);  
                    \draw[line width=1pt,black,{Latex[length=2mm]}-] (\x-0.5,-3.5) -- (\x,-4);
					\draw[line width=1pt,black,{Latex[length=2mm]}-] (\x+0.5,-3.5) -- (\x,-4);
                    \draw[line width=1pt,black,{Latex[length=2mm]}-] (\x-0.5,-4.5) -- (\x,-5);
					\draw[line width=1pt,black,{Latex[length=2mm]}-] (\x+0.5,-4.5) -- (\x,-5);
				}  
    \foreach \x in {0,1,2,3,4}
				{
    \draw[line width=1pt,blue,{Latex[length=2mm]}-]  (2,\x-5) -- (1.5,\x-5.5);
    \draw[line width=1pt,black,{Latex[length=2mm]}-]  (1.5,\x-4.5) -- (2,\x-5);
    }
			\draw[line width=1pt,white,{Latex[length=2mm]}-] (-0.5,-2.5) -- (0,-3);
   \draw[line width=1pt,gray,{Latex[length=2mm]}-,dashed] (-0.5,-2.5) -- (0,-3);
   \draw[line width=1pt,white,{Latex[length=2mm]}-] (-0.5,-4.5) -- (0,-5);
   \draw[line width=1pt,gray,{Latex[length=2mm]}-,dashed] (-0.5,-4.5) -- (0,-5);
				\draw[line width=1pt,white,{Latex[length=2mm]}-] (-0.5,-0.5) -- (0,-1);
				
				\draw[line width=1pt,blue,{Latex[length=2mm]}-] (0,-2) -- (-0.5,-2.5);
				\draw[line width=1pt,yellow!70!black,{Latex[length=2mm]}-] (0,-3) -- (-0.5,-3.5);
                \draw[line width=1pt,yellow!70!black,{Latex[length=2mm]}-] (0,-5) -- (-0.5,-5.5);
				\draw[line width=1pt,gray,{Latex[length=2mm]}-,dashed] (-0.5,-0.5) -- (0,-1);
				\draw[line width=1pt,blue,{Latex[length=2mm]}-] (0,0) -- (-0.5,-0.5);
				\draw[line width=1pt,yellow!70!black,{Latex[length=2mm]}-] (0,-1) -- (-0.5,-1.5);
    \foreach \x in {4,5}
				{
				
					\draw[line width=1pt,blue,{Latex[length=2mm]}-]  (\x,0) -- (\x-0.5,-0.5);
					\draw[line width=1pt,blue,{Latex[length=2mm]}-] (\x,0) -- (\x+0.5,-0.5);
					\draw[line width=1pt,black,{Latex[length=2mm]}-] (\x-0.5,-0.5) -- (\x,-1);
					\draw[line width=1pt,black,{Latex[length=2mm]}-] (\x+0.5,-0.5) -- (\x,-1);
					\draw[line width=1pt,blue,{Latex[length=2mm]}-]  (\x,-1) -- (\x-0.5,-1.5);
					\draw[line width=1pt,blue,{Latex[length=2mm]}-] (\x,-1) -- (\x+0.5,-1.5);
					\draw[line width=1pt,blue,{Latex[length=2mm]}-]  (\x,-2) -- (\x-0.5,-2.5);
					\draw[line width=1pt,blue,{Latex[length=2mm]}-] (\x,-2) -- (\x+0.5,-2.5);
					\draw[line width=1pt,black,{Latex[length=2mm]}-] (\x-0.5,-2.5) -- (\x,-3);
					\draw[line width=1pt,black,{Latex[length=2mm]}-] (\x+0.5,-2.5) -- (\x,-3);
					\draw[line width=1pt,blue,{Latex[length=2mm]}-]  (\x,-3) -- (\x-0.5,-3.5);
					\draw[line width=1pt,blue,{Latex[length=2mm]}-] (\x,-3) -- (\x+0.5,-3.5);
					\draw[line width=1pt,blue,{Latex[length=2mm]}-]  (\x,-4) -- (\x-0.5,-4.5);
					\draw[line width=1pt,blue,{Latex[length=2mm]}-] (\x,-4) -- (\x+0.5,-4.5);
                    \draw[line width=1pt,blue,{Latex[length=2mm]}-]  (\x,-5) -- (\x-0.5,-5.5);
					\draw[line width=1pt,blue,{Latex[length=2mm]}-] (\x,-5) -- (\x+0.5,-5.5);  
                    \draw[line width=1pt,black,{Latex[length=2mm]}-] (\x-0.5,-4.5) -- (\x,-5);
					\draw[line width=1pt,black,{Latex[length=2mm]}-] (\x+0.5,-4.5) -- (\x,-5);
				}  
    \foreach \x in {0,2}
				{
    \foreach \y in {2,3,4}
    {
    \draw[line width=1pt,black,{Latex[length=2mm]}-]  (\y+5.5,\x-3.5) -- (\y+6,\x-4);
    }
    \foreach \y in {2,3}
    {
    \draw[line width=1pt,black,{Latex[length=2mm]}-]  (\y+6.5,\x-3.5) -- (\y+6,\x-4);
    }
    }
    \foreach \x in {0,2,4}
				{
    \draw[line width=1pt,blue,{Latex[length=2mm]}-]  (6,\x-5) -- (5.5,\x-5.5);
    \draw[line width=1pt,black,{Latex[length=2mm]}-]  (5.5,\x-4.5) -- (6,\x-5);
    }
			\draw[line width=1pt,white,{Latex[length=2mm]}-] (3.5,-2.5) -- (4,-3);
   \draw[line width=1pt,gray,{Latex[length=2mm]}-,dashed] (3.5,-2.5) -- (4,-3);
   \draw[line width=1pt,white,{Latex[length=2mm]}-] (3.5,-4.5) -- (4,-5);
   \draw[line width=1pt,gray,{Latex[length=2mm]}-,dashed] (3.5,-4.5) -- (4,-5);
				\draw[line width=1pt,white,{Latex[length=2mm]}-] (3.5,-0.5) -- (4,-1);
				
				\draw[line width=1pt,blue,{Latex[length=2mm]}-] (4,-2) -- (3.5,-2.5);
				\draw[line width=1pt,yellow!70!black,{Latex[length=2mm]}-] (4,-3) -- (3.5,-3.5);
                \draw[line width=1pt,yellow!70!black,{Latex[length=2mm]}-] (4,-5) -- (3.5,-5.5);
				\draw[line width=1pt,gray,{Latex[length=2mm]}-,dashed] (3.5,-0.5) -- (4,-1);
				\draw[line width=1pt,blue,{Latex[length=2mm]}-] (4,0) -- (3.5,-0.5);
				\draw[line width=1pt,yellow!70!black,{Latex[length=2mm]}-] (4,-1) -- (3.5,-1.5);
    \begin{scriptsize}
        \draw (1.55,-0.1) node[anchor=north] {$y_{1,1}$};
        \draw (1.55,-2) node[anchor=north] {$y_{3,1}$};
        \draw (1.55,-4) node[anchor=north] {$y_{5,1}$};
        \draw (2.05,-1.6) node[anchor=north] {$y_{3,2}$};
        \draw (2.05,-3.6) node[anchor=north] {$y_{5,2}$};
        \draw (2.05,-0.6) node[anchor=north] {$y_{2}$};
        \draw (2.05,-2.6) node[anchor=north] {$y_{4}$};
        \draw (2.05,-4.6) node[anchor=north] {$y_{6}$};
        \draw (5.55,-0.1) node[anchor=north] {$y_{1,1}$};
        \draw (5.55,-2) node[anchor=north] {$y_{3,1}$};
        \draw (5.55,-4) node[anchor=north] {$y_{5,1}$};
        \draw (6.05,-0.6) node[anchor=north] {$y_{2}$};
        \draw (6.05,-2.6) node[anchor=north] {$y_{4}$};
        \draw (6.05,-4.6) node[anchor=north] {$y_{6}$};
        \draw (10.05,-1.6) node[anchor=north] {$y_{3,2}$};
        \draw (10.05,-3.6) node[anchor=north] {$y_{5,2}$};
    \end{scriptsize}
    \draw (2.8,-2.4) node[anchor=north] {$=$};
    \draw (6.8,-2.4) node[anchor=north] {$\times$};
			\end{tikzpicture}
   \caption{Graphical structure for the $m$-$\m{IRW}$ law with $m=3$ shown on the left. It can be decomposed into three independent $\m{IRW}$s as shown in the middle along with a Radon--Nikodym derivative $W_{\m{block}}$ coming from the black edges shown on the right.}
    \label{fig:enter-label}
\end{figure}

\begin{definition}\label{mirw} We define the $m$-$\m{IRW}$ law of length $T$ with boundary conditions
\begin{align}\label{newvec}
    \vec{y}:=(y_{1,1},y_2,y_{3,1},y_{3,2},y_4, y_{5,1},y_{5,2},\ldots, y_{2m})
\end{align}
to be the $\hslg$ Gibbs measure on the domain 
\begin{align}
    \label{lambdamt}
    \Lambda_{m,T}:=\{(i,j) : i\in \ll1,2m\rr, j\in \ll1,2T+\ind_{i\;\m{even}}-2\rr\}
\end{align} with boundary conditions $u_{2k-1,2T-1}=y_{2k-1,1}, u_{2k,2T}=y_{2k}, u_{2k-1,2T}=y_{2k-1,2}$ for $k\in \ll2,m\rr$, and  $u_{2m+1,2j}=-\infty$ for $j\in \ll1,T\rr$. We denote the law of this measure by $\Pr_{m}^{T;\vec{y}}$. Note that the boundary data of $m$-$\m{IRW}$ is an element of $\R^{3m-1}$. {In the following we will always write such boundary conditions as in \eqref{newvec}}. 
\end{definition}

\begin{remark}
    We remark that $m$-$\m{IRW}$ is absolutely continuous w.r.t.~the law of $m$ independent $\m{IRW}$s. Indeed, we have
\begin{align}\label{block}
    \Pr_{m}^{T;\vec{y}}(\m{A})= \frac{\Ex_{\m{IRW}^{(m)}}^{T;\vec{y}}[W_{\m{block}}\ind_{\m{A}}]}{\Ex_{\m{IRW}^{(m)}}^{T;\vec{y}}[W_{\m{block}}]}
\end{align}
where $$W_{\m{block}}=\exp\bigg(-\sum_{i=1}^{m-1}e^{y_{2i+1,2}-L_{2i}(2T-1)}-\sum_{i=1}^{m-1}\sum_{j=1}^{T-1} (e^{L_{2i+1}(2j)-L_{2i}(2j-1)}+e^{L_{2i+1}(2j)-L_{2i}(2j+1)})\bigg),$$ 
and $\Pr_{\m{IRW}^{(m)}}^{T;\vec{y}}$ denotes the joint law of $m$ independent $\m{IRW}$s of length $T$ with boundary conditions $(y_{2i-1,1},y_{2i})$ for $i\in \ll1,m\rr$ (see Figure \ref{fig:enter-label}).
\end{remark}

For $m\ge 2$, we shall also be interested in the $\hslg$ Gibbs measure on the domain
\begin{align}
\label{fre}\Upsilon=\Upsilon(m):=\big\{(i,j)\in \Z_{\ge 1}^2 : i\in \{2m-1,2m\}, j\in \ll1,2T+\ind_{i\;\m{even}}-2 \rr\big\}    
\end{align}
with boundary conditions $u_{2m-1,2T-1}=a$, $u_{2m,2T}=b$, $u_{2m-2,2j-1}=c_j$, and $u_{2m+1,2j}=-\infty$ for $j\in \ll1,T\rr$. We shall denote this measure by $\Pr_{\Upsilon}^{T;(a,b,\vec{c})}$. Here we have suppressed the dependency on $m$ from the notation. As shown in Figure \ref{fig:enter-label2},  this law can also be viewed as an $\m{IRW}$ law hit with a Radon--Nikodym derivative. Indeed, just like \eqref{def:Gibbslaw}, here we have
\begin{align}\label{def:Glaw2}
   \Pr_{\Upsilon}^{T;(a,b,\vec{c})}(\m{A})= \frac{\Ex_{\m{IRW}}^{T;(a,b)}[V\ind_{\m{A}}]}{\Ex_{\m{IRW}}^{T;(a,b)}\left[V\right]}, \quad V:=V(\vec{c}):=\exp\bigg(- \sum_{j=1}^{T-1}(e^{L_{2m-1}(2j)-c_j}+e^{L_{2m-1}(2j)-c_{j+1}})\bigg).
\end{align}

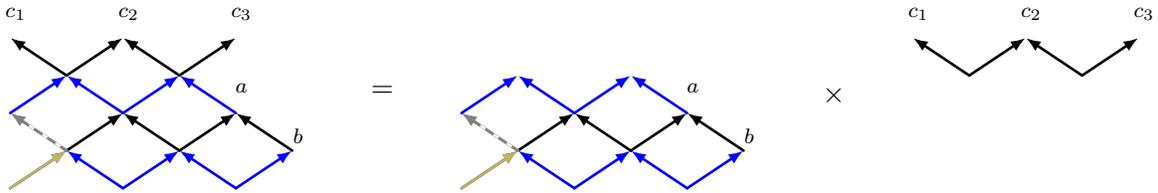
\begin{figure}[h!]
    \centering
    \begin{tikzpicture}[line cap=round,line join=round,>=triangle 45,x=1.5cm,y=1cm]
				\foreach \x in {0,1}
				{
					\draw[line width=1pt,blue,{Latex[length=2mm]}-]  (\x,-4) -- (\x-0.5,-4.5);
					\draw[line width=1pt,blue,{Latex[length=2mm]}-] (\x,-4) -- (\x+0.5,-4.5);
                    \draw[line width=1pt,blue,{Latex[length=2mm]}-]  (\x,-5) -- (\x-0.5,-5.5);
					\draw[line width=1pt,blue,{Latex[length=2mm]}-] (\x,-5) -- (\x+0.5,-5.5);  
                    \draw[line width=1pt,black,{Latex[length=2mm]}-] (\x-0.5,-3.5) -- (\x,-4);
					\draw[line width=1pt,black,{Latex[length=2mm]}-] (\x+0.5,-3.5) -- (\x,-4);
                    \draw[line width=1pt,black,{Latex[length=2mm]}-] (\x-0.5,-4.5) -- (\x,-5);
					\draw[line width=1pt,black,{Latex[length=2mm]}-] (\x+0.5,-4.5) -- (\x,-5);
				}  
    \foreach \x in {0}
				{
    \draw[line width=1pt,blue,{Latex[length=2mm]}-]  (2,\x-5) -- (1.5,\x-5.5);
    \draw[line width=1pt,black,{Latex[length=2mm]}-]  (1.5,\x-4.5) -- (2,\x-5);
    }
			\draw[line width=1pt,white,{Latex[length=2mm]}-] (-0.5,-2.5) -- (0,-3);
   \draw[line width=1pt,white,{Latex[length=2mm]}-] (-0.5,-4.5) -- (0,-5);
   \draw[line width=1pt,gray,{Latex[length=2mm]}-,dashed] (-0.5,-4.5) -- (0,-5);
                \draw[line width=1pt,yellow!70!black,{Latex[length=2mm]}-] (0,-5) -- (-0.5,-5.5);
    \foreach \x in {4,5}
				{
					\draw[line width=1pt,blue,{Latex[length=2mm]}-]  (\x,-4) -- (\x-0.5,-4.5);
					\draw[line width=1pt,blue,{Latex[length=2mm]}-] (\x,-4) -- (\x+0.5,-4.5);
                    \draw[line width=1pt,blue,{Latex[length=2mm]}-]  (\x,-5) -- (\x-0.5,-5.5);
					\draw[line width=1pt,blue,{Latex[length=2mm]}-] (\x,-5) -- (\x+0.5,-5.5);  
                    \draw[line width=1pt,black,{Latex[length=2mm]}-] (\x-0.5,-4.5) -- (\x,-5);
					\draw[line width=1pt,black,{Latex[length=2mm]}-] (\x+0.5,-4.5) -- (\x,-5);
				}  
    \foreach \x in {0}
				{
    \foreach \y in {2,3}
    {
    \draw[line width=1pt,black,{Latex[length=2mm]}-]  (\y+5.5,\x-3.5) -- (\y+6,\x-4);
    }
    \foreach \y in {2,3}
    {
    \draw[line width=1pt,black,{Latex[length=2mm]}-]  (\y+6.5,\x-3.5) -- (\y+6,\x-4);
    }
    }
    \foreach \x in {0}
				{
    \draw[line width=1pt,blue,{Latex[length=2mm]}-]  (6,\x-5) -- (5.5,\x-5.5);
    \draw[line width=1pt,black,{Latex[length=2mm]}-]  (5.5,\x-4.5) -- (6,\x-5);
    }
			
   \draw[line width=1pt,white,{Latex[length=2mm]}-] (3.5,-4.5) -- (4,-5);
   \draw[line width=1pt,gray,{Latex[length=2mm]}-,dashed] (3.5,-4.5) -- (4,-5);
				
                \draw[line width=1pt,yellow!70!black,{Latex[length=2mm]}-] (4,-5) -- (3.5,-5.5);
    \begin{scriptsize}
        \draw (1.55,-4) node[anchor=north] {$a$};
         \draw (5.55,-4) node[anchor=north] {$a$};
        \draw (1.55,-3) node[anchor=north] {$c_3$};
        \draw (0.55,-3) node[anchor=north] {$c_2$};
        \draw (-0.45,-3) node[anchor=north] {$c_1$};
        \draw (9.55,-3) node[anchor=north] {$c_3$};
        \draw (8.55,-3) node[anchor=north] {$c_2$};
        \draw (7.55,-3) node[anchor=north] {$c_1$};
        \draw (2.05,-4.6) node[anchor=north] {$b$};
        \draw (6.05,-4.6) node[anchor=north] {$b$};
    \end{scriptsize}
    \draw (2.8,-4) node[anchor=north] {$=$};
    \draw (6.8,-4) node[anchor=north] {$\times$};
			\end{tikzpicture}
   \caption{The $\Pr_{\Upsilon}^{T;(a,b,\vec{c})}$ measure with $T=3$. 
   It can be decomposed into an $\m{IRW}$ as shown in the middle times a Radon--Nikodym derivative $V$ coming from the top black edges shown on the right.}
    \label{fig:enter-label2}
\end{figure}

The $m$-$\m{IRW}$ measure arises upon conditioning the $\hslg$ line ensemble with one-sided boundary data. We can also condition on two-sided boundary data, giving rise to softly non-intersecting random walk bridges. The following proposition says that such bridges converge weakly to non-intersecting Brownian bridges under diffusive scaling provided the endpoints are separated (on the diffusive scale).

\begin{proposition}\label{sftnon} Suppose $(T_{k,i})_{k\in \{1,2\},i\in \ll1,m\rr}$ is a collection of integers satisfying $T_{2,i} > T_{1,i}+2\ge 5$ and $|T_{k,i_1}-T_{k,i_2}| \le 1$  for $k\in \{1,2\}$ and $i_1,i_2 \in \ll1,2m\rr$. Consider the $\hslg$ Gibbs measure on the domain $$\hat\Lambda:=\{(i,j) : i\in \ll1,m\rr, j\in \ll T_{1,i}+1,T_{2,i}-1\rr\}$$ 
with boundary conditions $(u_{i,j})_{(i,j)\in \partial\hat\Lambda}$ satisfying $u_{m+1,j}=-\infty$ for all $(m+1,j)\in \partial\hat\Lambda$. Assume
$$\frac{u_{i,j_1}}{\sqrt{T_{2,i}-T_{1,i}}} \to a_i \mbox{ and } \frac{u_{i,j_2}}{\sqrt{T_{2,i}-T_{1,i}}} \to b_i \mbox{ for }(i,j_1),(i,j_2)\in \partial\hat\Lambda \mbox{ with }j_1\le T_{1,i} \mbox{ and }j_2\ge T_{2,i},$$
as $T_{2,i}-T_{1,i} \to \infty.$
Suppose further that $a_1>a_2>\cdots>a_{m}$ and $b_1>b_2>\cdots>b_m$. Then we have
$$\bigg(\frac{L_i\big(T_{1,i}+x(T_{2,i}-T_{1,i})\big)}{\sqrt{T_{2,i}-T_{1,i}}}\bigg)_{i=1}^m \stackrel{d}{\longrightarrow} (B_i(x))_{i=1}^m$$
in the topology of uniform convergence on $C([0,1], \R^m)$. Here $(B_i)_{i=1}^m$ are $m$ Brownian bridges starting from $(a_1,\ldots, a_m)$ and ending at $(b_1,\ldots,b_m)$ conditioned not to intersect.
\end{proposition}

The above proposition for the case of (truly) non-intersecting random walk bridges essentially appears as Lemma 3.10 in \cite{serio}. The same proof goes through under soft non-intersection as well upon minor modification. We skip the details for brevity, but we refer to the proof of Proposition \ref{invarprin1} for a special case of the argument which illustrates how to deal with the soft non-intersection.

\subsection{Endpoint tightness} \label{seca1} The goal of this section is to show that the left endpoint of the line ensemble is tight (see Theorem \ref{thm:eptight} for precise statement). To begin with, we first claim that there are points on the $m$th curve which are at height $O(N^{1/3})$.

\begin{proposition}[High points on the $m$th curve] \label{p:high2k} Fix any $\e\in (0,1)$ and $m,k>0$. There exists $R_0(m,k,\e)>0$ such that for all $R\ge R_0$,
		\begin{align}\label{e:thigh}
			\liminf_{N\to \infty}\Pr\left(\sup_{p\in [kN^{2/3},RN^{2/3}]}\mathcal{L}_{m}^N(2p) \ge -\big(\tfrac18R^2\nu+2\sqrt{R}\big)N^{1/3} \right)>1-\e,
		\end{align}
		where $\nu:=\frac{(\Psi'(\theta))^2}{(-\Psi''(\theta))^{4/3}}.$
\end{proposition}
The $m=2$ case for the above proposition is Theorem 3.3 in \bcd. The strategy of our proof follows the same idea as in \bcd, so we will be brief.

	\begin{proof}[Proof of Theorem \ref{p:high2k}] For clarity we divide the proof into several steps.
		
		\medskip
		
		\noindent\textbf{Step 1.} In this step we define the notation and events used in the proof. Fix $\e\in (0,1)$ and $k>0$. By Proposition 3.4 in \bcd, there exists $M_0$ such that for all $Q>0$ we have
         \begin{align*}
	\Pr(\m{C}) \ge 1-\e, \quad 		\mathsf{C}:=\bigg\{ \sup_{p\in \ll QN^{2/3},(M_0+2Q)N^{2/3}\rr} {\L^N_1(2p+1)}+\SSS^2\nu N^{1/3}\ge -M_0N^{1/3}\bigg\}
		\end{align*}
  for all large enough $N$. We set $R$ large enough so that
		\begin{align}\label{eR}
			2^{-5}R\ge 2k+1, \quad M_0+\sqrt{R}-2^{-5}(\tfrac18R^2\nu+M_0)+R^{3/2} \le -M_0-2^{-10}R^2\nu, \quad R\ge 2M_0
		\end{align}
		and $\SSS:=2^{-5}R$. 
		We will assume some additional conditions on $R$ later, which will depend on certain probability bounds that will be specified in the next step. For convenience, we will also assume $kN^{2/3}$ and $RN^{2/3}$ are integers (instead of using floor functions below).  We set
		\begin{align*}
			a:=M_0N^{1/3}, \ \ b:=-\tfrac18R^2N^{1/3}\nu, \ \ n:=RN^{2/3}-kN^{2/3}, \ \ v:=-\big(\tfrac18R^2\nu+2\sqrt{R}\big)N^{1/3}.
		\end{align*}
		Let us define the sets $\mathcal{I}:=\ll \SSS N^{2/3},(M_0+2\SSS)N^{2/3}\rr$ and $\mathcal{J}:=\ll kN^{2/3},RN^{2/3}\rr$. Due to \eqref{eR}, we have $\mathcal{I} \subset \mathcal{J}$. Next, we define the following events:
		\begin{align*}
			\mathsf{A} &:=\bigg\{\sup_{p\in \mathcal{J}}{\mathcal{L}_{m}^N(2p)} \le v  \bigg\}, \quad \mathsf{B}_1 :=\bigg\{ \mathcal{L}_1^N(2kN^{2/3}+1) \le a,\mathcal{L}_1^N(2RN^{2/3}+1) \le b \bigg\}, \\
   \m{B}_2 & := \bigcap_{i=2}^{2m-1}\bigg\{  \L_{i}^N(2kN^{2/3}+1) \le a+2m(\log N)^2, \L_{i}^N(2RN^{2/3}+1) \le b+2m(\log N)^2 \bigg\}.
		\end{align*}
		Set $\m{B}:=\m{B}_1\cap \m{B}_2$.  By Propositions 3.4 and 3.5 in \bcd, we have $\Pr(\neg \m{B}_1) \le \e$ for all large enough $R$. On the other hand, by Proposition \ref{pp:order}, $\Pr(\m{B}_1 \cap \neg\m{B}_2) \le \e$. We claim that
		for all large enough $R$ we have
		\begin{align}\label{e:tosh}
			\Pr(\m{A}\cap\m{B}\cap \m{C}) \le \e.
		\end{align}
		We prove \eqref{e:tosh} in the subsequent steps. Assuming this, note that by union bound we have
		\begin{align*}
			\Pr(\neg \m{A}) \ge \Pr(\m{C})-\Pr(\neg \m{B})-\Pr(\m{A}\cap\m{B}\cap \m{C}) \ge 1-4\e.
		\end{align*}
		Changing $\e \mapsto \e/4$ we arrive at \eqref{e:thigh}. This completes the proof modulo \eqref{e:tosh}.
		
\medskip

		\noindent\textbf{Step 2.} We consider the $\sigma$-algebra
  $$\mathcal{F}:=\sigma\big(\L^N_{m}\ll1,2N-2m+2\rr, \L^N_{i}\ll1,2kN^{2/3}+1 \rr, \mathcal{L}_i^N\ll 2RN^{2/3}+1,2N-2i+2\rr) : i\in \ll1,m-1\rr\big).$$ Note that $\m{A}\cap \m{B} \in \mathcal{F}$. Hence $\Pr(\m{A}\cap \m{B}\cap \m{C})=\Ex\left[\ind_{\m{A}\cap\m{B}} \Ex\left[\ind_{\m{C}}\mid \mathcal{F}\right]\right].$ Using the Gibbs property we have $\Ex\left[\ind_{\m{C}}\mid \mathcal{F}\right] = \Pr_{\m{Gibbs}}^{\vec{x},\vec{y};\vec{z}}(\m{C}),$
where $\Pr_{\m{Gibbs}}^{\vec{x},\vec{y};\vec{z}}$ denotes the $\hslg$ Gibbs measure on the domain $\Lambda=\ll1,m-1\rr \times  \ll 2kN^{2/3}+2,2RN^{2/3}\rr$ with boundary conditions
\begin{align*}
   &  x_i:=\L_i^N(2kN^{2/3}+1), \, y_i:=\L_i^N(2RN^{2/3}+1) \mbox{ for } i\in \ll1,m-1\rr,  \\ & z_j:=\L_{m}^N(2kN^{2/3}+2j)  \mbox{ for } j\in \ll1,n\rr.
\end{align*}
  Observe that  on $\m{A}\cap \m{B}$,
\begin{align*}
    & x_i\le a_i:=a+\frac1{m}(m-i+1) \sqrt{n}, \quad 
    y_i \le b_i:=b+\frac1{m}(m-i+1)\sqrt{n}, \quad z_j \le v.
\end{align*}
By stochastic monotonicity the probability of the event $\m{C}$ increases as we increase the boundary data. Thus
\begin{equation}
\begin{split}
\ind_{\m{A}\cap\m{B}}\cdot\Pr_{\m{Gibbs}}^{\vec{x},\vec{y};\vec{z}}(\m{C}) & \le \ind_{\m{A}\cap\m{B}}\cdot\Pr_{\m{Gibbs}}^{\vec{a},\vec{b};v}(\m{C}) \\ & =\ind_{\m{A}\cap\m{B}}\cdot\frac{\Ex_{\m{Gibbs}}^{\vec{a},\vec{b};(-\infty)^n}\hspace{-0.1cm}\left[\exp\left(-2\displaystyle\sum_{j=kN^{2/3}+1}^{RN^{2/3}-1}e^{v-L_{m-1}(2j+1)}\right)\ind_{\m{C}}\right]}{\Ex_{\m{Gibbs}}^{\vec{a},\vec{b};(-\infty)^n}\left[\exp\left(-2\displaystyle\sum_{j=kN^{2/3}+1}^{RN^{2/3}-1}e^{v-L_{m-1}(2j+1)}\right)\right]}.\label{wcd}
\end{split}
\end{equation}
By Proposition \ref{sftnon}, we know under $\Pr_{\m{Gibbs}}^{\vec{a},\vec{b};(-\infty)^n}$

$$\left(L_i(2kN^{2/3}+nt+1)/\sqrt{n}\right)_{i=1}^{m-1} \stackrel{d}{\longrightarrow} (B_i(t))_{i=1}^{m-1}$$  as $N\to\infty$, where $B_i$ are non-intersecting Brownian bridges on $[0,2]$ with $B_i(0)=M_0/\sqrt{R-k}+(m-i+1)/m$ and $B_i(2)=-\frac18R^2\nu/\sqrt{R-k} +(m-i+1)/m$. Since $v/\sqrt{n} < B_{m-1}(0), B_{m-1}(2)$, there is a positive probability that $B_{m-1}(\cdot)$ stays above $v/\sqrt{n}+\delta$ for some $\delta>0$. But then for large enough $N$ we have $\Pr_{\m{Gibbs}}^{\vec{a},\vec{b};(-\infty)^n}(L_{m-1}(2j+1)-v \ge \tfrac12\delta\sqrt{n}) \ge \rho$ for some $\rho>0$. This forces
\begin{align}\label{wcd1}
    \Ex_{\m{Gibbs}}^{\vec{a},\vec{b};(-\infty)^n}\left[\exp\left(-2\displaystyle\sum_{j=kN^{2/3}+1}^{RN^{2/3}-1}e^{v-L_{m-1}(2j+1)}\right)\right] \ge \frac12\rho
\end{align}
for all large enough $N$. We now claim that for all large enough $R$ and $N$,
		\begin{equation}
  \begin{split}\label{edc}
			& \m{D} \subset \neg\m{{C}}, \quad \Pr_{\m{Gibbs}}^{\vec{a},\vec{b};(-\infty)^n}(\m{D})\ge 1-\tfrac12\e\rho, \ \ \mbox{ where } \\ &  \m{D}:=\left\{\sup_{i\in \ll 1,2n+1\rr} \left(L_1(i)-a_1-\tfrac{(i-1)(b_1-a_1)}{2n} \right) \le \sqrt{Rn} \right\}.
		\end{split}
  \end{equation}
		Note that \eqref{edc} implies 
  
  $$\Ex_{\m{Gibbs}}^{\vec{a},\vec{b};(-\infty)^n}\left[\exp\left(-2\displaystyle\sum_{j=kN^{2/3}+1}^{RN^{2/3}-1}e^{v-L_{m-1}(2j+1)}\right)\ind_{\m{C}}\right] \le \Pr_{\m{Gibbs}}^{\vec{a},\vec{b};(-\infty)^n}(\mathsf{C}) \le \tfrac12\e\rho.$$ Plugging this back in \eqref{wcd} along with the bound in \eqref{wcd1} yields that the r.h.s.~of \eqref{wcd} is at most $\e$. This proves \eqref{e:tosh}.	To verify \eqref{edc}, simply note that $\Pr_{\m{Gibbs}}^{\vec{a},\vec{b};(-\infty)^n}(\m{D})$ can be made arbitrarily close to $1$ by choosing $R$ and $N$ large enough due to the weak convergence from Proposition \ref{sftnon}.
		Let us now verify $\m{D}\subset \neg\m{C}$. For $q\ge \SSS$ we see that
		\begin{align*}
			a_1+\tfrac{(q-k)(b_1-a_1)}{R-k}+\sqrt{Rn} & \le \left(M_0+\sqrt{R} - \tfrac{\SSS-k}{R-k}(\tfrac18R^2\nu+M_0)+R^{3/2}\right)N^{1/3}  \\ & \le \left(M_0+\sqrt{R}  - 2^{-5}(\tfrac18R^2\nu+M_0)+R^{3/2}\right)N^{1/3} \le -\left(M_0 +\SSS^2\nu\right)N^{1/3}.
		\end{align*}
		The penultimate inequality follows by observing that as $\SSS=2^{-5} R$, we have $\SSS-k \ge 2^{-5}(R-k) >0$. Te last inequality follows from \eqref{eR}. Thus for all $p\ge \SSS N^{2/3}$,
		\begin{align*}
			a_1+\tfrac{(p-kN^{2/3})(b_1-a_1)}{(R-k)N^{2/3}}+\sqrt{Rn} \le -(M_0 + \SSS^2\nu) N^{1/3}.
		\end{align*}
		Clearly this implies $\m{D}\subset \neg\m{C}$, completing the proof of \eqref{edc}.
	\end{proof}

\begin{proposition} \label{p.nottoo} Fix $m\in \mathbb{Z}_{\geq 1}$ and $k,\e>0$. There exists $M=M(m,k,\e)>0$ such that
\begin{align} \label{e:highmn}
    \limsup_{N\to\infty}\Pr\big(\m{High}_{k,m,N}(M)\big) \le \e, \ \mbox{ where } \ \m{High}_{k,m,N}(M):=\bigg\{\sup_{j\in \ll 1, 2kN^{\frac23}\rr} \L_{m}^N(j)\ge M N^{\frac13}\bigg\}.
\end{align}
\begin{align} \label{e:lowmn}
    \limsup_{N\to\infty}\Pr\big(\m{Low}_{k,m,N}(M)\big) \le \e, \ \mbox{ where } \ \m{Low}_{k,m,N}(M):=\bigg\{\inf_{j\in \ll 1, 2kN^{\frac23}\rr} \L_{m}^N(j)\le -M N^{\frac13}\bigg\}.
\end{align}
\end{proposition}

\begin{proof} \eqref{e:highmn} follows easily from tightness of the top curve of the $\hslg$ line ensemble and Theorem \ref{pp:order} which forces the line ensemble to obey a certain ordering. We focus on the proof of \eqref{e:lowmn}. For convenience, we shall prove it for even-index curves: $m\mapsto 2m$.  Fix any $\e>0$.  From Proposition \ref{p:high2k} we get an $R$ so that
\begin{align*}
    \Pr\bigg(\bigcup_{p\in\ll kN^{2/3}, RN^{2/3}\rr } \m{A}_p\bigg) \ge 1-\e, \mbox{ where } \m{A}_p &:=\big\{\L_{2m}^N(2p) \ge -(\tfrac18R^2\nu+2\sqrt{R})N^{1/3}\big\}.
\end{align*}
Choose $\mathfrak{M}$ from Lemma \ref{lem:IRWupperbd} so that \eqref{e:IRWupperbd} holds. Let us set $T:=kN^{2/3}$, $M_1:=\tfrac18R^2\nu+2\sqrt{R}$,  and $\mathcal{I}:=\ll T, RN^{2/3}\rr$. We consider the disjoint decomposition of $\{\m{A}_p\}_{p\in \mathcal{I}}$ given by $\m{A}_p'  :=\m{A}_p\cap \bigcap_{q\in \ll p+1,RN^{2/3}\rr} \neg\m{A}_q$, so that $\m{A}:=\bigcup_{p\in \mathcal{I}}\m{A}_p=\bigsqcup_{p\in \mathcal{I}}\m{A}_p'$, where the latter is a union over a disjoint collection of events. For each $p\in \mathcal{I}$, we define the event
\begin{align*}
    \m{B}_p & :=\bigcap_{i=1}^m\bigg\{ \min\{\L_{2i-1}^N(2p-1) , \L_{2i-1}^N(2p)\} \ge -M_1N^{\frac13}-i(\mathfrak{M}+1)\sqrt{T}, \\ & \hspace{7cm} \L_{2i}^N(2p) \ge -M_1N^{\frac13}-i(\mathfrak{M}+2)\sqrt{T} \bigg\},
\end{align*}
and the $\sigma$-field $$\mathcal{F}_p:= \{\L_{i_1}^N\ll 2p-\ind_{i_1\;\m{odd}}, 2N-2i_1+2\rr, \L_{i_2}^N\ll 1,2N-2i_2+2\rr : i_1\in \ll1,2m\rr, i_2\in\ll 2m+1,N\rr\}.$$
Recall the event $\m{Ord}_{2m,N}$ from \eqref{t41} and write $\m{C}:= \m{Low}_{k,2m,N}(M)$ where we set $M:=M_1+4m(\mathfrak{M}+1)\sqrt{R}$. Using the disjointness of $\{\m{A}_p'\}_{p\in \mathcal{I}}$ we obtain 
\begin{align*}
    \Pr\big(\m{C}\cap \m{A}\cap \m{Ord}_{2m,N}\big)=\sum_{p\in \mathcal{I}} \Pr\big(\m{C}\cap \m{A}_p'\cap \m{Ord}_{2m,N}\big) \le \sum_{p\in \mathcal{I}} \Pr\big(\m{C}\cap \m{A}_p'\cap \m{B}_{p}\big) = \sum_{p\in \mathcal{I}}\Ex[\ind_{\m{A}_p'\cap \m{B}_p}\Ex[\ind_{\m{C}}\mid \mathcal{F}_p]].
\end{align*}
The above inequality follows by observing that
$\m{A}_p'\cap \m{Ord}_{2m,N} \subset \m{A}_p'\cap \m{B}_p$, and the final equality is due to the fact that $\m{A}_p'\cap \m{B}_p\in \mathcal{F}_p$. Invoking the Gibbs property, we have that $\Ex[\ind_{\m{C}}\mid \mathcal{F}_p]=\Pr_{m}^{T;\vec{x}}({\m{C}})$ where $\vec{x}$ is a vector of the type \eqref{newvec} with
$$x_{2i-1,1}:=\L_{2i-1}^N(2p-1), \quad x_{2i-1,2}:=\L_{2i-1}^N(2p), \quad x_{2i}:=\L_{2i}^N(2p).$$ 
By stochastic monotonicity,
\begin{align*}
    \ind_{\m{B}_p}\Ex[\ind_{\m{C}}\mid \mathcal{F}_p] = \ind_{\m{B}_p}\Pr_{m}^{T;\vec{x}}({\m{C}})\le \ind_{\m{B}_p}\Pr_{m}^{T;\vec{y}}({\m{C}}) 
\end{align*}
where $\vec{y}_p$ is of the form \eqref{newvec} with 
$$y_{2i-1,1}=y_{2i-1,2}=-M_1N^{1/3}-i(\mathfrak{M}+1)\sqrt{p}, \quad y_{2i}=-M_1N^{1/3}-i(\mathfrak{M}+2)\sqrt{p}.$$
Let us consider the event
\begin{align*}
    \m{D}_p& :=\bigcap_{i=1}^{2m} \bigg\{\sup_{j\in \ll 1,2p-\ind_{i\;\m{odd}}\rr}|L_i(j)+M_1N^{1/3}+i(M+2-\ind_{i\;\m{odd}})\sqrt{p}| \le M\sqrt{p}\bigg\}.
\end{align*}
Note that $\neg \m{C} \supset \m{D}_p$. Using \eqref{block} we thus get
$\Pr_{m}^{T;\vec{y}_p}(\neg{\m{C}}) \ge \Pr_{m}^{T;\vec{y}_p}({\m{D}_p}) \ge \Ex_{\m{IRW}^{(m)}}^{T;\vec{y}_p}[W_{\m{block}}\ind_{\m{D}_p}].$ On the event $\m{D}_p$, $W_{\m{block}} \ge \exp(-2mpe^{-\sqrt{p}})$ and by \eqref{e:IRWupperbd}, $\Pr_{\m{IRW}^{(m)}}^{T;\vec{y}_p}({\m{D}}) \ge (1-\e)^m \ge 1-m\e$.   Thus
\begin{align*}
    \Pr\big(\m{C}\cap \m{A}\cap \m{Ord}_{2m,N}\big)\le \sum_{p\in \mathcal{I}} (m+1)\e \cdot \Pr(\m{A}_p') \le (m+1)\e,
\end{align*}
and $\Pr(\m{C})\le \Pr\big(\m{C}\cap \m{A}\cap \m{Ord}_{2m,N}\big)+\Pr(\neg\m{A})+\Pr\big(\neg\m{Ord}_{2m,N}\big) \le (m+3)\e.$ Adjusting $\e$, we get the desired result.
\end{proof}

A similar result can be proven under the $m$-$\m{IRW}$ law. We record this in the following proposition.  
\begin{proposition}\label{p.2nto} Fix $\e>0$, $m\in \Z_{\ge 1}$, and $M>0$. Suppose $\vec{x}$ is a vector of the type \eqref{newvec} with all entries within $[-M\sqrt{T},M\sqrt{T}]$. Then there exist $M_1(m,M,\e)>0$ and $T_0(m,M,\e)>0$ such that 
\begin{align}\label{6.61}
 &   \Pr_{m}^{T;\vec{x}}\bigg( \sup_{i\in \ll1,2m\rr, j\in \ll1,2T-\ind_{i\;\m{odd}}\rr} |L_{i}(j)|\le  M_1\sqrt{T}\bigg) \ge 1-\e 
\end{align}
for all $T\ge T_0$. We also have $ \Pr_m^{2T;\vec{x}}(\widetilde{\m{Ord}}_{m,T}) \ge 1-\e $ where
\begin{align*}
        \widetilde{\m{Ord}}_{m,T}& :=\bigcap_{i=1}^{2m} \bigcap_{p=1}^{T} \big\{\max (L_i(2p+1),L_i(2p-1)) \le L_i(2p)+(\log T)^{7/6}\big\} \\ & \hspace{4cm}\cap\big\{L_{i+1}(2p) \le \min (L_i(2p-1),L_i(2p+1))+(\log T)^{7/6}\big\}.
    \end{align*}  
\end{proposition}
The $m=1$ version of the first statement is already present in Lemma \ref{lem:IRWupperbd} (which relies on Lemma 5.4 in \bcd). The general $m$ case follows easily by a slight modification of the argument in Lemma \ref{lem:IRWupperbd} and Lemma 5.4 in \bcd. On the other hand, the second statement above is the Gibbs measure version of Theorem \ref{pp:order}. 
The proof of the second part is in fact contained in the proof of Theorem \ref{pp:order}, as the argument in \bcd\ proceeds by conditioning on the boundary data and then proving the ordering property under the $\hslg$ Gibbs measure. We skip the details for brevity.

	\begin{theorem}[Endpoint tightness] \label{thm:eptight} The sequences $\{{N^{-1/3}}\L_{2m-1}^N(1)\}_{N\geq 1}$ and $\{{N^{-1/3}}\L_{2m}^N(2)\}_{N\geq 1}$ are tight for all $m\in\mathbb{Z}_{\geq 1}$.
	\end{theorem}

	\begin{proof}[Proof of Theorem \ref{thm:eptight}] We know $N^{-1/3}\L_1^N(1)$ is tight via Theorem 3.10 in \bcd. In particular $\L_1^N(1)$ is upper tight. On the other hand, from Proposition \ref{p.nottoo} we know $N^{-1/3}\L_{2m}^N(2)$ is lower tight.  By Proposition \ref{pp:order}, it follows that $N^{-1/3}\L_{2k-1}^N(1)$ and $N^{-1/3}\L_{2k}^N(2)$ are tight for all $k\in \ll1,m\rr$. 
	\end{proof}

\subsection{Process-level tightness} \label{seca2} Having established pointwise tightness, we next proceed to process-level tightness of the line ensemble.

We begin with a basic lemma ascertaining that the probability of $\m{IRW}$ passing through certain regions can be uniformly bounded below.

\begin{lemma}\label{l:region} Fix any $\gamma>0$. There exists $M_0>0$ such that for all $M\ge M_0$, there exists $\phi(M,\gamma)>0$ such that
    \begin{align*}
        & \liminf_{T\to\infty}\Pr_{\m{IRW}}^{T;(0,-\sqrt{T})}\bigg(\inf_{i\in \ll1,2\rr, j\in \ll1,T\rr}L_i(j)\ge M\sqrt{T},\sup_{i\in \ll1,2\rr, j\in \ll1,2T-\ind_{i\;\m{odd}}\rr}L_i(j)\le (M+\gamma)\sqrt{T}\bigg)\ge \phi, \\ &
        \liminf_{T\to\infty}\Pr_{\m{IRW}}^{T;(0,-\sqrt{T})}\bigg(\sup_{i\in \ll1,2\rr, j\in \ll1,T\rr}L_i(j)\le -M\sqrt{T},\inf_{i\in \ll1,2\rr, j\in \ll1,2T-\ind_{i\;\m{odd}}\rr}L_i(j)\ge -(M+\gamma)\sqrt{T}\bigg)\ge \phi.
    \end{align*}
\end{lemma}

\begin{figure}
\centering
\begin{overpic}[page=1,scale=0.5]{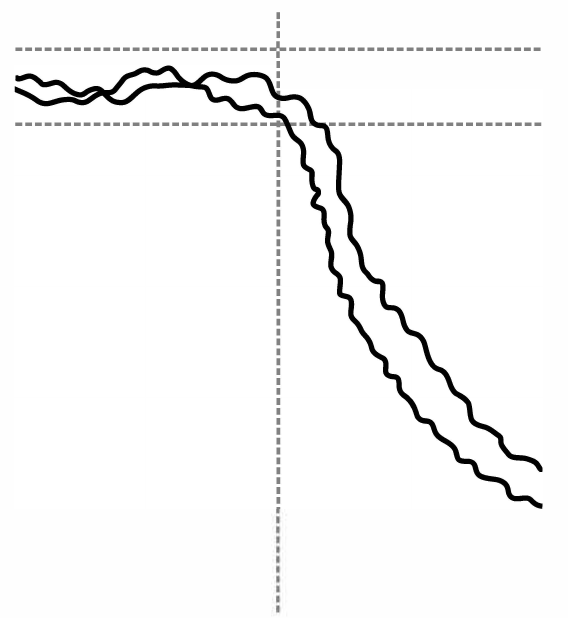}
\put(65,72){$M\sqrt{T}$}
\put(55,95){$(M+\gamma)\sqrt{T}$}
\put(38,2){$T$}
\put(5,2){$1$}
\put(70,2){$2T-1$}
\put(67,50){$L_1$}
\put(67,20){$L_2$}
\end{overpic}
\hspace{2cm}
\begin{overpic}[page=2,scale=0.5]{l6.4.pdf} \qquad
\put(10,22){$-M\sqrt{T}$}
\put(-2,-2){$-(M+\gamma)\sqrt{T}$}
\put(75,90){$L_1$}
\put(75,60){$L_2$}
\end{overpic}
\caption{Events considered in Lemma \ref{l:region}.}
\end{figure}
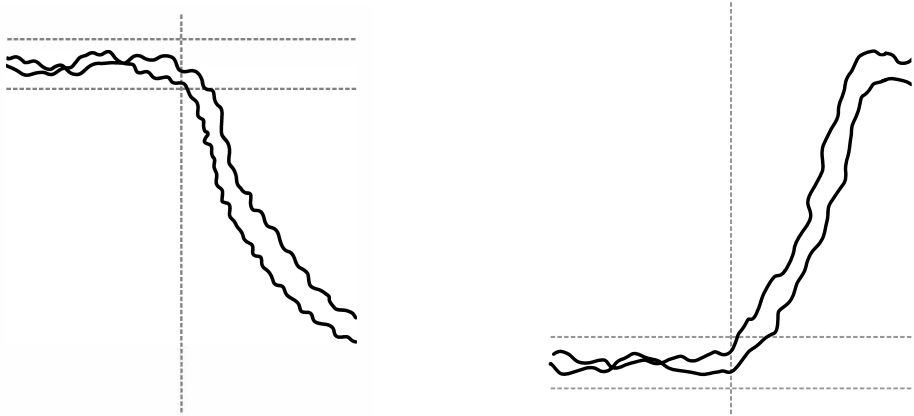

\begin{proof} Let us consider the events
\begin{align*}
    \m{A} & :=\bigcap_{i\in \ll1,2\rr, j\in \ll1,T\rr}\bigg\{L_i(j)\in [M\sqrt{T},(M+\tfrac\gamma2)\sqrt{T}]\bigg\}, \\ 
    \m{B} & :=\bigg\{\sup_{i\in \ll1,2\rr, j\in \ll 1,2T-\ind_{i\;\m{odd}}\rr}L_i(j)\le (M+\gamma)\sqrt{T}\bigg\}
\end{align*}
and the $\sigma$-field $\mathcal{F}:=\sigma\big\{L_i(j) : i\in \ll1,2\rr, j\in \ll1,T\rr\big\}$. We write $\Pr$ for $\Pr_{\m{IRW}}^{T;(0,-\sqrt{T})}$. To prove the first part of the lemma, it suffices to provide a lower bound for $\Pr(\m{A}\cap\m{B})$. By the tower property of conditional expectation, we write $\Pr(\m{A}\cap \m{B}) = \Ex[\ind_{\m{A}}\Ex[\ind_{\m{B}}\mid \mathcal{F}]].$ Note that by the Gibbs property, the above conditional law can be viewed as a $\hslg$ Gibbs measure. Since $\m{B}$ is a decreasing event w.r.t.~the boundary data, we may increase the left endpoints to $\big((M+3\gamma/4)\sqrt{T}, (M+\gamma/2)\sqrt{T}\big)$. By Proposition \ref{sftnon}, we can thus conclude $\ind_{\m{A}}\Ex[\ind_{\m{B}}\mid \mathcal{F}] \ge \ind_{\m{A}}\phi_1$ for some deterministic constant $\phi_1>0$. On the other hand, Proposition 4.1 in \bcd\ can be modified (see eq.~(4.27) in \bcd\ for a similar statement), to show $\Ex[\ind_{\m{A}}] \ge \phi_2$ for some constant $\phi_2>0$ and for all large enough $T$.
Thus combining we get that $\Pr(\m{A}\cap\m{B})\ge \phi_1\phi_2$. This proves the first inequality in Lemma \ref{l:region}. The second one follows similarly.
\end{proof}

We use the above result to prove the tightness of $m$-$\m{IRW}$ defined in Definition \ref{mirw}. Fix any $m,U\in \Z_{\ge 1}$, and for each $i\in \ll1,m\rr$ let $f_i : A_i \to \mathbb{R}$ where $\ll 1, U\rr \subset A_i \subset \mathbb{Z}$. We define the joint modulus of continuity for $f_1,\ldots,f_m$ as 
\begin{align*}
    \omega_{\delta}^T(f_1,f_2,\ldots,f_m;\ll1,U\rr)= \sup_{1\leq i\leq m} \sup_{\substack{x,y\in \ll1,U\rr \\ |x-y|\le \delta T}}|f_i(x)-f_i(y)|.
\end{align*}
With tightness of the left boundary of the $\hslg$ line ensemble in place, it suffices to show that the modulus of continuity for the line ensemble with $T=N^{2/3}$ upon dividing by $\sqrt{T}=N^{1/3}$ can be made arbitrarily small by taking $\delta\downarrow 0$. Towards this end, we first control the modulus of continuity at the level of $m$-$\m{IRW}$ in the following proposition.

\begin{proposition}\label{pimc} Fix any $m\in \mathbb{Z}_{\geq 1}$ and $M, k_1, k_2, \gamma, V ,\e >0$ with $k_1<k_2$.  
For each $R,T>0$, define the set
		\begin{align*}
I_{M} & :=\Big\{(u_{i,j})_{(i,j)\in \partial \Lambda_{m,4R}}\in \R^{|\partial\Lambda_{m,4R}|} : |u_{i,j}|\le M\sqrt{T} \mbox{ for all } (i,j)\in \partial\Lambda_{m,4R}, \\ & \hspace{4cm} u_{i,j_1}-u_{i+1,j_2}\ge -(\log T)^{7/6} \mbox{ for all } (i,j_1),(i+1,j_2)\in \partial\Lambda_{m,4R}\Big\},
		\end{align*}
where the domain $\Lambda_{m,4R}$ is defined in \eqref{lambdamt}.		There exist $\delta=\delta(M,V,k_1,k_2,\gamma,\e)>0$ and $T_0=T_0(M,V,k_1,k_2,\gamma,\e)>0$ such that for all $\vec{x}\in I_{M}$, $R\in \ll k_1T,k_2T\rr$, and $T\ge T_0$ we have	
		\begin{align*}
		\Pr_{m}^{4R;\vec{x}}(\m{A})\le \e, \mbox{ where } \m{A}:=\left\{ \big\{\omega_{\delta}^T(L_{2m-1},L_{2m}; \ll1,R\rr) \ge \gamma 	\sqrt{T}\big\}\cap \bigcap_{k=1}^m\big\{|L_{2k-1}(1)|+|L_{2k}(2)|\le V\sqrt{T}\big\}\right\}.
		\end{align*}
	\end{proposition}

\begin{proof}  For $m=1$, i.e., for $\m{IRW}$, the above result was established in Proposition 5.2 in \bcd. Our proof will rely on the $m=1$ case. We divide the proof into two steps for clarity.

\medskip

\noindent\textbf{Step 1.} Fix any $\e\in (0,1)$. First, from Proposition \ref{p.2nto} it follows that there exist $M_1,M_2$ large enough so that
\begin{equation}
    \label{events}
    \begin{aligned}
& \Pr_{m}^{4R;\vec{x}}(\m{C})  \ge 1-\e, \mbox{ where } \m{C}:=\big\{|L_{2m-1}(2R-1)|+|L_{2m}(2R)|\le M_1\sqrt{T}, \\ & \hspace{6.5cm} L_{2m-1}(2R-1)-L_{2m}(2R) \ge -(\log T)^{7/6}\big\}, \\ &
 \Pr_{m}^{4R;\vec{x}}(\m{D})  \ge 1-\e, \mbox{ where } \m{D}:=\bigg\{\inf_{j\in \ll 1,4R\rr} L_{2m-2}(j) \ge -M_2\sqrt{T}\bigg\}. 
\end{aligned}
\end{equation}
Let us consider the $\sigma$-algebra
$$\mathcal{F}=\big\{L_{2m-1}\ll 2R-1,8R-1\rr, L_{2m}\ll 2R, 8R\rr, L_i\ll1,8R+\ind_{i\;\m{even}}-2\rr : i\in \ll1,2m-2\rr \big\}.$$
Note by the Gibbs property that the conditional law of $(L_{2m-1}\ll 1,2R-2\rr, L_{2m}\ll 1,2R-1\rr)$ given $\mathcal{F}$ is $\Pr_{\Upsilon}^{R;(a,b,\vec{z})}$ defined in \eqref{def:Glaw2}. Thus we have the following representation:
\begin{align*}
    \Ex_{m}^{4R;\vec{x}}[\mathbf{1}_{\m{A}}\mid \mathcal{F}] =\Pr_{\Upsilon}^{R;(a,b,\vec{z})}(\m{A})= \frac{\Ex_{\m{IRW}}^{R;(a,b)}[V\ind_{\m{A}}]}{\Ex_{\m{IRW}}^{R;(a,b)}\left[V\right]}.
\end{align*}
Here we are abusing notation slightly: we are now using $L_{2m-1}(\cdot), L_{2m}(\cdot)$ to denote the underlying random variables in the $\Pr_{\m{IRW}}^{R;(a,b)}$ law. Let us set $\m{B}(\beta):=\big\{\Ex_{\m{IRW}}^{R;(a,b)}\left[V\right]\ge \beta\big\}$. We claim that
\begin{align}\label{acceptbot}
    \Pr_{m}^{4R;\vec{x}}(\neg\m{B}(\beta)\cap \m{D}) \le \e.
\end{align}
We shall prove \eqref{acceptbot} in \textbf{Step 2}. Assuming it, observe that
\begin{align*}
    \Pr_{m}^{4R;\vec{x}}(\m{A}\cap \m{B}(\beta)\cap\m{C}) & =\Ex_{m}^{4R;\vec{x}}\big[\ind_{\m{B}(\beta)\cap\m{C}}\cdot \Ex_{m}^{4R;\vec{x}}[\ind_{\m{A}}\mid \mathcal{F}]\big] \\ & =\Ex_m^{4R;\vec{x}}\left[\ind_{\m{B}(\beta)\cap\m{C}}\cdot\frac{\Ex_{\m{IRW}}^{R;(a,b)}[V\ind_{\m{A}}]}{\Ex_{\m{IRW}}^{R;(a,b)}\left[V\right]}\right] \le \beta^{-1}\cdot\Ex_{m}^{4R;\vec{x}}\big[\ind_{\m{C}}\cdot\Pr_{\m{IRW}}^{R;(a,b)}(\m{A})\big].
\end{align*}
By Proposition 5.2 in \bcd, we can take $\delta>0$ (depending on $\beta$ along with other parameters) small enough so that $\ind_{\m{C}}\cdot\Pr_{\m{IRW}}^{R;(a,b)}(\m{A})\le \beta\e$. Thus we get $\Pr_{m}^{4R;\vec{x}}(\m{A}\cap \m{B}(\beta)\cap\m{C}) \le \e$, which in view of \eqref{events} and \eqref{acceptbot} implies $\Pr_{m}^{4R;\vec{x}}(\m{A}) \le 4\e$. This completes the proof modulo \eqref{acceptbot}. 

\medskip

\noindent\textbf{Step 2.} To prove \eqref{acceptbot}, using the tower property of conditional expectation we write
\begin{align*}
    \Pr_{m}^{4R;\vec{x}}(\neg\m{B}(\beta)\cap \m{D}) = \Ex_m^{4R;\vec{x}}\left[\ind_{\m{D}}\cdot \Ex_{m}^{4R;\vec{x}}[\ind_{\neg\m{B}(\beta)} \mid \sigma\{L_1,\ldots,L_{2m-2}\}]\right]
\end{align*}
Invoking the Gibbs property again we have
\begin{align*}
    \Ex_{m}^{4R;\vec{x}}[\ind_{\neg\m{B}(\beta)} \mid \sigma\{L_1,\ldots,L_{2m-2}\}] = \Pr_{\Upsilon}^{4R;(a',b';\vec{z},\vec{w})}(\neg\m{B}(\beta)),
\end{align*}
where $\Pr_{\Upsilon}^{4R;(a',b',\vec{z},\vec{w})}$ denotes the $\hslg$ Gibbs measure on the domain $\Upsilon$ defined in \eqref{fre} with the boundary conditions $u_{2m-1,8R-1}=a'=x_{2m-1,1}$, $u_{2m,8R}=b'=x_{2m}$, and $u_{2m-2,2j-1}=z_j=L_{2m-2}(2j-1)$ for $j\in \ll1,R\rr$ and $u_{2m-2,2j-1}=w_{j-R}=L_{2m-2}(2j-1)$ for $j\in \ll R+1,4R\rr$. Note that on the event $\m{D}$, we know $z_j,w_j\ge -M_2\sqrt{T}$. Thus it suffices to provide an upper bound for $\Pr_{\Upsilon}^{4R;(a',b',\vec{z},\vec{w})}(\neg\m{B}(\beta))$ that is uniform over deterministic boundary conditions $\vec{x}\in I_M$ and $\vec{z},\vec{w}$ satisfying $z_j,w_j \ge -M_2\sqrt{T}$.

To do this, we use {the size biasing trick} to provide a lower bound. This trick is quite standard in line ensemble calculations (see e.g. Section 4.3 in \cite{bcd2} or Section 5 in \cite{bcd}). Essentially, the size biasing argument proceeds by writing, for any event $\m{E}\in \mathcal{F}$,
\begin{align}\label{sizebias}
    \Pr_{\Upsilon}^{4R;(a',b',\vec{z},\vec{w})}(\m{E}) = \frac{\Ex_{\Upsilon}^{4R;(a',b',(\infty)^R,\vec{w})}\big[\ind_{\m{E}} \exp(-e^{L_{2m-1}(2R)-z_{R}}) \Ex_{\m{IRW}}^{R;(a',b')}[V(\vec{z})]\big]}{\Ex_{\Upsilon}^{4R;(a',b',(\infty)^R,\vec{w})}\big[\exp(-e^{L_{2m-1}(2R)-z_{R}}) \Ex_{\m{IRW}}^{R;(a,b)}[V(\vec{z})]\big]}. 
\end{align}
The above formula follows from the definitions of each of the measures involved (see eq.~(5.25) in \bcd~for a similar formula).   By Lemma \ref{lem:IRWupperbd} and stochastic monotonicity (Proposition \ref{p:gmc}), we can choose $M_3$ large enough so that
\begin{align*}
    \Pr_{\m{IRW}}^{R;(a,b)}\bigg(\sup_{j\in \ll1,2R-1\rr} L_{2m-1}(j), \sup_{j\in \ll1,2R\rr} L_{2m}(j) \le M_3\sqrt{T}\bigg) \ge 1-\e
\end{align*}
for all $a,b\le 0$ (recall that we are using $L_{2m-1}(\cdot), L_{2m}(\cdot)$ to denote the underlying random variables in the $\Pr_{\m{IRW}}^{R;(a,b)}$ law). Using translation invariance (Lemma \ref{traninv}) and the definition of $V(\vec{z})$, this implies for $a,b \le -(M_2+M_3+1)\sqrt{T}$ and $z_j \ge -M_2\sqrt{T}$
\begin{align}\label{vlbd}
    \Ex_{\m{IRW}}^{R;(a,b)}[V(\vec{z})] & \ge \exp\big(-2(R-1)e^{-\sqrt{T}}\big)(1-\e) \ge \tfrac12 
\end{align}
for all large enough $T$. Let us now consider the event
\begin{align*}
    \m{G}:=\{ L_{2m-1}(2R), L_{2m-1}(2R-1), L_{2m}(2R) \le -(M_2+M_3+1)\sqrt{T}\}.
\end{align*}
Using \eqref{vlbd} we get
\begin{align*}
   & \Ex_{\Upsilon}^{4R;(a',b',(\infty)^R,\vec{w})}\big[\exp(-e^{L_{2m-1}(2R)-z_{R}}) \Ex_{\m{IRW}}^{R;(a,b)}[V(\vec{z})]\big] \\ & \ge \Ex_{\Upsilon}^{4R;(a',b',(\infty)^R,\vec{w})}\big[\ind_{\m{G}}\exp(-e^{L_{2m-1}(2R)-z_{R}}) \Ex_{\m{IRW}}^{R;(a,b)}[V(\vec{z})]\big] \ge \tfrac12\cdot \exp(-e^{-\sqrt{T}}) \cdot \Pr_{\Upsilon}^{4R;(a',b',(\infty)^R,\vec{w})}(\m{G}).
\end{align*}
By stochastic monotonicity, the probability of the event $\m{G}$ is decreasing as we increase the boundary data. As $\vec{x} \in I_M$, we can choose $M_4$ such that $x_{2m-1,1},x_{2m}\le M_4\sqrt{4R}$. Thus taking $w_j \to \infty$, $x_{2m-1,1} \to (M_4+1)\sqrt{4R}$, and $x_{2m} \to M_4\sqrt{4R}$ we get 
\begin{align*}
    \Pr_{\Upsilon}^{4R;(a',b',(\infty)^R,\vec{w})}(\m{G}) & \ge \Pr_{\Upsilon}^{4R;(a',b',(\infty)^{4R})}(\m{G}) =\Pr_{\m{IRW}}^{4R;x_{2m-1,1},x_{2m}}(\m{G}) \ge \Pr_{\m{IRW}}^{4R;((M_4+1)\sqrt{4R}, M_4\sqrt{4R})}(\m{G}).
\end{align*}
By translation invariance, the above probability is equal to
\begin{align*}
    \Pr_{\m{IRW}}^{4R;(0,-\sqrt{4R})}\big(L_{2m-1}(2R), L_{2m-1}(2R-1), L_{2m}(2R) \le -(M_2+M_3+1)\sqrt{T}-(M_4+1)\sqrt{4R}\big).
\end{align*}
By Lemma \ref{l:region}, this has a uniform (in $R,T$) lower bound by a positive constant. We thus see that for large enough $T$, the denominator of the r.h.s.~ of \eqref{sizebias} has a uniform lower bound by some constant $\phi>0$ for all $z_j,w_j \ge -M_2\sqrt{T}$. With $\m{E}=\neg \m{B}(\beta)$, from the definition of the event $\m{B}(\beta)$ we thus have
\begin{align*}
    \eqref{sizebias} \le \phi^{-1}\cdot \Ex_{\Upsilon}^{4R;(a',b',(\infty)^R,\vec{w})}\big[\ind_{\neg\m{B}(\beta)} \Ex_{\m{IRW}}^{R;(a',b')}[V(\vec{z})]\big] \le \phi^{-1} \beta.
\end{align*}
Taking $\beta$ small enough, we can make the above bound arbitrarily small. This completes the proof.
\end{proof}

 
\begin{lemma}\label{highmt} Fix $m,M_1,M_2>0$. There exists $\phi=\phi(m,M_1,M_2)>0$ and $T_0=T_0(m,M_1,M_2)>0$ such that for all $\vec{x}\in I_{M_1}$ and $T\ge T_0$ we have 
    \begin{align*}      \Pr_{m}^{T;\vec{x}}\bigg(\bigcap_{i=1}^{2m}\bigcap_{j\in \ll1,2T/4^m\rr} \big\{L_i(j) \ge M_2\sqrt{T}\big\}\bigg) \ge \phi.
    \end{align*}
\end{lemma}
\begin{proof} We proceed via induction. For $m=1$, this is Proposition 4.1 in \bcd. Let us assume it holds for $m-1$. We shall prove it holds for $m$. To avoid working with floor functions, we will assume $T$ is a multiple of $4^{m+1}$. Set $R=T/4^{m}$. We define several events to be used in the proof:
\begin{align*}
    & \m{A}:=\bigcap_{i=1}^{2m-2}\bigcap_{j\in \ll1,8R\rr} \big\{L_i(j) \ge 2M_2\sqrt{T}\big\}, \quad \m{B}(M_3):=\big\{L_{2m-1}(8R-1), L_{2m}(8R) \ge -M_3\sqrt{T}\big\}, \\
    & \m{C}:=\bigcap_{i=2m-1}^{2m}\bigcap_{j\in \ll1,2R\rr} \big\{L_i(j) \ge M_2\sqrt{T}\big\}, \quad \m{D}:=\bigcap_{i=2m-1}^{2m}\bigcap_{j\in \ll1,8R-\ind_{i\;\m{odd}}\rr} \big\{L_i(j) \le (M_2+1)\sqrt{T}\big\}.
\end{align*}
Note that by stochastic monotonicity and the inductive hypothesis, there exists $\phi_1>0$ such that $\Pr_{m}^{T;\vec{x}}(\m{A}) \ge \Pr_{m-1}^{T;\vec{x}'}(\m{A}) \ge  \phi_1,$ where $\vec{x}'$ is obtained from $\vec{x}$ by removing $x_{2m-1,1},x_{2m-1,2},x_{2m}$ from the list. From Proposition \ref{p.2nto},  we can get $M_3=M_3(\phi)>0$ large enough so that $\Pr_{m}^{T;\vec{x}}(\m{B}(M_3)) \ge 1-\phi_1/2.$ Let us fix this $M_3$ and write $\m{B}=\m{B}(M_3)$. We thus have $\Pr_{m}^{T;\vec{x}}(\m{A}\cap \m{B}) \ge \phi_1/2$. Consider the $\sigma$-field
\begin{align*}
    \mathcal{F}:=\big\{L_{2m-1}\ll 8R-1,2T-1\rr,L_{2m}\ll 8R,2T\rr,  L_i\ll 1, 2T-\ind_{i\;\m{odd}}\rr : i\in \ll1,2m-2\rr \big\}.
\end{align*}
We have $\Pr(\m{A}\cap \m{B}\cap \m{C})= \Ex[\ind_{\m{A}\cap \m{B}}\Ex[\ind_{\m{C}}\mid \mathcal{F}]]$. The conditional measure is given by $\Pr_{\Upsilon}^{4R;(a,b,\vec{z})}$ defined in \eqref{def:Glaw2} where $a=L_{2m-1}(8R-1)$, $b=L_{2m}(8R)$, and $z_j=L_{2m-2}(2j-1)$ for $j\in \ll1,4R\rr$. By stochastic monotonicity, reducing the boundary data will decrease the probability of $\m{C}$. Thus we get
\begin{align*}
    \ind_{\m{A}\cap \m{B}}\Ex[\ind_{\m{C}}\mid \mathcal{F}] = \ind_{\m{A}\cap \m{B}} \cdot \Pr_{\Upsilon}^{4R;(a,b,\vec{z})}(\m{C}) \ge \ind_{\m{A}\cap \m{B}} \cdot \Pr_{\Upsilon}^{4R;(a',b',\vec{z}\,')}(\m{C}) \ge  \ind_{\m{A}\cap \m{B}}\Pr_{\Upsilon}^{4R;(a',b',\vec{z}\,')}(\m{C}\cap \m{D}) 
\end{align*}
where $a'=-M_3\sqrt{T}, b'=-M_3\sqrt{T}-\sqrt{4R}$, $z_j'=2M_2\sqrt{T}$. Recall $V(\vec{c})$ from \eqref{def:Glaw2}. Note that under the event $\m{D}$ and the boundary conditions $(a',b',\vec{z}\,')$, we always have $V(\vec{z}) \ge \exp\big(-8Re^{-\sqrt{T}}\big) \ge \frac12$ for large enough $T$. Thus utilizing the relation in \eqref{def:Glaw2} we get 
\begin{align*}
    \Pr_{\Upsilon}^{4R;(a',b',\vec{z}')}(\m{C}\cap \m{D}) = \frac{\Ex_{\m{IRW}}^{4R;(a',b')}[V(\vec{z})\ind_{\m{C}\cap \m{D}}]}{\Ex_{\m{IRW}}^{4R;(a',b')}[V(\vec{z})]} \ge \tfrac12\cdot\Pr_{\m{IRW}}^{4R;(a',b')}(\m{C}\cap \m{D}) \ge \phi_3
\end{align*}
for some $\phi_3>0$. Here the last inequality follows by an application of Lemma \ref{l:region}. Plugging this estimate back into the above we get $\Pr(\m{A}\cap \m{B} \cap \m{C}) \ge \phi_1\phi_3/2$. 
\end{proof}

We now have all the necessary ingredients to complete the proof of Theorem \ref{p.tight3}.

\begin{proof}[Proof of Theorem \ref{p.tight3}] Fix any $m\in \mathbb{Z}_{\ge 1}$ and $V>0$. Let 
\begin{align*}
    \m{A}_1 & :=\bigcap_{k=1}^{2m} \{|\L_{2k-1}^N(1)|+|\L_{2k}^N(2)|\le VN^{1/3}\}, \qquad
    \m{A}_2  :=\big\{\omega_\delta^{N^{2/3}}(\L_{2m-1}^N,\L_{2m}^N; \ll 1, AN^{2/3} \rr) \ge \gamma N^{1/3}\big\}.
\end{align*}
Set $\m{A}:=\m{A}_1\cap \m{A}_2$. Due to endpoint tightness (Theorem \ref{thm:eptight}), it suffices to show that for each $V>0$, $\Pr(\m{A})$ can be made arbitrarily small by taking $\delta$ small enough. 

Set $T=AN^{2/3}/4^{m+1}$. To avoid working with floor functions, we will assume $T\in \mathbb{Z}_{\ge 1}$. Throughout the proof we will assume $N$ (and hence $T$) is large enough. For each $R\in \Z_{\ge 1}$, consider the events
\begin{align*}
    \m{C}_{R} & :=\bigcap_{i=1}^m\Big\{|\L_{2i-1}^N(2R-1)|,|\L_{2i-1}^N(2R)|, |\L_{2i}^N(2R)| \le M\sqrt{T}, \\ & \hspace{1cm} \min\{\L_{2i-1}^N(2R-1), \L_{2i-1}^N(2R)\} - \L_{2i}^N(2R) \ge -(\log T)^2, \\ & \hspace{2cm}\L_{2i}^N(2R)- \max \{\L_{2i+1}^N(2R-1), \L_{2i+1}^N(2R)\} \ge -(\log T)^2\Big\}, \\
    \m{D}' & := \Big\{\sup_{j\in \ll1,4^{m+1}T\rr} \L_{2m+1}^N(j) \le M_2\sqrt{T}\Big\}
\end{align*}
We write $\m{C}:=\m{C}_T$. Set $T':=4^{m+1}T$ and $\m{D}:=\m{D}'\cap \m{C}_{T'}$. Thanks to Propositions \ref{p.nottoo} and \ref{pp:order}, we may choose $M, M_2$ large enough so that
\begin{align}\label{events2}
    \Pr(\neg \m{D}) \le \e. 
\end{align}
For each $R\in \Z_{\ge 1}$ define
\begin{align*}
    \mathcal{F}_R:=\sigma\big\{\L_{i_1}^N(j), \L_{i_2}^N(\cdot) : i_1\in \ll1,2m\rr, j\ge 2R-\ind_{i\;\m{odd}}, i_2 \ge 2m+1 \big\}.
\end{align*}
The proof strategy roughly follows that of Proposition \ref{pimc}. Invoking the Gibbs property (Theorem \ref{thm:conn}),  the random variable $\Ex[\m{A}\mid \mathcal{F}_T]$ has the following representation:
\begin{align*}
    \Ex[\m{A}\mid \mathcal{F}_T] = \frac{\Ex_{m}^{T;\vec{x}}[U(\vec{z})\ind_{\m{A}}]}{\Ex_{m}^{T;\vec{x}}\left[U(\vec{z})\right]},
\end{align*}
where $x_{2i-1,1}=\L_{2i-1}^N(2T-1)$, $x_{2i-1,2}=\L_{2i-1}^N(2T)$, $x_{2i}=\L_{2i}^N(2T)$, and
\begin{align*}
    U(\vec{z}):=\exp\bigg(-e^{z_T-L_{2m}(2T-1)}-\sum_{j=1}^{T-1}(e^{z_j-L_{2m}(2j-1)}+e^{z_j-L_{2m}(2j+1)})\bigg),
\end{align*}
with $z_j:=\L_{2m+1}^N(2j)$ for $j\in \ll1,T\rr$. Let us set $\m{B}(\beta):=\{\Ex_{m}^{T;\vec{x}}\left[U(\vec{z})\right] \ge \beta\}.$  We claim that
\begin{align}\label{acceptbot2}
    \Pr(\neg\m{B}(\beta)\cap \m{D})\le \e.
\end{align}
We shall prove \eqref{acceptbot} in \textbf{Step 2}. Assuming it, observe that
\begin{align*}
    \Pr(\m{A}\cap \m{B}(\beta)\cap\m{C}) & =\Ex\big[\ind_{\m{B}(\beta)\cap\m{C}}\cdot \Ex[\ind_{\m{A}}\mid \mathcal{F}_{T}]\big]  =\Ex\left[\ind_{\m{B}(\beta)\cap\m{C}}\cdot\frac{\Ex_{m}^{T;\vec{x}}[U\ind_{\m{A}}]}{\Ex_{m}^{T;\vec{x}}\left[U\right]}\right] \le \beta^{-1}\cdot\Ex\big[\ind_{\m{C}}\cdot\Pr_{m}^{T;\vec{x}}(\m{A})\big].
\end{align*}
By Proposition \ref{pimc}, we can take $\delta>0$ (depending on $\beta$ along with other parameters) small enough so that $\ind_{\m{C}}\cdot\Pr_{m}^{T;\vec{x}}(\m{A})\le \beta\e$. Thus we get $\Pr(\m{A}\cap \m{B}(\beta)\cap\m{C}) \le \e$, which in view of \eqref{events2} and \eqref{acceptbot2} implies $\Pr(\m{A}) \le 4\e$. This completes the proof modulo \eqref{acceptbot2}. 

\medskip

\noindent\textbf{Step 2.} To prove \eqref{acceptbot2}, we use the tower property of conditional expectation to write
\begin{align*}
    \Pr(\neg\m{B}(\beta)\cap \m{D}) = \Ex\left[\ind_{\m{D}}\cdot \Ex[\ind_{\neg\m{B}(\beta)} \mid \mathcal{F}_{T'}]\right].
\end{align*}
By Theorem \ref{thm:conn}, the conditional measure can be viewed as a $\hslg$ Gibbs measure on the domain $\Lambda_{m,T'}$ (defined in \eqref{lambdamt}) with boundary conditions
\begin{align*}
    & u_{2i-1,j}=y_{2i-1,j}=\L_{2i-1}^N(2T'-1)  \mbox{ for } i\in \ll1,m\rr, j\in \ll 2T-1,2T\rr, \\
    & u_{2i}=y_{2i}=\L_{2i}^N(2T') \mbox{ for } i\in \ll1,m\rr,\\
    & u_{2m+1,2j}=z_j=\L_{2m+1}^N(2j) \mbox{ for }j\in \ll1,T\rr,  \\ 
    & u_{2m+1,2j}=w_{j-T}=\L_{2m+1}^N(2j)  \mbox{ for }j\in \ll T+1,T'\rr. 
\end{align*}
We denote this Gibbs measure as $\Pr_{m}^{T';\vec{y},\vec{z},\vec{w}}$. On the event $\m{D}$, we have  
\begin{equation}
    \label{bdata}
    \begin{aligned}
    & |u_{i,j}| \le M\sqrt{T} \mbox{ for all } (i,j) \in \partial \Lambda_{m,T} \mbox{ with }i\le 2m, \\
    & u_{i,j_1}-u_{i+1,j_2} \ge -(\log T)^{7/6}     \mbox{ for all } (i,j) \in \partial \Lambda_{m,T},   \\
    & z_j \le M_2\sqrt{T} \mbox{ for } j \in \ll1,T\rr, \qquad w_j \le M_2\sqrt{T} \mbox{ for } j\in \ll1,T-T'\rr.
\end{aligned}
\end{equation}
It therefore suffices to provide an upper bound for $\Pr_{m}^{T';\vec{y},\vec{z},\vec{w}}(\neg \m{B}(\beta))$ that is uniform over all deterministic boundary data satisfying \eqref{bdata}. A size biasing argument leads to 
\begin{align*}
    \Pr_{m}^{T';\vec{y},\vec{z},\vec{w}}(\neg\m{B}(\beta)) & = \frac{\Ex_{m}^{T';\vec{y},(-\infty)^T, \vec{w}}\big[\ind_{\neg\m{B}(\beta)} \exp(-e^{z_T-L_{2m}(2T+1)}) \Ex_{m}^{T;\vec{x}}[U(\vec{z})]\big]}{\Ex_{m}^{T';\vec{y},(-\infty)^T,\vec{w}}\big[\exp(-e^{z_T-L_{2m}(2T+1)}) \Ex_{m}^{T;\vec{x}}[U(\vec{z})]\big]} \\ & \le \beta\cdot \left(\Ex_{m}^{T';\vec{y},(-\infty)^T,\vec{w}}\big[\exp(-e^{z_T-L_{2m}(2T+1)}) \Ex_{m}^{T;\vec{x}}[U(\vec{z})]\big]\right)^{-1}. 
\end{align*}
Thus it suffices to show that the last expectation above has a uniform lower bound. Using Proposition \ref{p.2nto} and stochastic monotonicity,   we obtain $M_3>1$ such that $\Pr_{m}^{T;\vec{x}}(L_{2m}(2j-1) \ge -M_3\sqrt{T}) \ge 1-\e$ for all $\vec{x}$ of type \eqref{newvec} with entries all larger than $-\sqrt{T}$. By translation invariance this implies
\begin{align*}
    \Pr_{m}^{T;\vec{x}}(L_{2m}(2j-1) \ge (M_2+1)\sqrt{T}) \ge 1-\e
\end{align*}
for all $\vec{x}$ of type \eqref{newvec} with entries all larger than $(M_3+M_2)\sqrt{T}$. This forces
\begin{align*}
    \Ex_{m}^{T;\vec{x}}[U(\vec{z})] & \ge \exp\big(-2(T-1)e^{-\sqrt{T}}\big)(1-\e) \ge \tfrac12.  
\end{align*}
for all $\vec{x}$ of type \eqref{newvec} with entries all larger than $(M_3+M_2)\sqrt{T}$. Let us now consider the event $$\m{G}=\bigcap_{i=1}^{2m}\bigcap_{j\in \ll1,2T'/4^m\rr} \{L_{i}(j) \ge (M_2+M_3)\sqrt{T}\}.$$ 
By stochastic monotonicity and Lemma \ref{highmt},
\begin{align*}
    \Pr_{m}^{T';\vec{y},(-\infty)^T,\vec{w}}(\m{G}) & \ge \Pr_{m}^{T';\vec{y},(-\infty)^T,(-\infty)^T}(\m{G})  =\Pr_{m}^{T';\vec{y}}(\m{G}) \ge \phi
\end{align*}
for some $\phi>0$. Combining the above, we find
\begin{align*}
   & \Ex_{m}^{T';\vec{y},(-\infty)^T,\vec{w}}\big[\exp(-e^{z_T-L_{2m}(2T+1)}) \Ex_{m}^{T;\vec{x}}[U(\vec{z})]\big] \\ &\qquad \ge \Ex_{m}^{T';\vec{y},(-\infty)^T,\vec{w}}\big[\ind_{\m{G}}\exp(-e^{z_T-L_{2m}(2T+1)}) \Ex_{m}^{T;\vec{x}}[U(\vec{z})]\big] \\ &\qquad \ge \tfrac12\cdot \exp(-e^{-\sqrt{T}}) \cdot \Pr_{m}^{T';\vec{y},(-\infty)^T,\vec{w}}(\m{G}) \ge \tfrac12 \cdot \exp(-e^{-\sqrt{T}}) \cdot \phi >0,
\end{align*}
and the proof is complete.
\end{proof}

	\bibliographystyle{alpha}		
	\bibliography{half3}

\begin{thebibliography}{GRASY15}

\bibitem[Abr80]{phy2}
Douglas~B. Abraham.
\newblock Solvable model with a roughening transition for a planar {I}sing
  ferromagnet.
\newblock {\em Physical Review Letters}, 44(18):1165, 1980.

\bibitem[BBC20]{bbc20}
Guillaume Barraquand, Alexei Borodin, and Ivan Corwin.
\newblock Half-space {M}acdonald processes.
\newblock In {\em Forum of Mathematics, Pi}, volume~8. Cambridge University
  Press, 2020.

\bibitem[BBCS18]{bbcs}
Jinho Baik, Guillaume Barraquand, Ivan Corwin, and Toufic Suidan.
\newblock {P}faffian {S}chur processes and last passage percolation in a
  half-quadrant.
\newblock {\em The Annals of Probability}, 46(6):3015--3089, 2018.

\bibitem[BBNV18]{bete}
Dan Betea, J{\'e}r{\'e}mie Bouttier, Peter Nejjar, and Mirjana Vuleti{\'c}.
\newblock The free boundary {S}chur process and applications {I}.
\newblock In {\em Annales Henri Poincar{\'e}}, volume~19, pages 3663--3742.
  Springer, 2018.

\bibitem[BC23]{bc22}
Guillaume Barraquand and Ivan Corwin.
\newblock {Stationary measures for the log-gamma polymer and KPZ equation in
  half-space}.
\newblock {\em The Annals of Probability}, 51(5):1830--1869, 2023.

\bibitem[BCD21]{bcd2}
Guillaume Barraquand, Ivan Corwin, and Evgeni Dimitrov.
\newblock Fluctuations of the log-gamma polymer free energy with general
  parameters and slopes.
\newblock {\em Probability Theory and Related Fields}, 181(1):113--195, 2021.

\bibitem[BCD23]{bcd}
Guillaume Barraquand, Ivan Corwin, and Evgeni Dimitrov.
\newblock Spatial tightness at the edge of {G}ibbsian line ensembles.
\newblock {\em Communications in Mathematical Physics}, pages 1--78, 2023.

\bibitem[BCD24]{half1}
Guillaume Barraquand, Ivan Corwin, and Sayan Das.
\newblock {KPZ} exponents for the half-space log-gamma polymer.
\newblock {\em Probability Theory and Related Fields}, pages 1--131, 2024.

\bibitem[BCY24]{bcy}
Guillaume Barraquand, Ivan Corwin, and Zongrui Yang.
\newblock Stationary measures for integrable polymers on a strip.
\newblock {\em Inventiones mathematicae}, 237(3):1567--1641, 2024.

\bibitem[BFO20]{ale1}
Dan Betea, Patrik~L. Ferrari, and Alessandra Occelli.
\newblock Stationary half-space last passage percolation.
\newblock {\em Communications in Mathematical Physics}, 377(1):421--467, 2020.

\bibitem[BHL83]{phy3}
E.~Br{\'e}zin, B.I. Halperin, and S.~Leibler.
\newblock Critical wetting in three dimensions.
\newblock {\em Physical Review Letters}, 50(18):1387, 1983.

\bibitem[BK21]{bk}
Wlodek Bryc and Alexey Kuznetsov.
\newblock {Markov limits of steady states of the KPZ equation on an interval}.
\newblock {\em arXiv preprint arXiv:2109.04462}, 2021.

\bibitem[BKLD20]{bkld2}
Guillaume Barraquand, Alexandre Krajenbrink, and Pierre Le~Doussal.
\newblock Half-space stationary {K}ardar-{P}arisi-{Z}hang equation.
\newblock {\em Journal of Statistical Physics}, 181(4):1149--1203, 2020.

\bibitem[BKLD22]{bkld}
Guillaume Barraquand, Alexandre Krajenbrink, and Pierre Le~Doussal.
\newblock Half-space stationary {K}ardar-{P}arisi-{Z}hang equation beyond the
  {B}rownian case.
\newblock {\em Journal of Physics A: Mathematical and Theoretical}, 2022.

\bibitem[BLD22]{bd22}
Guillaume Barraquand and Pierre Le~Doussal.
\newblock Steady state of the {KPZ} equation on an interval and {L}iouville
  quantum mechanics.
\newblock {\em Europhysics Letters}, 137(6), 2022.

\bibitem[BOZ21]{boz}
Elia Bisi, Neil O’Connell, and Nikos Zygouras.
\newblock The geometric {B}urge correspondence and the partition function of
  polymer replicas.
\newblock {\em Selecta Mathematica}, 27:1--39, 2021.

\bibitem[BR01a]{br1}
Jinho Baik and Eric~M. Rains.
\newblock Algebraic aspects of increasing subsequences.
\newblock {\em Duke Mathematical Journal}, 109(1):1--65, 2001.

\bibitem[BR01b]{br01}
Jinho Baik and Eric~M. Rains.
\newblock The asymptotics of monotone subsequences of involutions.
\newblock {\em Duke Mathematical Journal}, 109(2):205--281, 2001.

\bibitem[BR01c]{br3}
Jinho Baik and Eric~M. Rains.
\newblock Symmetrized random permutations.
\newblock {\em Random {M}atrix {M}odels and their {A}pplications}, 40:1--19,
  2001.

\bibitem[BW22]{bw}
Guillaume Barraquand and Shouda Wang.
\newblock An identity in distribution between full-space and half-space
  log-gamma polymers.
\newblock {\em International Mathematics Research Notices}, rnac132, 2022.

\bibitem[BZ19]{bz19}
Elia Bisi and Nikos Zygouras.
\newblock Point-to-line polymers and orthogonal {W}hittaker functions.
\newblock {\em Transactions of the American Mathematical Society},
  371(12):8339--8379, 2019.

\bibitem[BZ22]{bz22}
Elia Bisi and Nikos Zygouras.
\newblock Transition between characters of classical groups, decomposition of
  {G}elfand-{T}setlin patterns and last passage percolation.
\newblock {\em Advances in Mathematics}, 404:108453, 2022.

\bibitem[CH14]{CH14}
Ivan Corwin and Alan Hammond.
\newblock {B}rownian {G}ibbs property for {A}iry line ensembles.
\newblock {\em Inventiones mathematicae}, 195(2):441--508, 2014.

\bibitem[CK21]{ck}
Ivan Corwin and Alisa Knizel.
\newblock {Stationary measure for the open KPZ equation}.
\newblock {\em arXiv preprint arXiv:2103.12253}, 2021.

\bibitem[CN16]{cn16}
Francis Comets and Vu-Lan Nguyen.
\newblock Localization in log-gamma polymers with boundaries.
\newblock {\em Probability Theory and Related Fields}, 166:429--461, 2016.

\bibitem[Cor22]{cor22}
Ivan Corwin.
\newblock {Some recent progress on the stationary measure for the open KPZ
  equation}.
\newblock {\em Toeplitz Operators and Random Matrices: In Memory of Harold
  Widom}, pages 321--360, 2022.

\bibitem[COSZ14]{cosz14}
Ivan Corwin, Neil O’Connell, Timo Sepp{\"a}l{\"a}inen, and Nikolaos Zygouras.
\newblock Tropical combinatorics and {W}hittaker functions.
\newblock {\em Duke Mathematical Journal}, 163(3):513--563, 2014.

\bibitem[DIM77]{DurrBM}
Richard Durrett, Donald Iglehart, and Douglas Miller.
\newblock Weak convergence to {B}rownian meander and {B}rownian excursion.
\newblock {\em Ann. Probab.}, 5(1):117--129, 1977.

\bibitem[DZ24]{dz23}
Sayan Das and Weitao Zhu.
\newblock The half-space log-gamma polymer in the bound phase.
\newblock {\em Communications in Mathematical Physics}, 405(8):184, 2024.

\bibitem[Fel71]{feller}
William Feller.
\newblock {\em An {I}ntroduction to {P}robability {T}heory and its
  {A}pplications}, volume~2.
\newblock John Wiley \& Sons, 1971.

\bibitem[Gin23]{victor}
Victor Ginsburg.
\newblock Pinning, diffusive fluctuations, and {G}aussian limits for half-space
  directed polymer models.
\newblock {\em arXiv preprint arXiv:2312.11439}, 2023.

\bibitem[GRASY15]{grasy15}
Nicos Georgiou, Firas Rassoul-Agha, Timo Sepp{\"a}l{\"a}inen, and Atilla
  Yilmaz.
\newblock {Ratios of partition functions for the log-gamma polymer}.
\newblock {\em The Annals of Probability}, 43(5):2282 -- 2331, 2015.

\bibitem[Igl74]{igl}
Donald Iglehart.
\newblock Functional central limit theorems for random walks conditioned to
  stay positive.
\newblock {\em The Annals of Probability}, 2(4):608--619, 1974.

\bibitem[IMS22]{ims22}
Takashi Imamura, Matteo Mucciconi, and Tomohiro Sasamoto.
\newblock Solvable models in the {KPZ} class: approach through periodic and
  free boundary {S}chur measures.
\newblock {\em arXiv preprint arXiv:2204.08420}, 2022.

\bibitem[IT18]{it}
Yasufumi Ito and Kazumasa~A. Takeuchi.
\newblock {When fast and slow interfaces grow together: connection to the
  half-space problem of the Kardar-Parisi-Zhang class}.
\newblock {\em Physical Review E}, 97(4):040103, 2018.

\bibitem[JRA20]{jra20}
Christopher Janjigian and Firas Rassoul-Agha.
\newblock Uniqueness and ergodicity of stationary directed polymers on
  $\mathbb{Z}^2$.
\newblock {\em Journal of Statistical Physics}, 179(3):672--689, 2020.

\bibitem[Kar85]{kar2}
Mehran Kardar.
\newblock Depinning by quenched randomness.
\newblock {\em Physical review letters}, 55(21):2235, 1985.

\bibitem[K{\"o}n05]{konig}
Wolfang K{\"o}nig.
\newblock Orthogonal polynomial ensembles in probability theory.
\newblock {\em Probability Surveys}, 2:385--447, 2005.

\bibitem[KPZ86]{kpz86}
Mehran Kardar, Giorgio Parisi, and Yi-Cheng Zhang.
\newblock Dynamic scaling of growing interfaces.
\newblock {\em Physical Review Letters}, 56(9):889, 1986.

\bibitem[NZ17]{nz17}
Vu-Lan Nguyen and Nikos Zygouras.
\newblock Variants of geometric {RSK}, geometric {PNG}, and the multipoint
  distribution of the log-gamma polymer.
\newblock {\em International Mathematics Research Notices},
  2017(15):4732--4795, 2017.

\bibitem[OSZ14]{osz14}
Neil O’Connell, Timo Sepp{\"a}l{\"a}inen, and Nikos Zygouras.
\newblock Geometric {RSK} correspondence, {W}hittaker functions and symmetrized
  random polymers.
\newblock {\em Inventiones mathematicae}, 197(2):361--416, 2014.

\bibitem[PSW82]{phy1}
Rahul Pandit, M.~Schick, and Michael Wortis.
\newblock Systematics of multilayer adsorption phenomena on attractive
  substrates.
\newblock {\em Physical Review B}, 26(9):5112, 1982.

\bibitem[Rai00]{rai00}
Eric~M. Rains.
\newblock Correlation functions for symmetrized increasing subsequences.
\newblock {\em arXiv preprint arXiv:math/0006097}, 2000.

\bibitem[Sep12]{timo}
Timo Sepp{\"a}l{\"a}inen.
\newblock Scaling for a one-dimensional directed polymer with boundary
  conditions.
\newblock {\em The Annals of Probability}, 40(1):19--73, 2012.

\bibitem[Ser23]{serio}
Christian Serio.
\newblock Tightness of discrete {G}ibbsian line ensembles.
\newblock {\em Stochastic Processes and their Applications}, 159:225--285,
  2023.

\bibitem[SI04]{sis}
Tomohiro Sasamoto and Takashi Imamura.
\newblock Fluctuations of the one-dimensional polynuclear growth model in
  half-space.
\newblock {\em Journal of Statistical Physics}, 115(3):749--803, 2004.

\bibitem[Spi60]{spit}
Frank Spitzer.
\newblock A {T}auberian theorem and its probability interpretation.
\newblock {\em Transactions of the American Mathematical Society},
  94(1):150--169, 1960.

\end{thebibliography}

\end{document}